\documentclass[12pt]{amsart}
\usepackage[margin=1in]{geometry}
\usepackage{graphicx, enumitem, amssymb, hyperref, tikz-cd}
\DeclareEmphSequence{\bfseries\itshape}
\newtheorem{theorem}{Theorem}[section]
\newtheorem{cor}[theorem]{Corollary}
\newtheorem{lem}[theorem]{Lemma}
\newtheorem{prop}[theorem]{Proposition}
\newtheorem{observation}[theorem]{Observation}
\newtheorem{claim}[theorem]{Claim}
\theoremstyle{definition}
\newtheorem{definition}[theorem]{Definition}
\newtheorem{remark}[theorem]{Remark}
\newtheorem{notation}[theorem]{Notation}
\newtheorem{example}[theorem]{Example}
\newtheorem{conceptremark}[theorem]{Conceptual Remark}
\newtheorem{question}[theorem]{Question}
\newtheorem{application}[theorem]{Application}
\newtheorem*{example*}{Example}
\title[Tropical coordinates for $\mathrm{SL}_3$-webs: naturality]{Tropical Fock--Goncharov coordinates for $\mathrm{SL}_3$-webs on surfaces II: naturality}
\author[D. C. Douglas]{Daniel C. Douglas}
\author[Z. Sun]{Zhe Sun}
\address{Daniel C. Douglas\\Virginia Tech\\Department of Mathematics\\225 Stanger Street\\Blacksburg\\VA 24061 (USA)}
\email{dcdouglas@vt.edu}
\address{Zhe Sun\\University of Science and Technology of China\\School of Mathematical Sciences\\96 Jinzhai Road\\Hefei 230026 (China)}
\email{sunz@ustc.edu.cn}
\thanks{This work was partially supported by the U.S. National Science Foundation grants DMS-1107452, 1107263, 1107367 ``RNMS: GEometric structures And Representation varieties'' (the GEAR Network).  The first author was also partially supported by the U.S. National Science Foundation grants DMS-1406559 and 1711297, and the second author by the China Postdoctoral Science Foundation grant 2018T110084, the FNR AFR Bilateral grant COALAS 11802479-2, and the Huawei Young Talents Program at IHES}
\begin{document}
\begin{abstract}
In a companion article, we constructed nonnegative integer coordinates $\Phi_\mathcal{T}(\mathcal{W}_{3, \widehat{S}}) \subset \mathbb{Z}_{\geq 0}^N$ for the collection $\mathcal{W}_{3, \widehat{S}}$ of reduced $\mathrm{SL}_3$-webs on a finite-type punctured surface $\widehat{S}$, depending on an ideal triangulation $\mathcal{T}$ of $\widehat{S}$.  We show that these coordinates are natural with respect to the choice of triangulation, in the sense that if a different triangulation $\mathcal{T}^\prime$ is chosen, then the coordinate change map relating $\Phi_\mathcal{T}(\mathcal{W}_{3, \widehat{S}})$ to $\Phi_{\mathcal{T}^\prime}(\mathcal{W}_{3, \widehat{S}})$ is a tropical $\mathcal{A}$-coordinate cluster transformation.  We can therefore view the webs $\mathcal{W}_{3, \widehat{S}}$ as a concrete topological model for the Fock--Goncharov--Shen positive integer tropical points $\mathcal{A}_{\mathrm{PGL}_3, \widehat{S}}^+(\mathbb{Z}^t)$.
\end{abstract}
\maketitle
\begin{figure}[h]
\includegraphics[width=.75\textwidth]{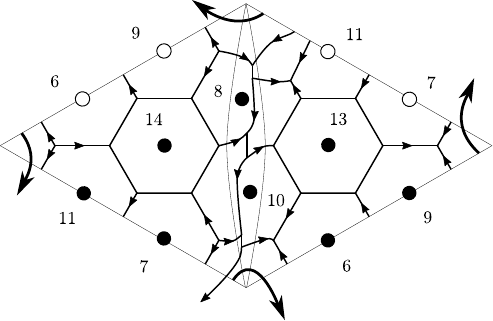}
\caption{    Positive tropical integer $\mathcal{A}$-coordinates for a reduced $\mathrm{SL}_3$-web on the once punctured torus, with respect to an ideal triangulation $\mathcal{T}$.}
\label{fig:coordinates-example-alt}
\end{figure}
	
For a finitely generated group $\Gamma$ and  a suitable Lie group $G$, a primary object of study in higher Teichm\"uller theory  \cite{wienhard2018invitation}  is the $G$-character variety 
\begin{equation*}
\mathcal{R}_{G, \Gamma} = \{  \rho : \Gamma \to G  \} /\!\!/ G
\end{equation*}	
consisting of group homomorphisms from the group $\Gamma$ to the Lie group $G$, considered up to conjugation.  Here, the double bar indicates that the quotient is being taken in the algebraic geometric sense of geometric invariant theory \cite{Mumford94}.   
		
We are interested in studying the character variety $\mathcal{R}_{\mathrm{SL}_3, \pi_1(S)}$, which we simply denote by $\mathcal{R}_{\mathrm{SL}_3, S}$, in the case where the group $\Gamma = \pi_1(S)$ is the fundamental group of a finite-type punctured surface $S$ with negative Euler characteristic, and where the Lie group $G=\mathrm{SL}_3$ is the special linear group.  

Sikora  \cite{SikoraTrans01} associated to any  $\mathrm{SL}_3$-web $W$   in the surface $S$ (Figure \ref{fig:coordinates-example-alt}) a trace regular function $\mathrm{Tr}_W \in \mathcal{O}(\mathcal{R}_{\mathrm{SL}_3, S})$ on the $\mathrm{SL}_3$-character variety.  A theorem of Sikora--Westbury \cite{SikoraAlgGeomTop07} implies that the preferred subset $\mathcal{W}_{3,S}$ of reduced $\mathrm{SL}_3$-webs indexes, by taking trace functions,  a linear basis for the algebra $\mathcal{O}(\mathcal{R}_{\mathrm{SL}_3, S})$ of regular functions on the $\mathrm{SL}_3$-character variety.
	
In a companion paper \cite{DouglasArxiv20}, we constructed explicit nonnegative integer coordinates for this $\mathrm{SL}_3$-web basis $\mathcal{W}_{3, S}$.  In particular, we identified $\mathcal{W}_{3, S}$ with the set of solutions in $\mathbb{Z}_{\geq 0}^N$ of finitely many Knutson--Tao inequalities \cite{KnutsonJAmerMathsoc99} and modulo 3 congruence conditions.  These coordinates depend on a choice of an ideal triangulation $\mathcal{T}$ of the punctured surface $S$.
	
In the present article, we prove that these web coordinates satisfy a surprising naturality property with respect to this choice of ideal triangulation $\mathcal{T}$.  Specifically, if another ideal triangulation $\mathcal{T}^\prime$ is chosen, then the induced coordinate change map takes the form of a tropicalized $\mathcal{A}$-coordinate cluster transformation \cite{FockIHES06, FominJAmerMathSoc02}.

\subsection*{Global aspects}
	
More precisely, let $\widehat{S}$ be a marked surface, namely a compact oriented surface together with a finite subset $M \subset \partial \widehat{S}$ of preferred points, called marked points, lying on some of the boundary components of $\widehat{S}$.  By a puncture we mean a boundary component of $\widehat{S}$ containing no marked points, which is thought of as shrunk down to a point.  We say the surface $\widehat{S} = S$ is non-marked if $M = \emptyset$.  We always assume that $\widehat{S}$ admits an ideal triangulation $\mathcal{T}$, namely a triangulation whose vertex set is equal to the set of punctures and marked points.  See Section \ref{ssec:markedsurfacesidealtriangulations}.   

\subsubsection*{Fock--Goncharov duality}
	
Fock--Goncharov \cite{FockIHES06} introduced a pair of mutually dual moduli spaces $\mathcal{X}_{\mathrm{PGL}_n,\widehat{S}}$ and $ \mathcal{A}_{\mathrm{SL}_n,\widehat{S}}$ (as well as for more general Lie groups).  In the case $\widehat{S} = S$ of non-marked surfaces, the spaces $\mathcal{X}_{\mathrm{PGL}_n,S}$ and $\mathcal{A}_{\mathrm{SL}_n,S}$ are variations of the $\mathrm{PGL}_n$- and $\mathrm{SL}_n$-character varieties; for $n=2$, they generalize the enhanced Teichm\"uller space \cite{Fock07} and the decorated Teichm\"uller space \cite{penner1987decorated}, respectively.  Fock--Goncharov duality is a canonical mapping
\begin{equation*}
\mathbb{I} : \mathcal{A}_{\mathrm{SL}_n, S}(\mathbb{Z}^t)
\to  \mathcal{O}(\mathcal{X}_{\mathrm{PGL}_n, S})
\end{equation*}
from the discrete set $\mathcal{A}_{\mathrm{SL}_n, S}(\mathbb{Z}^t)$ of tropical integer points of the moduli space $\mathcal{A}_{\mathrm{SL}_n,S}$ to the algebra $\mathcal{O}(\mathcal{X}_{\mathrm{PGL}_n, S})$ of regular functions on the moduli space $\mathcal{X}_{\mathrm{PGL}_n, S}$, satisfying enjoyable properties; for instance, the image of $\mathbb{I}$ should form a linear basis for the algebra of functions $\mathcal{O}(\mathcal{X}_{\mathrm{PGL}_n, S})$.  In the case $n=2$, Fock--Goncharov  gave a concrete topological construction of duality by identifying the tropical integer points with laminations on the surface.  

There are various ways to formulate  Fock--Goncharov duality.  A closely related version is 
\begin{equation*}
\mathbb{I} : \mathcal{A}_{\mathrm{PGL}_n, S}(\mathbb{Z}^t)
\to  \mathcal{O}(\mathcal{X}_{\mathrm{SL}_n, S})
\end{equation*} 
(compare \cite[Theorem 12.3 and the following Remark]{FockIHES06} for $n=2$).  There are also formulations of duality in the setting of marked surfaces $\widehat{S}$, where the moduli spaces $\mathcal{X}_{\mathrm{PGL}_n, \widehat{S}}$ and $\mathcal{X}_{\mathrm{SL}_n, \widehat{S}}$ are replaced \cite{GoncharovInvent15, GoncharovArxiv19} by slightly more general constructions $\mathcal{P}_{\mathrm{PGL}_n, \widehat{S}}$ and $\mathcal{P}_{\mathrm{SL}_n, \widehat{S}}$.  

Investigating Fock--Goncharov duality has led to many exciting developments.  By employing powerful conceptual methods (scattering diagrams, broken lines, theta functions, Donaldson--Thomas transformations),  works such as \cite{GoncharovAdvMath18, GoncharovArxiv19, GrossJAmerMathSoc18} have  established general formulations of duality.  On the other hand, explicit higher rank constructions, in the spirit of Fock--Goncharov's topological approach in the case $n=2$, are not as well understood.  

Following \cite{GoncharovInvent15} (see also \cite[Proposition 12.2]{FockIHES06}), we focus on the positive points $ \mathcal{A}_{\mathrm{PGL}_n, \widehat{S}}^+(\mathbb{Z}^t) \subset  \mathcal{A}_{\mathrm{PGL}_n, \widehat{S}}(\mathbb{Z}^t)$, defined with respect to the tropicalized Goncharov--Shen potential $P^t : \mathcal{A}_{\mathrm{PGL}_n, \widehat{S}}(\mathbb{Z}^t) \to \mathbb{Z}$ by $\mathcal{A}_{\mathrm{PGL}_n, \widehat{S}}^+(\mathbb{Z}^t) = (P^t)^{-1}(\mathbb{Z}_{\geq 0})$.  These positive tropical integer points play an important role in a variation of the previously mentioned duality, 
\begin{equation*}\tag{a}
\label{eq:duality-conjecture}
\mathbb{I} : \mathcal{A}_{\mathrm{PGL}_n, \widehat{S}}^+(\mathbb{Z}^t)
\to  \mathcal{O}(\mathcal{R}_{\mathrm{SL}_n, \widehat{S}})
\end{equation*}
(see \cite[Conjecture 10.11 and Theorem 10.12, as well as Theorems 10.14, 10.15 for $G=\mathrm{PGL}_2$]{GoncharovInvent15}).  Here, the space $\mathcal{R}_{\mathrm{SL}_n, \widehat{S}}$, introduced in \cite[Section $10.2$]{GoncharovInvent15} (they denote it  by $\mathrm{Loc}_{\mathrm{SL}_n, \widehat{S}}$), is a generalized (twisted) version of the $\mathrm{SL}_n$-character variety $\mathcal{R}_{\mathrm{SL}_n, S}$ valid for marked surfaces $\widehat{S}$.

As $\mathrm{PGL}_n$ is not simply connected, the moduli space $\mathcal{A}_{\mathrm{PGL}_n, \widehat{S}}$ does not have the standard Fock--Goncharov cluster structure, but it does have a positive structure.  So the tropical spaces $\mathcal{A}_{\mathrm{PGL}_n, \widehat{S}}(\mathbb{Z}^t)$ and $\mathcal{A}_{\mathrm{PGL}_n, \widehat{S}}^+(\mathbb{Z}^t)$ are defined; moreover, they are contained in the real tropical space $\mathcal{A}_{\mathrm{SL}_n, \widehat{S}}(\mathbb{R}^t)$, which has a tropical cluster structure.   Our goal is to construct, in the case $n=3$, a concrete topological model for the space $\mathcal{A}_{\mathrm{PGL}_3, \widehat{S}}^+(\mathbb{Z}^t)$ of positive tropical integer points, which also exhibits this tropical cluster structure.  

See  Appendix \ref{sec:preliminary-for-tropical-points} for a brief overview of the underlying Fock--Goncharov--Shen theory.

\subsubsection*{Topological indexing of linear bases}
	
One of our guiding principles is that tropical integer points should correspond to topological objects generalizing laminations \cite{Thurston97} on surfaces in the case $n=2$.  Such so-called higher laminations can be studied from many points of view, blending ideas from geometry, topology, and physics; see, for instance, \cite{akhmejanovMR4213127, FontaineCompositio13,  GoncharovInvent15, LeGeomTop16, MR4339354, XieArxiv13}.  In the present article, we focus attention on one of the topological approaches to studying higher laminations, via  webs \cite{CautisMathAnn14, KuperbergCommMathPhys96, SikoraTrans01}; see also \cite{GaiottoAnnHenriPoincare13}.  Webs are certain $n$-valent graphs-with-boundary embedded in the marked surface $\widehat{S}$ (considered up to equivalence in $\widehat{S} - M$).  Webs also appear naturally in the context of quantizations of character varieties via skein modules and algebras  \cite{jordan2021quantumdecoratedcharacterstacks, MullerQuantumTopology16, PrzytyckiBullPolishAcad91, SikoraAlgGeomTop05, Turaev89, WittenCommMathPhys89}.  

We begin by reviewing the case $n=2$.  For a marked surface $\widehat{S}$, define the set $\mathcal{L}_{2, \widehat{S}}$ of (positive bounded) $2$-laminations on $\widehat{S}$ so that  $\ell \in \mathcal{L}_{2, \widehat{S}}$ is a finite collection of mutually-non-intersecting simple loops and arcs on $\widehat{S}$  such that: first, there are no contractible loops; and, second, arcs end only on boundary components of $\widehat{S}$ containing marked points, and there are no arcs contracting to a boundary interval without marked points.  

In the case where the surface $\widehat{S} = S$ is non-marked, a 2-lamination $\ell \in \mathcal{L}_{2, S}$ corresponds to a trace function $\mathrm{Tr}_\ell \in \mathcal{O}(\mathcal{R}_{\mathrm{SL}_2, S})$, namely the regular function on the character variety $\mathcal{R}_{\mathrm{SL}_2, S}$ defined by sending $\rho: \pi_1(S) \to \mathrm{SL}_2$ to the product $\prod_\gamma \mathrm{Tr}(\rho(\gamma))$ of the traces along the components $\gamma$ of $\ell$.  It is well-known \cite{BullockCommentMathHelv97, PrzytyckiTopology00} that the trace functions $\mathrm{Tr}_\ell$, varying over the $2$-laminations $\ell \in \mathcal{L}_{2, S}$, form a linear basis for the algebra $\mathcal{O}(\mathcal{R}_{\mathrm{SL}_2, S})$ of regular functions on the $\mathrm{SL}_2$-character variety.  

On the opposite topological extreme, consider the case where the surface $\widehat{S} = \widehat{D}$ is a disk with $k$ marked points $m_i$ on its boundary, cyclically ordered.  For each $i$, assign a positive integer $n_i$ to the $i$-th boundary interval located between the marked points $m_i$ and $m_{i+1}$.  This determines a subset $\mathcal{L}_{2, \widehat{D}}(n_1, \dots, n_k) \subset \mathcal{L}_{2, \widehat{D}}$ consisting of the $2$-laminations $\ell$ having geometric intersection number equal to $n_i$ on the $i$-th boundary interval.  It follows from the Clebsch--Gordan theorem (see, for instance, \cite[Section $2.2$,$2.3$]{KuperbergCommMathPhys96}) that the subset $\mathcal{L}_{2, \widehat{D}}(n_1, \dots, n_k)$ of $2$-laminations indexes a linear basis for the space of $\mathrm{SL}_2$-invariant tensors $(V_{n_1}  \otimes \cdots \otimes V_{n_k})^{\mathrm{SL}_2}$, where $V_{n_i}$ is the unique $n_i$-dimensional irreducible representation of $\mathrm{SL}_2$.  

For a general marked surface $\widehat{S}$, Goncharov--Shen's moduli space $\mathcal{R}_{\mathrm{SL}_2, \widehat{S}}$ simultaneously generalizes both (a twisted version of) the character variety $\mathcal{R}_{\mathrm{SL}_2, S}$ for non-marked surfaces $\widehat{S} = S$, as well as the spaces of invariant tensors $(V_{n_1} \otimes V_{n_2} \otimes \cdots \otimes V_{n_k})^{\mathrm{SL}_2}$ for marked disks  $\widehat{S} = \widehat{D}$.   By \cite[Theorem 10.14]{GoncharovInvent15}, the set of $2$-laminations $\mathcal{L}_{2, \widehat{S}}$ canonically indexes a linear basis for the algebra of functions $\mathcal{O}(\mathcal{R}_{\mathrm{SL}_2, \widehat{S}})$ on the generalized character variety for the marked surface $\widehat{S}$, closely related to the linear bases in the specialized cases $\widehat{S} = S$ and $\widehat{S} = \widehat{D}$.  

We now turn to the case $n=3$.  In the setting of the disk $\widehat{S} = \widehat{D}$ with $k$ marked points on its boundary, the integers $n_i$ are replaced with highest weights $\lambda_i$ of irreducible $\mathrm{SL}_3$-representations $V_{\lambda_i}$, and the object of interest is the space $(V_{\lambda_1} \otimes V_{\lambda_2} \otimes \cdots \otimes V_{\lambda_k})^{\mathrm{SL}_3}$ of $\mathrm{SL}_3$-invariant tensors.  Kuperberg \cite{KuperbergCommMathPhys96} proved that the set $\mathcal{W}_{3, \widehat{D}}(\lambda_1, \dots, \lambda_k)$ of non-convex non-elliptic $3$-webs $W$ on $\widehat{D}$, matching certain fixed topological boundary conditions corresponding to the weights $\lambda_i$, indexes a linear basis for the invariant space $(V_{\lambda_1} \otimes V_{\lambda_2} \otimes \cdots \otimes V_{\lambda_k})^{\mathrm{SL}_3}$ (so can be thought of as the $\mathrm{SL}_3$-analogue of the subset $\mathcal{L}_{2, \widehat{D}}(n_1, \dots, n_k) \subset \mathcal{L}_{2, \widehat{D}}$).  

On the other hand, for non-marked surfaces $\widehat{S} = S$, Sikora \cite{SikoraTrans01} constructed,  for any $3$-web $W$ on $S$, a trace function $\mathrm{Tr}_W$ on the character variety $\mathcal{R}_{\mathrm{SL}_3, S}$, generalizing the trace functions $\mathrm{Tr}_\ell$ for $2$-laminations $\ell \in \mathcal{L}_{2, S}$ (Sikora also constructed $\mathrm{Tr}_W \in \mathcal{O}(\mathcal{R}_{\mathrm{SL}_n, S})$ for any $n$-web $W$).  A theorem of Sikora--Westbury \cite{SikoraAlgGeomTop07} implies that the subset $\mathcal{W}_{3, S}$ of non-elliptic $3$-webs $W$ indexes, by taking trace functions $\mathrm{Tr}_W$,  a linear basis for the algebra  of regular functions $\mathcal{O}(\mathcal{R}_{\mathrm{SL}_3, S})$ on the $\mathrm{SL}_3$-character variety.  

For a general marked surface $\widehat{S}$, Frohman--Sikora's work \cite{FrohmanMathZ2022} suggests that a good definition for the (positive bounded) $3$-laminations  is the set $\mathcal{W}_{3, \widehat{S}}$ of reduced $3$-webs $W$ on $\widehat{S}$, which in particular are allowed to have boundary; see Section \ref{section:wa}.  Indeed, by \cite[Proposition 4]{FrohmanMathZ2022}, this set $\mathcal{W}_{3, \widehat{S}}$ forms a linear basis  for the reduced $\mathrm{SL}_3$-skein algebra. As for non-marked surfaces $S$, where skein algebras quantize character varieties, we suspect that Frohman--Sikora's reduced $\mathrm{SL}_3$-skein algebra is a quantization of Goncharov--Shen's generalized $\mathrm{SL}_3$-character variety $\mathcal{R}_{\mathrm{SL}_3, \widehat{S}}$. In particular, we suspect that the set $\mathcal{W}_{3, \widehat{S}}$ indexes a canonical linear basis for the algebra of regular functions $\mathcal{O}(\mathcal{R}_{\mathrm{SL}_3, \widehat{S}})$, generalizing the case $n=2$ \cite[Theorem 10.14]{GoncharovInvent15}; see   \cite[Conjecture $23$]{FrohmanMathZ2022}.  

\subsubsection*{Tropical coordinates for higher laminations}  
	
Let a positive integer cone mean a subset of $\mathbb{Z}_{\geq 0}^k$ closed under addition and containing zero.  

As in \cite{FockIHES06, Fock07}, in the case $n=2$, given a choice of ideal triangulation $\mathcal{T}$, with $N_2$ edges,  of the marked surface $\widehat{S}$,  one assigns $N_2$ nonnegative integer coordinates to a given $2$-lamination $\ell \in \mathcal{L}_{2, \widehat{S}}$ by taking the geometric intersection numbers of $\ell$ with the edges of the ideal triangulation $\mathcal{T}$.   This assignment determines an injective coordinate mapping
\begin{equation*}
\Phi^{(2)}_\mathcal{T} : \mathcal{L}_{2, \widehat{S}} \hookrightarrow \mathbb{Z}_{\geq 0}^{N_2}
\end{equation*}
on the set of $2$-laminations $\mathcal{L}_{2, \widehat{S}}$.  Moreover, the image of $\Phi_\mathcal{T}^{(2)}$ is a positive integer cone in $\mathbb{Z}_{\geq 0}^{N_2}$, which is characterized as the set of solutions of finitely many inequalities and parity conditions of the form
\begin{equation*}
a+b-c \geq 0 
\text{ and }
a + b - c \in  2\mathbb{Z}
\,\,
(  a, b, c \in \mathbb{Z}_{\geq 0}  ).
\end{equation*}

Moreover, these integer coordinates are natural with respect to the choice of $\mathcal{T}$, in the sense that if a different ideal triangulation $\mathcal{T}^\prime$ is chosen, then the induced coordinate transformation is the $\mathrm{SL}_2$ tropical $\mathcal{A}$-coordinate cluster transformation \cite[Figure 8]{Fock07}.  These natural coordinates provide an identification $\mathcal{L}_{2, \widehat{S}} \cong \mathcal{A}_{\mathrm{PGL}_2, \widehat{S}}^+(\mathbb{Z}^t)$ as in \cite[Theorem 10.15]{GoncharovInvent15}.  Taken together, \cite[Theorems 10.14, 10.15]{GoncharovInvent15} constitute a compelling topological version of the duality \eqref{eq:duality-conjecture}  in the case $n=2$; see \cite[the two paragraphs after Theorem 10.15]{GoncharovInvent15}.  

Our main result generalizes these natural coordinates to the setting $n=3$.  

More precisely, given an ideal triangulation $\mathcal{T}$ of a marked surface $\widehat{S}$, put $N_3$ to be twice the number of edges (including boundary edges) of $\mathcal{T}$ plus the number of triangles of $\mathcal{T}$.  Recall the set $\mathcal{W}_{3, \widehat{S}}$ of (equivalence classes of) reduced $3$-webs on $\widehat{S}$, discussed above.  

\begin{theorem}
\label{thm:first-theorem-intro}
Given an ideal triangulation $\mathcal{T}$ of the marked surface $\widehat{S}$,  there is an injection
\begin{equation*}
\Phi_\mathcal{T} : \mathcal{W}_{3, \widehat{S}}
\hookrightarrow
\mathbb{Z}_{\geq 0}^{N_3}
\end{equation*}  
satisfying the property that the image of $\Phi_\mathcal{T}$ is a positive integer cone in  $\mathbb{Z}_{\geq 0}^{N_3}$,  which is characterized  as the set of solutions of finitely many Knutson--Tao rhombus inequalities \cite{KnutsonJAmerMathsoc99} and modulo $3$ congruence conditions of the form
\begin{equation*}
a+b-c-d \geq 0 
\text{ and }
a + b - c-d   \in  3\mathbb{Z}
\,\,
(  a, b, c, d \in \mathbb{Z}_{\geq 0}  ).
\end{equation*}

Moreover, these coordinates are natural with respect to the action of the mapping class group of the marked surface $\widehat{S}$.  More precisely, if a different ideal triangulation $\mathcal{T}^\prime$ is chosen, then the coordinate change map relating $\Phi_\mathcal{T}$ and $\Phi_{\mathcal{T}^\prime}$ is given by the $\mathrm{SL}_3$ tropical $\mathcal{A}$-coordinate cluster transformation \cite{FockIHES06, FominJAmerMathSoc02}, expressed locally as in  \eqref{eq:boundarycoords}-\eqref{equation:mu4}; see Figure {\upshape\ref{figure:flip}}.
\end{theorem}

See Theorems \ref{thm:main-theorem}, \ref{thm:second-main-theorem} and Corollary \ref{cor:second-main-theorem}.  The construction of $\Phi_\mathcal{T}$ (Theorem \ref{thm:main-theorem}) was done in \cite{DouglasArxiv20}.

This construction was motivated by earlier work of Xie \cite{XieArxiv13} and Goncharov--Shen \cite{GoncharovInvent15}.  

In particular, Goncharov--Shen used the Knutson--Tao rhombus inequalities associated to an ideal triangulation $\mathcal{T}$ of  $\widehat{S}$ to index the set  of positive $\mathcal{A}$ tropical integer points, which they showed parametrizes a linear basis for the algebra of regular functions $\mathcal{O}(\mathcal{R}_{\mathrm{SL}_3, \widehat{S}})$; see \cite[Section $3.1$ and Theorem 10.12 (stated for more general Lie groups)]{GoncharovInvent15}.  Their parametrization is not mapping class group equivariant; see the remark in \cite[page 614]{GoncharovInvent15} immediately after the aforementioned theorem.  In \cite{GoncharovArxiv19} they construct equivariant bases using the abstract machinery of \cite{GrossJAmerMathSoc18}.  Theorem \ref{thm:first-theorem-intro} provides a concrete  model  indexing the set $\mathcal{A}_{\mathrm{PGL}_3, \widehat{S}}^+(\mathbb{Z}^t)$ of positive tropical integer points, also based on the Knutson--Tao inequalities, which in addition is equivariant with respect to the action of the mapping class group.  This natural indexing $\mathcal{W}_{3, \widehat{S}} \cong \mathcal{A}_{\mathrm{PGL}_3, \widehat{S}}^+(\mathbb{Z}^t)$ provided by Theorem \ref{thm:first-theorem-intro}  generalizes the $n=2$ case \cite[Theorem $10.15$]{GoncharovInvent15}.

We think of the web coordinates of Theorem \ref{thm:first-theorem-intro} as positive tropical integer $\mathcal{A}$-coordinates.  We call the positive integer cone $\Phi_\mathcal{T}(\mathcal{W}_{3, \widehat{S}}) \subset \mathbb{Z}_{\geq 0}^{N_3}$ the $\mathrm{SL}_3$ Knutson--Tao--Goncharov--Shen (KTGS) cone  with respect to the ideal triangulation $\mathcal{T}$ of $\widehat{S}$.

These  tropical web coordinates were constructed for some simple examples, such as the triangle webs shown in Figure  \ref{figure:triangle}, in \cite{XieArxiv13}.  They also appeared implicitly in \cite[Theorem 8.22]{SunGeomFunctAnal20}, in the geometric context of eruption flows on the $\mathrm{PGL}_n(\mathbb{R})$-Hitchin component ($n=3$).  Xie \cite{XieArxiv13} checked the mapping class group equivariance, in the above sense, of these coordinates on a handful of examples.

Frohman--Sikora \cite{FrohmanMathZ2022} independently constructed nonnegative integer coordinates for the set $\mathcal{W}_{3, \widehat{S}}$ of reduced 3-webs.  Their coordinates are  related to, but different than, the coordinates of Theorem \ref{thm:first-theorem-intro}.  

As an application, Kim \cite{KimArxiv20} constructed an explicit $\mathrm{SL}_3$-version of Fock--Goncharov  duality using the tropical web coordinates of Theorem \ref{thm:first-theorem-intro}.  We expect that Kim's approach, together with the $\mathrm{SL}_3$-quantum trace map \cite{MR4790744, Douglas1, KimArxiv20}, will lead to an explicit $\mathrm{SL}_3$-version of quantum Fock--Goncharov duality \cite{FockENS09}; see \cite{AllegrettiAdvMath17} for the $n=2$ case.  

As another application, Ishibashi--Kano \cite{Ishibashi22} generalized the coordinates of Theorem \ref{thm:first-theorem-intro} to an $\mathrm{SL}_3$-version of shearing coordinates for (unbounded) 3-laminations (with pinnings).  

To end this section, we briefly recall from \cite{DouglasArxiv20} the construction of the coordinate map $\Phi_\mathcal{T}$ from Theorem \ref{thm:first-theorem-intro}; see Section \ref{section:wa}.  Given the ideal triangulation $\mathcal{T}$, form the split ideal triangulation $\widehat{\mathcal{T}}$ by replacing each edge $E$ of $\mathcal{T}$ with two parallel edges $E^\prime$ and $E^{\prime\prime}$; in other words, fatten each edge $E$ into a bigon.  One then puts a given reduced $3$-web $W \in \mathcal{W}_{3, \widehat{S}}$ into good position with respect to the split ideal triangulation $\widehat{\mathcal{T}}$.  The result is that most of the complexity of the $3$-web $W$ is pushed into the bigons (Figure \ref{figure:bigon}), whereas over each triangle there is only a single (possibly empty) honeycomb together with finitely many arcs lying on the corners (Figure \ref{figure:localtr}).  Once the 3-web $W$ is in good position, its coordinates $\Phi_\mathcal{T}(W) \in \mathbb{Z}_{\geq 0}^{N_3}$ are readily computed.  For an example in the once punctured torus, see Figure \ref{fig:coordinates-example-alt}.

As has already been partly discussed, in principle the model presented in this paper should be translatable into the language of \cite{akhmejanovMR4213127, GoncharovArxiv19, le2019imrn, le2019cluster}.  In particular, compare Theorem \ref{thm:second-main-theorem} to the main result of \cite{akhmejanovMR4213127}.

\subsection*{Local aspects}
	
The first new contribution of the present work is a proof of the naturality statement appearing in Theorem \ref{thm:first-theorem-intro}; see Section \ref{section:fe}.  This is a completely local statement, since any two ideal triangulations $\mathcal{T}$ and $\mathcal{T}^\prime$ are related by a sequence of diagonal flips inside ideal squares.  It therefore suffices to check the desired tropical coordinate change formulas  for a single square:
\begin{equation*}\tag{b}
\label{eq:boundarycoords}
x_i = x_i^\prime \,\, ( i=1,2,\dots,8 ),
\end{equation*}
\begin{equation*}\tag{c}
\label{equation:mu1}
\max\{x_2+y_3, y_1+x_3\}-y_2=z_2,
\end{equation*}
\begin{equation*}\tag{d}
\label{equation:mu2}
\max\{y_1+x_6, x_7+y_3\}-y_4=z_4,
\end{equation*}
\begin{equation*}\tag{e}
\label{equation:mu3}
\max\{x_1^\prime+z_4, x_8^\prime+z_2\}-y_1=z_1,
\end{equation*}
\begin{equation*}\tag{f}
\label{equation:mu4}
\max\{z_2+x_5^\prime, z_4+x_4^\prime\}-y_3=z_3.
\end{equation*}
See Figure \ref{figure:flip} for the notation.  

\begin{figure}[t]
\includegraphics[scale=0.64]{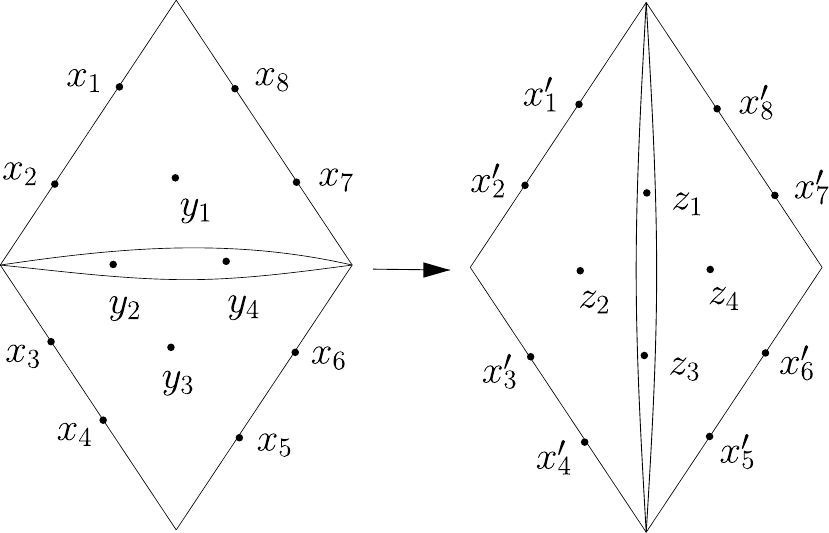}
\caption{     Local $\mathrm{SL}_3$ tropical $\mathcal{A}$-coordinate cluster transformation, corresponding to a diagonal flip $\mathcal{T} \to \mathcal{T}^\prime$ in the square.  See \eqref{eq:boundarycoords}-\eqref{equation:mu4}.  }
\label{figure:flip}
\end{figure}

Given a $3$-web $W \in \mathcal{W}_{3, \widehat{S}}$ in good position with respect to $\mathcal{T}$, the restriction $W|_\Box$ of $W$ to a triangulated ideal square $(\Box, \mathcal{T}|_\Box) \subset (\widehat{S}, \mathcal{T})$   falls into one of 42 families $\mathcal{W}^k_{\mathcal{T}|_\Box} \subset \mathcal{W}_{3, \Box}$ for $k=1, 2, \dots, 42$; see Section \ref{ssec:42-reduced-web-families-in-the-square}.    Depending on which family $\mathcal{W}^k_{\mathcal{T}|_\Box}$ the restricted web $W|_\Box$ belongs to, there is an explicit topological description of how $W|_\Box$ rearranges itself into good position after the flip; see Appendix \ref{ssec:flip-examples}.  These local 42 families of $3$-webs in the square  have a  geometric interpretation, leading to our second main result.  

Let $\widehat{S} = \Box$ be a disk with four marked points, namely an ideal square, and let $\mathcal{T}$ be a choice of diagonal of $\Box$.  Theorem \ref{thm:first-theorem-intro} says that the set $\mathcal{W}_{3, \Box}$ of reduced $3$-webs in $\Box$ embeds via $\Phi_\mathcal{T}$ as a positive integer cone inside $\mathbb{Z}_{\geq 0}^{12}$.  This cone possesses a finite subset of irreducible elements spanning it over $\mathbb{Z}_{\geq 0}$, called its Hilbert basis \cite{hilbert1890ueber, schrijver1981total}; see Section \ref{section:Hsq}.

\begin{theorem}
\label{thm:second-theorem-intro}
The Knutson--Tao--Goncharov--Shen cone $\Phi_\mathcal{T}(\mathcal{W}_{3, \Box}) \subset \mathbb{Z}_{\geq 0}^{12}$ associated to the triangulated ideal square $(\Box, \mathcal{T})$ has a Hilbert basis consisting of $22$ elements, corresponding via $\Phi_\mathcal{T}$ to $22$ reduced $3$-webs $W_\mathcal{T}^i \in \mathcal{W}_{3, \Box}$ for $i=1, 2, \dots, 22$.  
	
Moreover, this positive integer cone
\begin{equation*}
\Phi_\mathcal{T}(\mathcal{W}_{3, \Box}) = \bigcup_{k=1}^{42} \mathcal{C}_\mathcal{T}^k 
 \subset \mathbb{Z}_{\geq 0}^{12}
\end{equation*}
can be decomposed into $42$ sectors $\mathcal{C}_\mathcal{T}^k$ such that:  
\begin{enumerate}[label=\textnormal{(\Roman*)}]     
\item each sector is generated over $\mathbb{Z}_{\geq 0}$ by $12$ of the $22$ Hilbert basis elements;  
\item adjacent sectors are separated by a codimension $1$ wall, and these $42$ sectors $\mathcal{C}_\mathcal{T}^k$ are  in one-to-one correspondence with the $42$ families $\mathcal{W}^k_{\mathcal{T}} \subset \mathcal{W}_{3, \Box}$ of $3$-webs in the square, discussed just above.
\end{enumerate}

Lastly,  each family $\mathcal{W}^k_{\mathcal{T}} \subset \mathcal{W}_{3, \Box}$ contains $12$ distinguished $3$-webs $W_\mathcal{T}^{i(k, j)} \in \{ W_\mathcal{T}^i \}_{i=1,2,\dots,22}$ for $j=1, 2, \dots, 12$, corresponding via $\Phi_\mathcal{T}$ to the $12$ Hilbert basis elements generating the sector $\mathcal{C}_\mathcal{T}^k$.  We refer to the set $\{ W_\mathcal{T}^{i(k, j)} \}_{j=1,2,\dots,12}$ of these $12$ distinguished $3$-webs as the topological type of the sector $\mathcal{C}_\mathcal{T}^k$.  Then, two sectors $\mathcal{C}_\mathcal{T}^k$ and $\mathcal{C}_\mathcal{T}^{k^\prime}$ are adjacent if and only if their topological types  differ by exactly one distinguished $3$-web; see Figure {\upshape\ref{figure:wallscross}}.  
\end{theorem}

\begin{figure}[t]
\includegraphics[width=.95\textwidth]{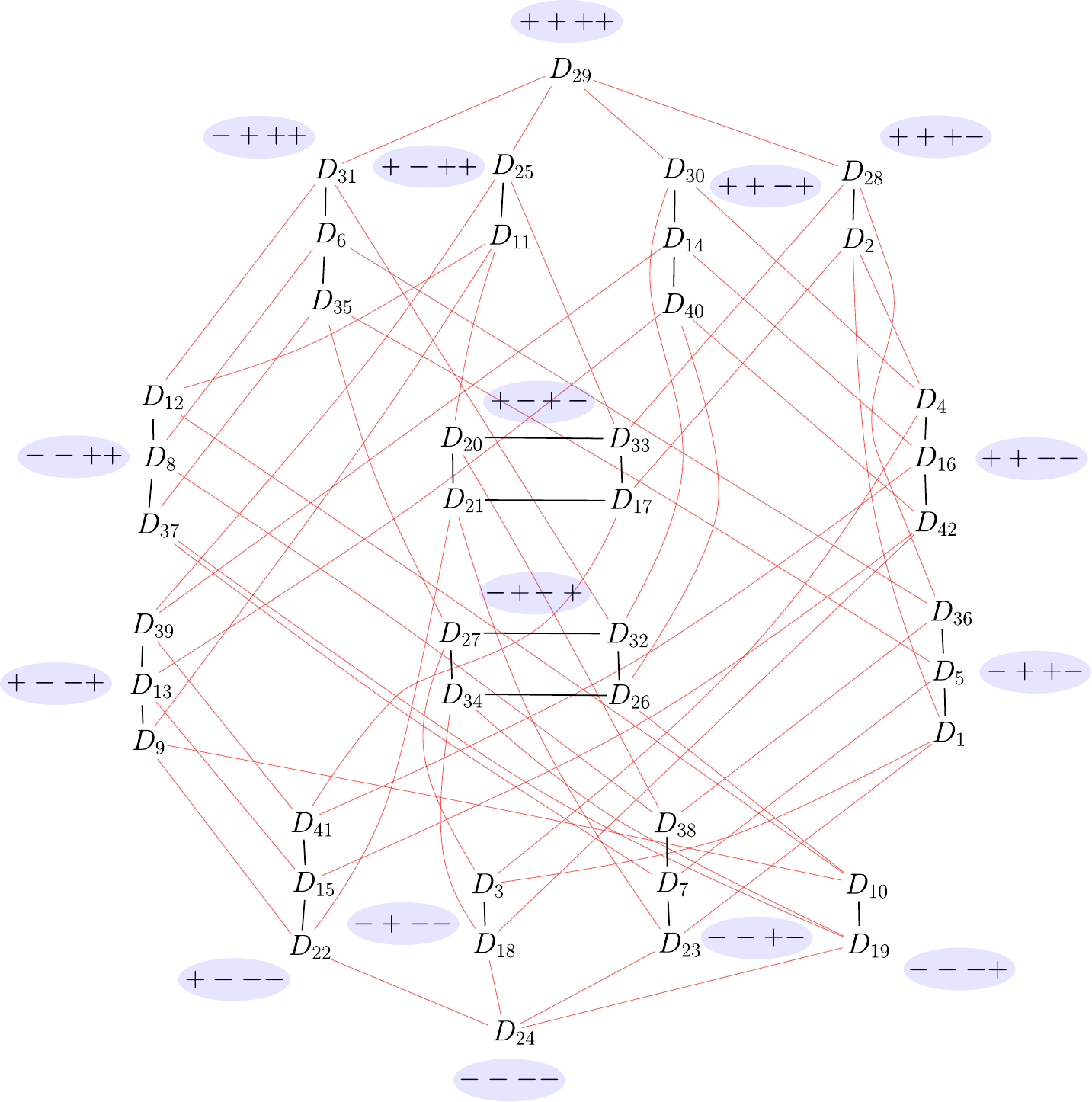}
\caption{     Sectors and walls in the Knutson--Tao--Goncharov--Shen (KTGS) cone $\Phi_\mathcal{T}(\mathcal{W}_{3, \Box}) \subset \mathbb{Z}_{\geq 0}^{12}$ for a triangulated ideal square $(\Box, \mathcal{T})$.  More precisely, displayed is a corresponding sector decomposition $\{ D_i \}_{i=1,2,\dots,42}$ of (a projection to $\mathbb{R}^4$ of a real version of) an isomorphic cone in $\mathbb{Z}_+^8 \times \mathbb{Z}^4$, obtained from the KTGS cone via a transformation  defined using the 4 tropical integer $\mathcal{X}$-coordinates.  The sectors $D_i$ are grouped depending on which orthant of $\mathbb{R}^4$ they belong to.  These sectors  are the vertices of a 4-valent graph, where two sectors are connected by an edge if and only if they share a wall; equivalently, their topological types differ by a single web.  See Theorem \ref{thm:second-theorem-intro}. }  
\label{figure:wallscross}
\end{figure} 

See Theorems \ref{theorem:basis} and \ref{theorem:decomp} (as well as Remark \ref{rem:topwallcrossphen}).

For a related appearance of Hilbert bases, in the $n=2$ setting, see \cite{AbdielAlgGeomTop17}. 

The proof of Theorem \ref{thm:second-theorem-intro} is geometric in nature and might be of independent interest. Recall \cite{FockIHES06} there are two dual sets of coordinates for the two dual moduli spaces of interest, respectively, the $\mathcal{A}$-coordinates and the $\mathcal{X}$-coordinates, as well as their tropical counterparts.  For a triangulated ideal square $(\Box, \mathcal{T})$, via the mapping $\Phi_\mathcal{T}$ each $3$-web $W \in \mathcal{W}_{3, \Box}$ is assigned 12 positive tropical integer $\mathcal{A}$-coordinates $\Phi_\mathcal{T}(W) \in \mathbb{Z}_{\geq 0}^{12}$.  We show that there are also assigned to $W$ four internal tropical integer $\mathcal{X}$-coordinates valued in $\mathbb{Z}$, two associated to the unique internal edge of $\mathcal{T}$ and one for each triangle of $\mathcal{T}$; see Figure \ref{fig:tropical-X-coords}.  We find that the decomposition of the $\mathcal{A}$-cone $\Phi_\mathcal{T}(\mathcal{W}_{3, \Box}) \subset \mathbb{Z}^{12}_{\geq 0}$ into 42 sectors is mirrored by a corresponding decomposition of the $\mathcal{X}$-lattice $\mathbb{Z}^4$ into 42 sectors; see   Figure \ref{figure:wallscross}.  We think of this as a manifestation of Fock--Goncharov's tropicalized canonical map:
\begin{equation*}
p^t : \mathcal{W}_{3, \Box} \cong \Phi_\mathcal{T}(\mathcal{W}_{3, \Box}) \cong \mathcal{A}^+_{\mathrm{PGL}_3, \Box}(\mathbb{Z}^t)_\mathcal{T} \subset \mathcal{A}_{\mathrm{SL}_3, \Box}(\mathbb{R}^t)_\mathcal{T} \overset{\mathrm{canonical}}{\to} \mathcal{X}_{\mathrm{PGL}_3, \Box}(\mathbb{R}^t)_\mathcal{T}.
\end{equation*} 
The image of the map $p^t$ is $\mathcal{X}_{\mathrm{PGL}_3, \Box}(\mathbb{Z}^t)_\mathcal{T} \cong \mathbb{Z}^4$, and $p^t$ maps sectors of the positive integer cone $\Phi_\mathcal{T}(\mathcal{W}_{3, \Box}) \cong \mathcal{A}^+_{\mathrm{PGL}_3, \Box}(\mathbb{Z}^t)_\mathcal{T}$ to sectors of the integer lattice $\mathcal{X}_{\mathrm{PGL}_3, \Box}(\mathbb{Z}^t)_\mathcal{T} \cong \mathbb{Z}^4$.  See Section \ref{section:ssq}.

\section*{Acknowledgements}
We are profoundly grateful to Dylan Allegretti, Francis Bonahon, Charlie Frohman, Sasha Goncharov, Linhui Shen, Daping Weng, Tommaso Cremaschi, and Subhadip Dey for many very helpful conversations and for generously offering their time as this project developed.  

Much of this work was completed during very enjoyable visits to Tsinghua University in Beijing, supported by a GEAR graduate internship grant, and the University of Southern California in Los Angeles.  We would like to take this opportunity to extend our enormous gratitude to these institutions for their warm hospitality.

\section{Tropical points and webs}
\label{section:wa}

We introduce the main object of study, the Knutson-Tao-Goncharov-Shen cone $\mathcal{C}_\mathcal{T} \subset \mathbb{Z}_+^N$ associated to an ideal triangulation $\mathcal{T}$ of a marked surface $\widehat{S}$, and we summarize the work of \cite{DouglasArxiv20} relating tropical points to topological objects called webs.

\subsection{Marked surfaces, ideal triangulations, and rhombi}
\label{ssec:markedsurfacesidealtriangulations}
	  
\begin{definition}
A \emph{marked surface} $\widehat{S}$ is a pair $(S,m_b)$ where $S$ is a compact oriented finite-type surface with at least one boundary component, and $m_b \subset \partial S$ is a finite set of \emph{marked points} on $\partial S$.  Let $m_p \subset \{ \text{components of } \partial S \}$ be the set of \emph{punctures}, defined as the subset of boundary components without marked points; as is common in the literature, for the remainder of the article we identify such unmarked boundary components in $m_p$ with the  (actual) punctures obtained by removing them and shrinking the resulting hole down to a point.

We assume the Euler characteristic condition $\chi(S) < d/2$, where $d$ is the number of components  of $\partial S - m_b$ limiting to a marked point.  (For example, $d=3$ for a once punctured disk with three marked points on its boundary.)  This topological condition is equivalent to the existence of an \emph{ideal triangulation} $\mathcal{T}$ of $\widehat{S}$, namely a triangulation of the compactified surface whose set of vertices is equal to $m_b \cup m_p$; the vertices of $\mathcal{T}$ are called  \emph{ideal vertices}.

For simplicity, we always assume that $\mathcal{T}$ does not contain any \emph{self-folded triangles}.  That is, we assume each triangle of $\mathcal{T}$ has three distinct sides.  (Our results should generalize, essentially without change, to allow for self-folded triangles.)

Given an ideal triangulation $\mathcal{T}$ of $\widehat{S}$, we define the \emph{ideal $3$-triangulation} $\mathcal{T}_3$ of $\mathcal{T}$ to be the triangulation of $\widehat{S}$ obtained by subdividing each ideal triangle $\Delta$ of $\mathcal{T}$ into $9$ triangles; see Figure \ref{figure:acoor}.  The 3-triangulation $\mathcal{T}_3$ has as many ideal vertices as $\mathcal{T}$, and has $N$ \emph{non-ideal vertices}, where $N$ is defined in Notation \ref{not:plus2}.

A \emph{pointed ideal triangle} is a triangle $\Delta$ in an ideal triangulation $\mathcal{T}$ together with a preferred ideal vertex; $\Delta$ is called a \emph{pointed ideal 3-triangle} when subdivided as part of the associated 3-triangulation $\mathcal{T}_3$.   

Given a pointed ideal 3-triangle, we may talk about the three associated \emph{rhombi}; see Figure \ref{figure:acoor1}.  In the figure, the red rhombus is called the \emph{corner rhombus}, and the yellow and green rhombi are called the \emph{interior rhombi}.  Each rhombus has two \emph{acute vertices} and two \emph{obtuse vertices}.  Note that exactly one of these eight vertices, the \emph{corner vertex}, is an ideal vertex of $\mathcal{T}_3$; specifically, the top (acute) vertex of the corner rhombus.  (We will see that the other vertices correspond to Fock-Goncharov $\mathcal{A}$-coordinates.)
\end{definition}

\begin{notation}
\label{not:plus2}
$ $
\begin{enumerate}[label=\textnormal{(\Roman*)}]
\item  The natural number $N$ is defined as twice the total number of edges (including boundary edges) of $\mathcal{T}$ plus the number of triangles of $\mathcal{T}$.  (Note that $N$ is what we called $N_3$ in the introduction.)
\item  It will be convenient to denote the nonnegative real numbers by $\mathbb{R}_+ = \mathbb{R}_{\geq 0}$ and the nonnegative integers by $\mathbb{Z}_+ = \mathbb{Z}_{\geq 0}$.  Similarly, put $\mathbb{R}_- = \mathbb{R}_{\leq 0}$ and $\mathbb{Z}_- = \mathbb{Z}_{\leq 0}$.
\end{enumerate}
\end{notation}

\subsection{The Knutson--Tao--Goncharov--Shen cone and reduced webs}
\label{ssec:reducedwebsandtheKTGScone}
	  
Let $\widehat{S}$ be a marked surface.  (In this subsection, we will use some of the terminology of Appendix \ref{sec:cones}.)

\subsubsection{KTGS cone}
\label{sssec:KTGS-cone}
	
\begin{definition}
Given a pointed ideal triangle $\Delta$ in an ideal triangulation $\mathcal{T}$ of $\widehat{S}$ (Section \ref{ssec:markedsurfacesidealtriangulations}), assume  integers (see also Remark \ref{rem:realsolutionsareinteger}\eqref{item:annoyingminusisgn})  $a,b,c,d \in \mathbb{Z}$ (resp. $a, b, c \in \mathbb{Z}$) are assigned to some interior (resp. corner) rhombus, where the numbers $a,b$ are assigned to the two obtuse vertices, and the numbers $c,d$ are assigned to the two acute vertices.  To such an assigned rhombus, we associate a \emph{Knutson-Tao rhombus inequality} $a+b-c-d \geq 0$ and a \emph{modulo 3 congruence condition} $(a+b-c-d)/3 \in \mathbb{Z}$.  Here, we set $d=0$ if the rhombus is a corner rhombus, where then $d$ corresponds to the corner vertex.  
\end{definition}

Recall  the definition (Notation \ref{not:plus2}) of the natural number $N$.  This is the same as the number of non-ideal points of the 3-triangulation $\mathcal{T}_3$.  We order these $N$ non-ideal points arbitrarily in the following definition,   so that to each such non-ideal point of $\mathcal{T}_3$ we associate a coordinate of $\mathbb{Z}^N$.  In this way, a point of $\mathbb{Z}^N$ assigns to each rhombus in a pointed ideal triangle $\Delta$ four numbers $a,b,c,d \in \mathbb{Z}$ as above.

\begin{definition}
\label{def:KTGS-cone}
Given an ideal triangulation $\mathcal{T}$ of $\widehat{S}$, let the \emph{Knutson-Tao-Goncharov-Shen cone} $\mathcal{C}_\mathcal{T} \subset \mathbb{Z}^N$, or just the \emph{KTGS cone} for short, be the submonoid (Definition \ref{def:submonoid}) defined by the property that its points satisfy all of the Knutson-Tao rhombus inequalities and modulo 3 congruence conditions, varying over all rhombi of all pointed ideal triangles $\Delta$ of $\mathcal{T}$.  
\end{definition}

\begin{prop}
\label{prop:KTGS-is-positive-integer-cone}
The KTGS cone $\mathcal{C}_\mathcal{T} \subset \mathbb{Z}_+^N \subset \mathbb{Z}^N$ is a positive integer cone \textup{(}Definition {\upshape\ref{def:positive-integer-cone}}\textup{)}.
\end{prop}

\begin{proof}
This is by \cite[Corollary 6.7 and Definition 6.10]{DouglasArxiv20}; see also Remark \ref{rem:realsolutionsareinteger}\eqref{item:annoyingminusisgn}.
\end{proof}

\begin{conceptremark}[throughout the paper, conceptual remarks make reference to the theory reviewed in Appendix \ref{sec:preliminary-for-tropical-points}]
\label{rem:conceptual-remarks}
We think of the KTGS cone $\mathcal{C}_\mathcal{T} \subset \mathbb{Z}_+^N$, defined above, as the isomorphic coordinate chart  
\begin{equation*}
\mathcal{C}_\mathcal{T} \cong -3 \mathcal{A}^+_{\mathrm{PGL}_3,\widehat{S}}(\mathbb{Z}^t)_\mathcal{T}
\end{equation*}
where  $\mathcal{A}_{\mathrm{PGL}_3,\widehat{S}}(\mathbb{Z}^t)_\mathcal{T} \subset ((1/3) \mathbb{Z})^N \cong \mathcal{A}_{\mathrm{SL}_3,\widehat{S}}((1/3)\mathbb{Z}^t)_\mathcal{T}$ and $\mathcal{A}^+_{\mathrm{PGL}_3,\widehat{S}}(\mathbb{Z}^t)_\mathcal{T} \subset   \mathcal{A}_{\mathrm{PGL}_3,\widehat{S}}(\mathbb{Z}^t)_\mathcal{T} \cap (-(1/3) \mathbb{Z}_+)^N$ as in Remark \ref{rem:realsolutionsareinteger}\eqref{item:annoyingminusisgn}.  
\end{conceptremark}

\subsubsection{Reduced webs}
\label{sssec:reduced-webs}
	
\begin{definition}
A \emph{web (possibly with boundary)} $W$ in $\widehat{S}$ \cite[Section $9.1$]{DouglasArxiv20} is an oriented trivalent graph embedded in $\widehat{S}$ such that:
\begin{enumerate}[label=\textnormal{(\Roman*)}]
\item  the boundary $\partial W = W \cap (\partial \widehat{S} - m_b)$ of the web lies on the boundary of the surface (minus the marked points) and may be nonempty, in which case its boundary points are required to be monovalent vertices;
\item  the three edges of $W$ at an internal vertex are either all oriented in or all oriented out;
\item  we allow  $W$ to have components homeomorphic to the circle, called \emph{loops}, which do not contain any vertices;
\item  we allow $W$ to have components homeomorphic to the closed interval, called \emph{arcs}, which have exactly two vertices on $\partial \widehat{S} - m_b$ and do not have any internal vertices.
\end{enumerate}
Webs are considered up to \emph{parallel equivalence}, meaning related either by an ambient isotopy of $\widehat{S} - m_b$ or a homotopy in $\widehat{S} - m_b$ exchanging two `parallel' loop (resp. arc) components of $W$ bounding an embedded annulus (resp. rectangle, two of whose sides are contained in $\partial \widehat{S} - m_b$, as  in Figure \ref{fig:boundary-parallel-move}).

A \emph{face} of a web $W$ \cite[Section $9.1$]{DouglasArxiv20}  is a contractible component of the complement $W^c \subset \widehat{S}$.  \emph{Internal} (resp. \emph{external}) faces are those not intersecting (resp. intersecting) the boundary $\partial \widehat{S} - m_b$.  A face with $n$ sides (counted with multiplicity, and including sides on the boundary $\partial \widehat{S} - m_b$) is called a \emph{$n$-face}.  Internal faces always have an even number of sides.  An \emph{external H-4-face} is an external 4-face limiting to a single component of $W$ (there is only one type of external 2- or 3-face).  

A web $W$ is \emph{reduced} if each internal face has at least six sides, and there are no external 2-, 3-, or H-4-faces.  (Reduced webs were called `rung-less essential webs' in \cite[Section $9.2$]{DouglasArxiv20}; see also \cite{FrohmanMathZ2022}.)  Denote by $\mathcal{W}_{\widehat{S}}$ the set of reduced webs up to parallel equivalence.  (Note that $\mathcal{W}_{\widehat{S}}$ is what we called $\mathcal{W}_{3, \widehat{S}}$ in the introduction.)  
\end{definition}

\begin{remark}
As a caution, throughout we tend to be sloppy about distinguishing between web equivalence classes and their representatives, for example writing $W \in \mathcal{W}_{\widehat{S}}$ to indicate a representative web $W$, as this distinction will generally be immaterial for our purposes.
\end{remark}

\begin{figure}[t]
\includegraphics[scale=1.08]{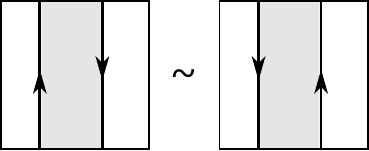}
\caption{     Boundary parallel move in the ideal square.}
\label{fig:boundary-parallel-move}
\end{figure}

\subsection{Web tropical coordinate map}
\label{ssec:tropicalwebcoordinatemap}
	  
In \cite[Section $9.2$]{DouglasArxiv20}, for any marked surface $\widehat{S}$ and for each ideal triangulation $\mathcal{T}$ of $\widehat{S}$, we defined a bijection of sets 
\begin{equation*}
\Phi_{\mathcal{T}}:\mathcal{W}_{\widehat{S}}\overset{\sim}{\to} \mathcal{C}_\mathcal{T}
\end{equation*}
from the set $\mathcal{W}_{\widehat{S}}$ of parallel equivalence classes of reduced webs to the KTGS cone $\mathcal{C}_\mathcal{T} \subset \mathbb{Z}_+^N$, called the `web tropical coordinate map'.  We now recall the definition of this map.  

\subsubsection{Split ideal triangulations, good positions, and web schematics}
\label{sssec:splitidealtriangulationsandgoodposition}
	  
\begin{definition}
The \emph{split ideal triangulation} associated to $\mathcal{T}$, which by abuse of notation we also denote by $\mathcal{T}$, is defined by splitting each ideal edge of $\mathcal{T}$ (including boundary edges) into two disjoint ideal edges. In particular, the surface $\widehat{S}$ is cut into ideal triangles and \emph{bigons}, as shown in Figure \ref{figure:tr}.  Note that although bigons do not admit ideal triangulations (in particular, they do not satisfy the hypothesis $\chi < d/2$ of Section \ref{ssec:markedsurfacesidealtriangulations} since $d=2$), we can still consider them as marked surfaces, where all the definitions for webs make sense.

As proved in \cite{FrohmanMathZ2022} and \cite[Section $9.2$]{DouglasArxiv20}, by isotopy we can put any reduced web $W \in \mathcal{W}_{\widehat{S}}$ into \emph{good position} with respect to the split ideal triangulation $\mathcal{T}$, meaning (see just below for more details):
\begin{enumerate}[label=\textnormal{(\Roman*)}]
\item the restriction of the web $W$ to any bigon of $\mathcal{T}$ is a \emph{ladder web} (see the left hand side of Figure \ref{figure:bigon});
\item the restriction of the web $W$ to any triangle of $\mathcal{T}$ is a \emph{honeycomb web}, namely an oriented honeycomb  together with oriented corner arcs (see the left hand side of Figure \ref{figure:localtr}).  
\end{enumerate}
\end{definition}

More precisely, the triangle condition (called `rung-less essential' in \cite[Section $4.4$]{DouglasArxiv20}) is equivalent to saying that the restriction of $W$ to the triangle is reduced.  The bigon condition (called `essential' in \cite[Section $4.4$]{DouglasArxiv20}) is equivalent to asking that (1) all internal faces have at least six sides; and (2) for each edge $E$ of the bigon, and for every compact embedded arc $\alpha$ in the bigon such that $\partial \alpha = \alpha \cap E$ and such that $\alpha$ intersects $W$ generically, we have that the number of intersection points $W \cap \overline{E}$ does not exceed the number of intersection points $W \cap \alpha$; here, $\overline{E} \subset E$ is the segment in $E$ between the two endpoints of $\alpha$.  Note this is a weaker condition than $W|_\text{bigon}$ being reduced, since, although it does not allow for external 2- or 3-faces, it does allow for external H-4-faces (also called  `rungs' of the ladder web).

In particular,  $W$ has minimal geometric intersection with the split ideal triangulation $\mathcal{T}$.  

Note that for a web $W$ in good position: there are two types of honeycombs in triangles, `out-' and `in-honeycombs' (see Figure \ref{figure:localtr}); there may or may not be a honeycomb in a given triangle; and,  no conditions on the orientations of the corner arcs in a triangle are assumed.  

\begin{remark}
For an earlier appearance of these honeycomb webs in ideal triangles, see \cite[pp. 140-141]{KuperbergCommMathPhys96}.  
\end{remark}

In the right hand side of Figure \ref{figure:bigon} we show the `bigon schematic diagram' for a ladder web in a bigon, where each `H' is replaced by a crossing.  

In the right hand side of Figure \ref{figure:localtr} we show the `triangle schematic diagram' for a honeycomb web in a triangle.  Here, the honeycomb component is completely determined by two pieces of information:  its orientation (either all in or all out) and a nonnegative integer $x \in \mathbb{Z}_+$.  Note that the schematic for corner arcs is not a `faithful' diagrammatic representation, in general, because it forgets the ordering of the oriented arcs on each corner; see Remark \ref{rem:cornerarcschematics}.  However, as we will see, this schematic is sufficient to recover the web tropical coordinates.  

\begin{remark}
\label{rem:cornerarcschematics}
Note that the schematic is indeed faithful at the level of parallel equivalence classes of reduced webs in the ideal triangle.  This is because permuting corner arcs preserves the equivalence class of the web; see Section \ref{ssec:reducedwebsandtheKTGScone}.  (Recall also Figure \ref{fig:boundary-parallel-move}, showing a boundary parallel move in the ideal square.)
\end{remark}

\begin{definition}
Given the split ideal triangulation as in Figure \ref{figure:adm}, suppose we are given two oriented arcs intersecting  in the bigon along the ideal edge between the two triangles. The intersection is called a:
\begin{enumerate}[label=\textnormal{(\Roman*)}]
\item \emph{non-admissible crossing} if the arcs go toward a common ideal triangle (left hand side of Figure \ref{figure:adm});
\item \emph{admissible crossing} if the arcs go toward different ideal triangles (right hand side of Figure \ref{figure:adm}).
\end{enumerate} 
\end{definition}

The following fact is essentially by definition.  

\begin{observation}
\label{lemma:bigon}
For any reduced web $W$ in good position with respect to the split ideal triangulation $\mathcal{T}$, the schematic diagram \textup{(}right hand side of Figure {\upshape \ref{figure:bigon}}\textup{)} of any ladder web obtained by restricting $W$ to a bigon has only admissible crossings.  \qed
\end{observation}

\begin{figure}[t]
\includegraphics[scale=0.35]{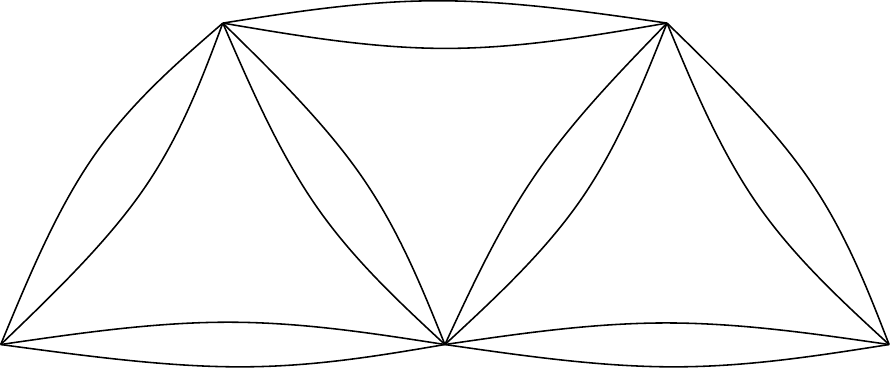}
\caption{     Split ideal triangulation.}
\label{figure:tr}
\end{figure}

\begin{figure}[t]
\includegraphics[scale=0.5]{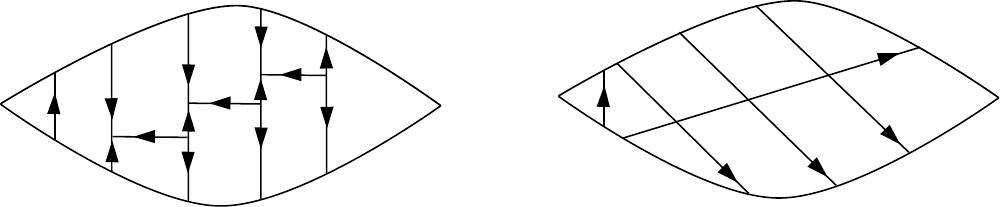}
\caption{      Ladder web in  a bigon.}
\label{figure:bigon}
\end{figure}

\begin{figure}[t]
\includegraphics[scale=0.4]{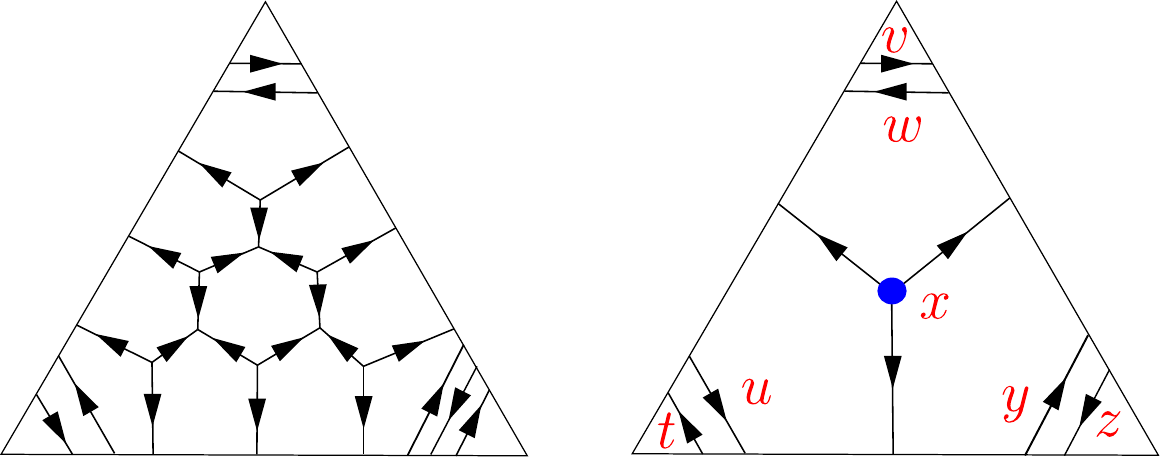}
\caption{     Honeycomb web in a triangle:  $x=3$, $y=2$, and $z=t=u=v=w=1$.  Here the honeycomb is oriented outward (there may also be inward oriented honeycombs). }
\label{figure:localtr}
\end{figure}

\begin{figure}[t]
\includegraphics[scale=0.51]{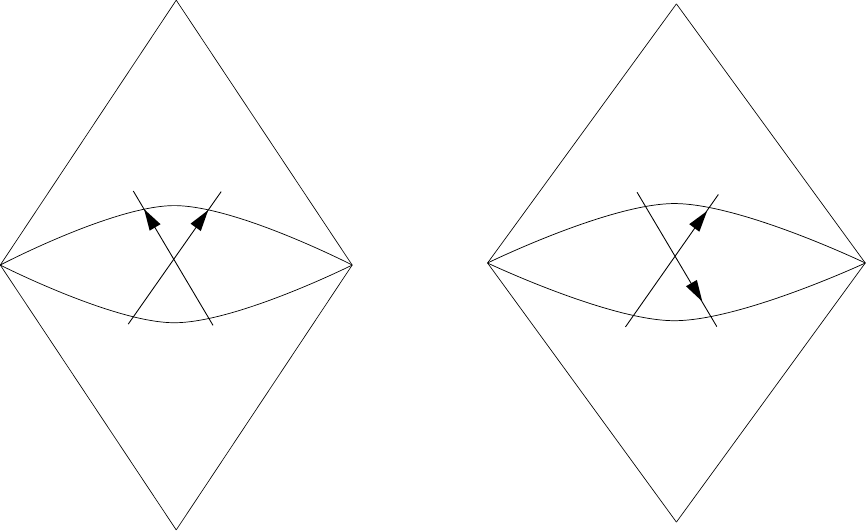}
\caption{     Left: Non-admissible crossing. Right: Admissible crossing.}
\label{figure:adm}
\end{figure}

\subsubsection{Definition of the web tropical coordinates}
\label{sssec:definitionofthetropicalwebcoordinates}
	  
Another way to think of an ideal triangle $\Delta$ is as an ideal polygon with three marked points $(a,b,c)$ on its boundary, labeled counterclockwise say.  (An ideal square is an ideal polygon with four marked points on its boundary, and so on.)  

Let a reduced web $W$ be in good position with respect to a split ideal triangulation $\mathcal{T}$ of $\widehat{S}$.  We start by defining the web tropical coordinates $\Phi(W|_\Delta) \in \mathcal{C}_\Delta$ `locally' for each restriction $W|_\Delta$ of $W$ to an ideal triangle $\Delta$ of $\mathcal{T}$, as in the left hand side of Figure \ref{figure:localtr}.  

First, the images in $\mathcal{C}_\Delta \subset \mathbb{Z}_+^7$ under $\Phi$ of the eight `irreducible' (see Section \ref{section:Hsq})  local reduced webs $R_a,L_a,R_b,L_b,R_c,L_c, T_{in}, T_{out}$ displayed in Figure \ref{figure:triangle} are defined as in that figure.  One checks directly that these images satisfy the Knutson-Tao rhombus inequalities and the modulo 3 congruence conditions (Section \ref{ssec:reducedwebsandtheKTGScone}).  

Then, the image under $\Phi$ of the restriction $W|_\Delta$ is defined as follows.  Let $T \in \{ T_{in}, T_{out} \}$ be the oriented honeycomb appearing in $W|_\Delta$.  Let the nonnegative integers $(x,w,v,u,t,y,z) \in  \mathbb{Z}_{+}^7$ be defined by the schematic for $W|_\Delta$, as in the right hand side of Figure \ref{figure:localtr}.  Put
\begin{equation*}
\Phi(W|_\Delta) := x \Phi(T) + v \Phi(L_a) + w \Phi(R_a) + t \Phi(L_b) + u \Phi(R_b) + z \Phi(L_c) + y \Phi(R_c)   \in \mathcal{C}_\Delta \subset \mathbb{Z}_+^7.  
\end{equation*}

Lastly, the web tropical coordinates $\Phi_\mathcal{T}(W) \in \mathcal{C}_\mathcal{T} \subset \mathbb{Z}_+^N$ for $W$ are defined by `gluing together' the local coordinates $\Phi(W|_\Delta$) for the triangles $\Delta$ across the edges of $\mathcal{T}$.  Note that the pair of coordinates of $\Phi(W|_\Delta)$ along an edge $E$ at the bigon interface between two triangles $\Delta$ and $\Delta^\prime$ matches the corresponding pair of coordinates of $\Phi(W|_{\Delta^\prime})$ along the other bigon edge $E^\prime$, since these coordinates depend only on the number of oriented in- and out-strands crossing the bigon at either boundary edge $E$ or $E^\prime$. Thus, this gluing procedure is well-defined.  In particular, in this way coordinates are assigned to the un-split ideal triangulation $\mathcal{T}$; this is why, in practice, one can go back and forth between the split and un-split triangulation.

See Figure \ref{fig:coordinates-example} for an example where $\widehat{S}$ is the once punctured torus.  As another example, the face coordinate (namely, the coordinate that is 3 for $T_{in}$ and $T_{out}$) for the honeycomb web $W$ shown in the left hand side of Figure \ref{figure:localtr} is $3\times3+4\times1+3\times2=19$.  There are plenty of examples of computing web coordinates throughout the paper; for instance, see Appendix \ref{ssec:flip-examples}.

In \cite[Section $9.2$]{DouglasArxiv20} we showed $\Phi_\mathcal{T}(W) \in \mathcal{C}_\mathcal{T}$ is independent of the choice of good position of $W$ with respect to the split ideal triangulation $\mathcal{T}$.  Moreover, we proved the result mentioned at the beginning of this subsection:  

\begin{theorem}[{\cite[Theorem 9.1]{DouglasArxiv20}}]
\label{thm:main-theorem}
For each ideal triangulation $\mathcal{T}$ of the marked surface $\widehat{S}$, the web tropical coordinate map
\begin{equation*}
\Phi_{\mathcal{T}}:\mathcal{W}_{\widehat{S}}\overset{\sim}{\to}\mathcal{C}_{\mathcal{T}}
\end{equation*}
from the set $\mathcal{W}_{\widehat{S}}$ of parallel equivalence classes of reduced webs in $\widehat{S}$ to the Knutson-Tao-Goncharov-Shen cone $\mathcal{C}_\mathcal{T} \subset \mathbb{Z}_+^N$ is a bijection of sets.  
\end{theorem}

We will need the following fact, which is immediate from the definitions.  
\begin{observation}
\label{lemma:ww}
For any disjoint reduced webs  $W,W'\in \mathcal{W}_{\widehat{S}}$, we have $W \cup W^\prime \in \mathcal{W}_{\widehat{S}}$ and
\begin{equation*}\Phi_{\mathcal{T}}(W\cup W')=\Phi_{\mathcal{T}}(W)+\Phi_{\mathcal{T}}(W')  \in \mathcal{C}_\mathcal{T}.\qed\end{equation*}
\end{observation}

\begin{figure}[t]
\includegraphics[scale=0.47]{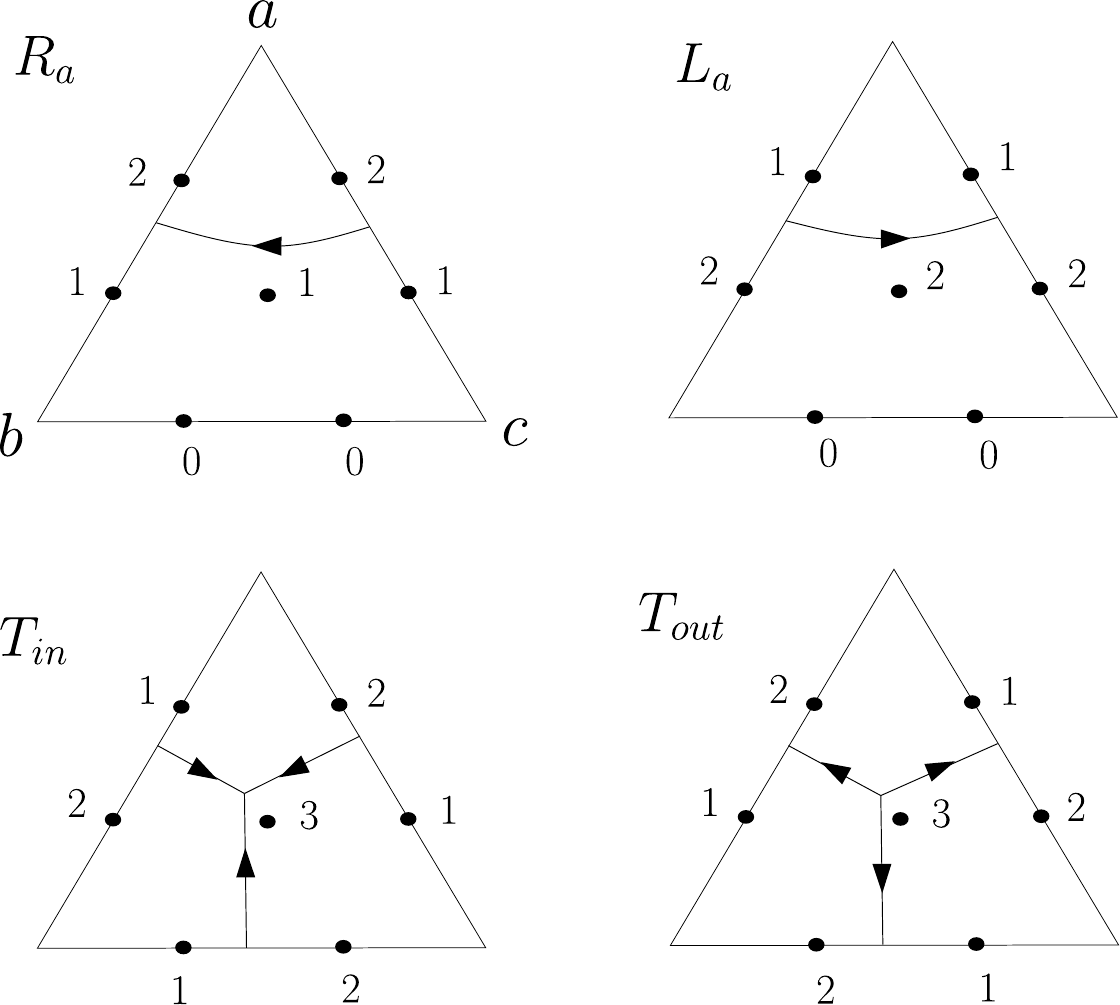}
\caption{     Tropical web coordinates for the eight `irreducible' reduced webs in the triangle.  (The coordinates for the other four arcs $R_b$, $L_b$, $R_c$, $L_c$ are obtained by triangular symmetry.)}
\label{figure:triangle}
\end{figure}

\begin{figure}[t]
\includegraphics[width=\textwidth]{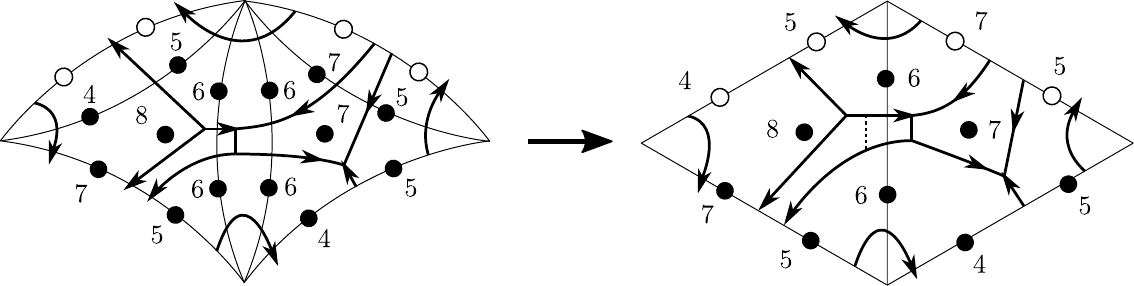}
\caption{     Gluing construction for the tropical coordinates for a reduced web in the once punctured torus.}
\label{fig:coordinates-example}
\end{figure}

\section{Naturality of the web tropical coordinates}
\label{section:fe}
	
In Section \ref{section:wa}, we recalled the construction \cite{DouglasArxiv20} of the web tropical coordinate map $\Phi_\mathcal{T} : \mathcal{W}_{\widehat{S}} \to \mathcal{C}_\mathcal{T} \subset \mathbb{Z}_+^N$, depending on a choice of ideal triangulation $\mathcal{T}$ of the marked surface $\widehat{S}$.  By Theorem \ref{thm:main-theorem}, $\Phi_\mathcal{T}$ is a bijection. 

In this section, we show that these coordinates are `natural' with respect to changing the triangulation $\mathcal{T} \to \mathcal{T}^\prime$.  That is, the induced coordinate change map $\mathcal{C}_\mathcal{T} \to \mathcal{C}_{\mathcal{T}^\prime}$ is a tropical $\mathcal{A}$-coordinate cluster transformation, in the language of Fock-Goncharov \cite{FockIHES06}.  

\begin{remark}
\label{rem:regarding-flip-proof}
See \cite{Shen1} for a more conceptual proof of the main result of this section, and paper,   Theorem \ref{thm:second-main-theorem}.  Moreover, in \cite{Shen1} the tropical web coordinates are further realized in a new and more topological way, via an algebraic intersection number between the webs and the dual web laminations.  The work of \cite{Shen1} more satisfactorily explains the generalization to the rank $2$ setting of Fock--Goncharov's theory of rank $1$ laminations.

The first version of the arXiv version \cite{DouglasArxiv20b} of this article proved Theorem \ref{thm:second-main-theorem} via a case-by-case analysis consisting of $42$ cases, $9$ of which were demonstrated and the remaining $33$ of which were variants of these $9$ cases and so were omitted; compare Section \ref{ssec:42-reduced-web-families-in-the-square}.  In the subsequent arXiv versions and present version of this article, we have replaced this case-by-case analysis with a more uniform and complete proof.  (In Appendix \ref{ssec:flip-examples}, as a concrete demonstration we still provide $3$ of the cases from our original proof of Theorem \ref{thm:second-main-theorem}.)
\end{remark}

\subsection{Naturality for the square}
\label{ssec:precisestatementofnaturality}
	  
\begin{definition}	
Recall that an \emph{ideal square} $\square$ is a disk with four marked points on its boundary.    An ideal triangulation of $\Box$ is a choice of diagonal of the square; there are two such triangulations, related by a \emph{diagonal flip}.  
\end{definition}
	
\begin{definition}
\label{def:mutationmap}
Let $\mathcal{T}$ and $\mathcal{T}^\prime$ be the two ideal triangulations of the square $\Box$, as in Figure \ref{figure:flip} ($\mathcal{T}$ on the left, $\mathcal{T}^\prime$ on the right).  The \emph{tropical $\mathcal{A}$-coordinate cluster transformation for the square} is the piecewise-linear function 
\begin{equation*}
\mu_{\mathcal{T}^\prime, \mathcal{T}} : \mathbb{Z}^{12} \to \mathbb{Z}^{12}
\end{equation*}
defined by
\begin{equation*}
\mu_{\mathcal{T}^\prime, \mathcal{T}}(x_1, x_2, x_3, x_4, x_5, x_6, x_7, x_8, y_1, y_2, y_3, y_4)
=
(x^\prime_1, x^\prime_2, x^\prime_3, x^\prime_4, x^\prime_5, x^\prime_6, x^\prime_7, x^\prime_8, z_1, z_2, z_3, z_4),
\end{equation*}
where the right hand side of the equation is given by   \eqref{eq:boundarycoords}, \eqref{equation:mu1}, \eqref{equation:mu2}, \eqref{equation:mu3}, \eqref{equation:mu4} from the introduction.  See also Figure \ref{figure:flip}.  (Here, we think of the domain of $\mu_{\mathcal{T}^\prime, \mathcal{T}}$ as associated to $\mathcal{T}$, and the codomain to $\mathcal{T}^\prime$.)  
\end{definition}	
	
Note  that  \eqref{equation:mu3}, \eqref{equation:mu4} use  \eqref{eq:boundarycoords}, \eqref{equation:mu1}, \eqref{equation:mu2}.
	
The main result of this paper is: 
	
\begin{theorem}
\label{thm:second-main-theorem}
Let $\mathcal{T}$ and $\mathcal{T}^\prime$ be the two ideal triangulations of the square $\Box$, and let $\Phi_\mathcal{T} : \mathcal{W}_\Box \to \mathcal{C}_\mathcal{T} \subset \mathbb{Z}_+^{12}$ and $\Phi_{\mathcal{T}^\prime} : \mathcal{W}_\Box \to \mathcal{C}_{\mathcal{T}^\prime} \subset \mathbb{Z}_+^{12}$ be the associated web tropical coordinate maps. Then,
\begin{equation*}
\mu_{\mathcal{T}^\prime, \mathcal{T}}(c) =\Phi_{\mathcal{T}^\prime} \circ \Phi_\mathcal{T}^{-1}(c)   \in  \mathcal{C}_{\mathcal{T}^\prime}  \,\,  ( c \in \mathcal{C}_\mathcal{T} ).
\end{equation*}
\end{theorem}

\begin{remark}
Note it is not even clear, a priori, from the definitions that $\mu_{\mathcal{T}^\prime, \mathcal{T}}(c) \geq 0$ for $c \in \mathcal{C}_\mathcal{T}$.  
\end{remark}

\subsection{Proof of Theorem \ref{thm:second-main-theorem}}
\label{ssec:proofofnaturality}
	   
By definition of the tropical coordinates, and of good position of a reduced web $W$ in $\mathcal{W}_\Box$ with respect to the triangulations $\mathcal{T}$ and $\mathcal{T}^\prime$, we immediately get:  

\begin{observation}
\label{obs:eq1forallwebs}
For all $W \in \mathcal{W}_\Box$,  the images $c=\Phi_\mathcal{T}(W) \in \mathcal{C}_\mathcal{T}$ and $\Phi_{\mathcal{T}^\prime} \circ \Phi_\mathcal{T}^{-1}(c) \in \mathcal{C}_{\mathcal{T}^\prime}$ satisfy  \eqref{eq:boundarycoords}.  \qed
\end{observation}

\begin{definition}
\label{def:corner-arcs}
Let the punctures of the square $\Box$ be labeled $a$, $b$, $c$, $d$ as in Figure \ref{figure:square1}.  Also as in the figure, define the 8 oriented \emph{corner arcs} $L_a, R_a, L_b, R_b, L_c, R_c, L_d, R_d$ in $\mathcal{W}_\Box$.  Their 12 coordinates are provided in the figure as well.  
\end{definition}

One checks by direct computation that:

\begin{observation}
\label{obs:thmforarcs}
The images  $c=\Phi_\mathcal{T}(W) \in \mathcal{C}_\mathcal{T}$, for $W=L_a, R_a, L_b, R_b, L_c, R_c, L_d, R_d$ any of the $8$ corner arcs, satisfy Theorem {\upshape\ref{thm:second-main-theorem}}.  \qed
\end{observation}

\begin{definition}
\label{def:cornerless-webs-in-the-square}
A given reduced web $W$ in $\mathcal{W}_\Box$ is the disjoint union of, first, all its corner arc components, together called the \emph{corner part} and denoted $W_r$; and, second, their complement $W_c = W - W_r$, which we call the \emph{cornerless part} of the web $W$.    

A reduced web $W$ is \emph{cornerless} if $W=W_c$.  That is, $W$ has no corner arcs.  

Let $\mathcal{R} \subset \mathcal{W}_\Box$ be the set of \emph{corner webs}, that is, webs whose cornerless parts are empty:  $W = W_r$.   That is, an element of $\mathcal{R}$ is a disjoint union of corner arcs.  
\end{definition}

\begin{lem}
\label{lemma:add}
For any disjoint reduced webs $W \in \mathcal{R}$ and $W'\in \mathcal{W}_\square$, we have $W \cup W^\prime \in \mathcal{W}_\Box$ and 
\begin{equation*}
\mu_{\mathcal{T}^\prime, \mathcal{T}}(\Phi_\mathcal{T}(W))+\mu_{\mathcal{T}^\prime, \mathcal{T}}(\Phi_\mathcal{T}(W'))=\mu_{\mathcal{T}^\prime, \mathcal{T}}(\Phi_\mathcal{T}(W\cup W'))    \in  \mathbb{Z}^{12}.
\end{equation*}
\end{lem}

\begin{proof}
By Observation \ref{lemma:ww}, we get
\begin{equation*}
\Phi_\mathcal{T}(W)+\Phi_\mathcal{T}(W')=\Phi_\mathcal{T}(W\cup W')    \in  \mathcal{C}_\mathcal{T}.
\end{equation*}
For any corner arc (Figure \ref{figure:square1}) thus for any $W\in \mathcal{R}$ (again by Observation \ref{lemma:ww}), the left hand sides of  \eqref{equation:mu1}, \eqref{equation:mu2}, \eqref{equation:mu3}, \eqref{equation:mu4} are always of the form $\max\{u,u\}-v$. Since 
\begin{equation*}
(\max\{u,u\}-v)+ (\max\{x,y\}-z)=\max\{u+x,u+y\}-(v+z)  \in \mathbb{Z}
\end{equation*}
we obtain the desired equality.  
\end{proof}

\begin{proof}[Proof of Theorem {\upshape \ref{thm:second-main-theorem}}]

Recall by Theorem \ref{thm:main-theorem} that any $c \in \mathcal{C}_\mathcal{T}$ is of the form $c=\Phi_\mathcal{T}(W)$ for some $W \in \mathcal{W}_\Box$. For any reduced web $W\in \mathcal{W}_\square$, suppose that its coordinates via $\Phi_\mathcal{T}$ are labeled as in the left hand side of Figure \ref{figure:flip},  and via $\Phi_{\mathcal{T}'}$ as in the right hand side of Figure \ref{figure:flip}.  By Observation \ref{obs:eq1forallwebs},  \eqref{eq:boundarycoords} is satisfied for any web $W$ in $\mathcal{W}_\Box$.  In addition,  by Observation \ref{obs:thmforarcs},  the    \eqref{equation:mu1}, \eqref{equation:mu2}, \eqref{equation:mu3}, \eqref{equation:mu4} are satisfied for any web $W$ in $\mathcal{R}$, that is, $W$ consisting only of corner arcs.  By Lemma   \ref{lemma:add} together with another application of Observation \ref{lemma:ww} to $\mathcal{T}^\prime$, we have thus reduced the problem to establishing  \eqref{equation:mu1}, \eqref{equation:mu2}, \eqref{equation:mu3}, \eqref{equation:mu4} for any cornerless web $W=W_c$. 

The main difficulty  is that, for a given cornerless web $W=W_c$ in good position with respect to the ideal triangulation $\mathcal{T}$, after flipping the diagonal $\mathcal{T} \to \mathcal{T}^\prime$ it is not obvious how $W$ `rearranges itself' back into good position with respect to the new triangulation $\mathcal{T}'$.  (See, however, Appendix \ref{ssec:flip-examples} for examples of this rearrangement into good position after the flip.)  

We circumvent this difficulty by solving the problem `uniformly', that is, without knowing how the new good position looks after the flip.  The hypothesis that the web $W=W_c$ does not have any corner arcs will be important here. 

To start, observe that it suffices to establish just  \eqref{equation:mu1}.  Indeed,  \eqref{equation:mu2}, \eqref{equation:mu3}, \eqref{equation:mu4} then immediately follow by 90 degree rotational symmetry.  (Solve for $y_1$ and $y_3$, respectively, in the last two equations.)

With this goal in mind, we argue
\begin{equation*}\tag{g}
\label{eq:main-lemma}
z_2 = x^\prime_2 + x^\prime_3 = x_2 + x_3 
=  \mathrm{max}(x_2 + y_3, x_3 + y_1) - y_2
  \in  \mathbb{Z}_+.
\end{equation*}

Throughout, consider Figure \ref{fig:flip-proof}, recalling the notion of a web schematic; see Section \ref{sssec:splitidealtriangulationsandgoodposition} and Remark \ref{rem:cornerarcschematics}.   

The second equation of \eqref{eq:main-lemma} has already been justified, by Observation \ref{obs:eq1forallwebs}.  

Let us justify the first equation of \eqref{eq:main-lemma}.  There are two cases, namely when $m^\prime$ represents an out- or an in-honeycomb.  

When $m^\prime$ is `out', we compute:
\begin{equation*}
x_2^\prime = b^\prime + 2 z^\prime + m^\prime, \,\,
x_3^\prime =  c^\prime + 2 y^\prime + 2 m^\prime, \,\,
z_2 = b^\prime + c^\prime + 2 y^\prime + 2 z^\prime + 3 m^\prime.
\end{equation*}
When $m^\prime$ is `in', we compute:
\begin{equation*}
x_2^\prime = b^\prime + 2 z^\prime + 2m^\prime, \,\,
x_3^\prime =  c^\prime + 2 y^\prime + m^\prime, \,\,
z_2 = b^\prime + c^\prime + 2 y^\prime + 2 z^\prime + 3 m^\prime.
\end{equation*}
In both cases, the desired formula $z_2 = x^\prime_2 + x^\prime_3$ holds.  

The justification of the third equation of \eqref{eq:main-lemma} is  more involved.  We begin with a topological consequence. 

\begin{claim}
\label{claim:firstmainclaiminproof}
Let $W = W_c$ be a cornerless reduced web in the square.  Up to $180$ degree rotational symmetry of the square, there are three cases.
\begin{enumerate}[label=\textnormal{(\Roman*)}]
\item  When the $n$ and $m$ honeycombs are both `out':  Then, 
\begin{equation*}
a + n + x = b + y \text{ and }  w+d = z + m +c.
\end{equation*}
Moreover, if $y \geq n+x$, then $b=0$; and, if $y \leq n+x$, then $a=0$.  
	 
\textup{(}Note this is the case displayed in the left hand side of  Figure {\upshape\ref{fig:flip-proof}}.\textup{)}	 
\item  When the $n$ honeycomb is `out', and the $m$ honeycomb is `in':  Then,  
\begin{equation*}
a + n + x = b + y + m \text{ and }  w+d = z  +c.
\end{equation*}
Moreover, if $y+m \geq n+x$, then $b=0$; and, if $y+m  \leq n+x$, then $a=0$.  	
\item  When the $n$ and $m$ honeycombs are both `in':  Then, 
\begin{equation*}
a +  x = b + y + m \text{ and }  w+d + n = z  +c.
\end{equation*}
Moreover, if $y +m \geq x$, then $b=0$; and, if $y+m \leq x$, then $a=0$.  
\end{enumerate}
\end{claim}

The key topological property used to prove all three statements of the claim is the following:  The number of `out' strands (resp. `in' strands) along one boundary edge of the bigon, as displayed on the left hand side of Figure \ref{fig:flip-proof}, is equal to the number of `in' strands (resp. `out' strands) along the other boundary edge of the bigon.  

We prove the first statement, (I), of the claim; the proofs of the second and third statements are similar.  So assume the $n$ and $m$ honeycombs are both `out'.  

By the above topological property, we have the desired two identities of the statement.  

For the second part of the statement:  When $y \geq n+x$, if $b$ were nonzero, then $a$ would have to be nonzero, since $b+y=a+n+x $.  Then $b$ would be attaching to $a$; see the schematic shown in the left hand side of Figure   \ref{fig:flip-proof-lemma} (see also  the caption of Figure \ref{fig:flip-proof}).   But this contradicts the hypothesis that $W$ has no corner arcs.  Similarly, $a=0$ when $y \leq n+x$; see  the right hand side of Figure   \ref{fig:flip-proof-lemma}.  This establishes the claim. 

\begin{claim}
\label{claim:secondmainclaiminproof}
Let $W = W_c$ be a cornerless reduced web in the square.  Up to $180$ degree rotational symmetry of the square, there are three cases.
\begin{enumerate}[label=\textnormal{(\Roman*)}]
\item  When the $n$ and $m$ honeycombs are both `out':  Then, $x_2 + y_3 \nabla x_3 + y_1$ if and only if $y \nabla n+x$, for $\nabla\in\{>,=,<\}$.

\textup{(}Note this is the case displayed in the left hand side of  Figure {\upshape\ref{fig:flip-proof}}.\textup{)}
\item  When the $n$ honeycomb is `out', and the $m$ honeycomb is `in':  Then, $x_2 + y_3 \nabla x_3 + y_1$ if and only if $y+m \nabla n+x$, for $\nabla\in\{>,=,<\}$.
\item  When the $n$ and $m$ honeycombs are both `in':  Then, $x_2 + y_3 \nabla x_3 + y_1$ if and only if $y+m \nabla x$, for $\nabla\in\{>,=,<\}$.  
\end{enumerate}
\end{claim}

We prove the first statement; the proofs of the second and third statements are similar.  So assume the $n$ and $m$ honeycombs are both `out'.  By Figure \ref{fig:flip-proof}, we compute:
\begin{gather*}
x_2 = w + 2a +n,  \,\,    
y_3 = b + c + 2y +2z + 3m,    \\
x_3 =  z + 2b + 2m,    \,\,
y_1 =  a + d + 2x + 2w + 3n.
\end{gather*}
Thus,
\begin{gather*}
(x_2 +y_3) - (x_3 + y_1) = -w +a -2n - b + c + 2y + z + m -d -2x \nabla 0
\\	\Leftrightarrow  a + c + 2y + z + m \nabla w + 2n + b + d +2x.
\end{gather*}

By applying the two identities of the first part of Claim \ref{claim:firstmainclaiminproof}, the above inequality is equivalent to $3y \nabla 3n + 3x$ as desired.  This establishes the claim.  

We are now prepared to justify the third equation of \eqref{eq:main-lemma}, which we recall is
\begin{equation*}\tag{h}
\label{eq:maincalculation}
x_2 + x_3 = \mathrm{max}(x_2 + y_3, x_3 + y_1) - y_2.
\end{equation*}

First, let us assume the $n$ and $m$ honeycombs are both `out', as in the  left hand side of  Figure \ref{fig:flip-proof}.   The values of $x_2, y_3, x_3, y_1$ were computed above, and we gather
\begin{gather*}
x_2 + x_3 = w + 2a + n + z + 2b + 2m,     \\
x_2 + y_3 = w + 2a + n + b + c + 2y + 2z + 3m,     \\
x_3 + y_1 = z + 2b + 2m + a + d + 2x + 2w + 3n.
\end{gather*}
By Figure \ref{fig:flip-proof}, there are two ways to express $y_2$:
\begin{equation*}
y_2 = w + d + 2a + 2n + 2x = z + m + c + 2b + 2y.
\end{equation*}

There are two cases to establish \eqref{eq:maincalculation}.  In the case $x_2 + y_3 \geq x_3 + y_1$, we compute, using the second form of $y_2$ above:
\begin{gather*}
\mathrm{max}(x_2 + y_3, x_3 + y_1) - y_2 = (x_2 + y_3) - y_2     \\
=     w + 2a + n - b + z + 2m    \overset{?}{=} x_2 + x_3    
\Leftrightarrow   b \overset{?}{=} 0.  
\end{gather*}
For this case, by the first part of Claim \ref{claim:secondmainclaiminproof}, we have $y \geq n+x$.  Thus, $b=0$ by the first part of Claim \ref{claim:firstmainclaiminproof}, as desired.  

In the case $x_2 + y_3 \leq x_3 + y_1$, we compute, using the first form of $y_2$ above:
\begin{gather*}
\mathrm{max}(x_2 + y_3, x_3 + y_1) - y_2 = (x_3 + y_1) - y_2     \\
=    z + 2b + 2m - a + w + n    
\overset{?}{=} x_2 + x_3    
\Leftrightarrow   a \overset{?}{=} 0.  
\end{gather*}
For this case, by the first part of Claim \ref{claim:secondmainclaiminproof}, we have $y \leq n+x$.  Thus, $a=0$ by the first part of Claim \ref{claim:firstmainclaiminproof}, as desired.  

This establishes \eqref{eq:maincalculation} when both the honeycombs are `out'.  When the $n$ honeycomb is `out', and the $m$ honeycomb is `in'; or, when the $n$ and $m$ honeycombs are both `in':  By essentially the same calculation, one computes again that, in the case $x_2 +y_3 \geq x_3 + y_1$, then \eqref{eq:maincalculation} is equivalent to $b=0$, and in the case $x_2 + y_3 \leq x_3 + y_1$, then \eqref{eq:maincalculation} is equivalent to $a=0$.  These are justified by parts (II) and (III), respectively, of Claims \ref{claim:secondmainclaiminproof} and \ref{claim:firstmainclaiminproof}.  

This completes the proof of the main result.  
\end{proof}

\begin{figure}[t]
\includegraphics[scale=.6]{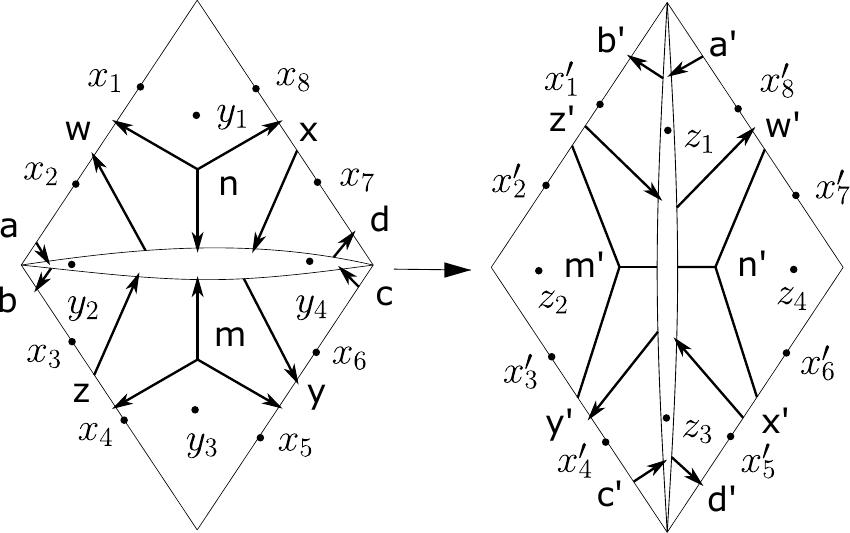}
\caption{     Schematic for the cornerless web $W=W_c$ in the square, before and after the flip.    The variables $a, b, c, d, x, y, z, w, n, m$ are known, and can be read off from the good position of $W$ with respect to $\mathcal{T}$.  The primed variables $a^\prime, b^\prime, \dots, m^\prime$ are not assumed to be known.  Because $W$ has no corner arcs, there are no arcs at the top and bottom vertices before the flip, nor at the left and right vertices after the flip; it follows by Observation \ref{lemma:bigon} that we  cannot have  $a$ and $b$ (or $c$ and $d$)  simultaneously nonzero.    To be concrete, we have shown the case where the honeycombs labeled $n$ and $m$ are out-honeycombs; we will justify the other cases as well.  Note that the orientations of the $n^\prime$ and $m^\prime$ honeycombs are not assumed to be known (and do not follow from the orientations of the $n$ and $m$ honeycombs).  }
\label{fig:flip-proof}
\end{figure}

\begin{figure}[t]
\includegraphics[scale=.59]{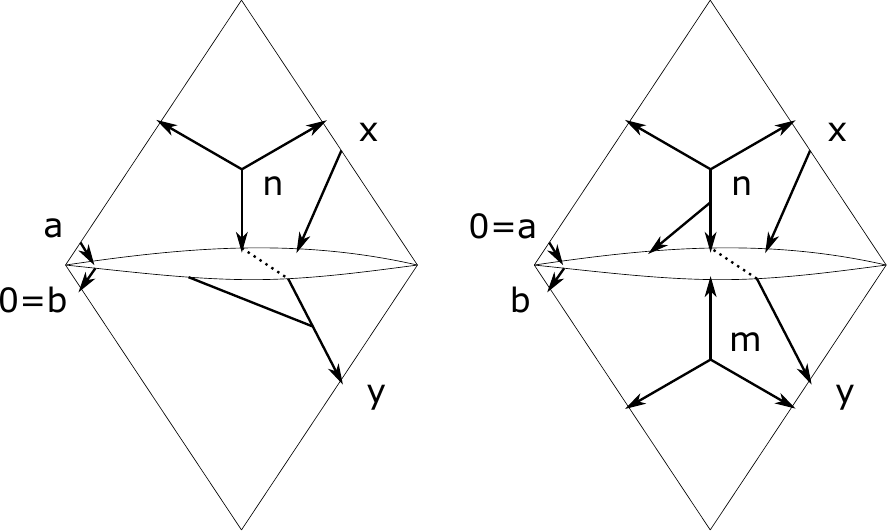}
\caption{     Proof of Claim \ref{claim:firstmainclaiminproof}.   The web is assumed not to have any corner arcs.  Shown is the case when the $n$ and $m$ honeycombs are both `out'.  Left:  $y \geq n + x$.  Right:  $y \leq n + x$.}
\label{fig:flip-proof-lemma}
\end{figure}

\subsection{Naturality for a marked surface}

We now discuss the natural generalization of Theorem \ref{thm:second-main-theorem} to any marked surface $\widehat{S}$, according to  the standard cluster theory \cite{FockIHES06, FominJAmerMathSoc02}.  

Let $\mathcal{T}$ be an ideal triangulation of $\widehat{S}$, and let $N$ denote the number of global tropical coordinates; see Section \ref{ssec:markedsurfacesidealtriangulations}.  In Section \ref{section:wa}, we introduced the web tropical coordinate map $\Phi_\mathcal{T} : \mathcal{W}_{\widehat{S}} \to \mathcal{C}_\mathcal{T}$, where we implicitly chose an inclusion $\mathcal{C}_\mathcal{T} \subset \mathbb{Z}_+^N$ of the KTGS cone of $\mathcal{T}$ (permutations of the coordinates of $\mathbb{Z}^N$ determine different inclusions).  This choice played essentially no role there, since we were only considering a single triangulation.  As we are now changing the triangulation, it becomes necessary to keep track of this choice.  

	\subsubsection{Dotted triangulations}
	\label{sec:base-triangulations}
	$  $  
		
	More precisely, let $\mathcal{T} = \mathcal{T}_0$ be an initial choice of ideal triangulation, which we mark with $N$ dots in the usual way (as for instance in Figure \ref{figure:acoor}).  Such a dotted initial triangulation $\mathcal{T} = \mathcal{T}_0$ is called a `base (dotted) triangulation'.  Fix a labeling of the dots of $\mathcal{T}$ from $1, 2, \dots, N$; we say that the base triangulation $\mathcal{T} = \mathcal{T}_0$ is `labeled'.  This determines a bijection $\left\{ 1, 2, \dots, N \right\} \to \left\{ \textnormal{dots} \right\}$, hence also an inclusion $\mathcal{C}_\mathcal{T} \subset \mathbb{Z}_+^N$ (since a point in $\mathcal{C}_\mathcal{T}$ is, most precisely, a function $\left\{ \textnormal{dots} \right\} \to \mathbb{Z}_+$; see Figure \ref{fig:coordinates-example-alt} for instance).  

 \begin{example*}[part 1]
 See the first diagram 0 in Figure \ref{fig:sl3-pentagon-relation}.  
 \end{example*}
	
	We now imagine forgetting $\mathcal{T}=\mathcal{T}_0$, but leaving the dots associated to $\mathcal{T}$ where they were on the surface.  Another triangulation $\mathcal{T}^\prime$ is `dotted (with respect to the dotting of $\mathcal{T}$)' if there are two dots (from $\mathcal{T}$) on each edge of $\mathcal{T}^\prime$ and one dot (from $\mathcal{T}$) in each face of $\mathcal{T}^\prime$.  Note that, using the labeling of the dots fixed along with the base triangulation $\mathcal{T}=\mathcal{T}_0$, each dotted triangulation $\mathcal{T}^\prime$ determines an inclusion $\mathcal{C}_{\mathcal{T}^\prime} \subset \mathbb{Z}_+^N$ just like for $\mathcal{T}=\mathcal{T}_0$.  
	
	We refer to $\mathcal{T}^\prime$ without a dotting as the underlying `topological triangulation'.  
	
	\begin{example*}[part 2] 
The triangulations $\mathcal{T}_1$, $\mathcal{T}_2$, $\mathcal{T}_5$, $\mathcal{T}_{35}$ shown in the last four  diagrams 1, 2, 5, 35 in Figure \ref{fig:sl3-pentagon-relation} are dotted with respect to the dotted triangulation $\mathcal{T}_0$ in the first diagram  0.  Note that, as topological triangulations, $\mathcal{T}_0 = \mathcal{T}_5 = \mathcal{T}_{35}$, but, as dotted triangulations, $\mathcal{T}_0 = \mathcal{T}_{35} \neq \mathcal{T}_5$.  
\end{example*}

	\begin{figure}[t]
	\includegraphics[scale=.66]{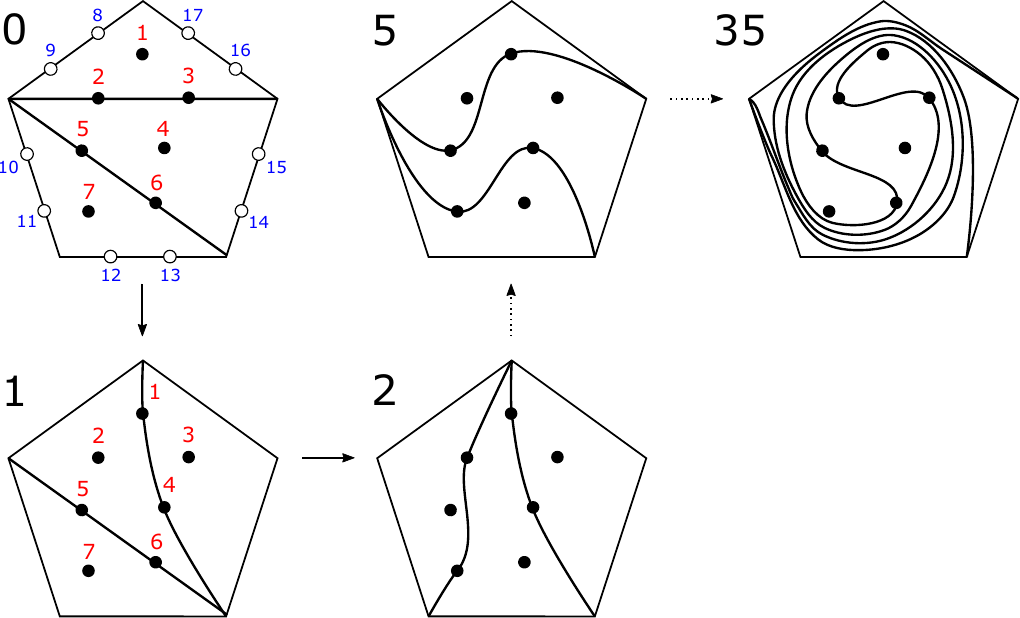}
	\caption{Pentagon relation for $\mathrm{SL}_3$, namely the sequence of five diagonal flips from $\mathcal{T}_0$ to $\mathcal{T}_5$.  It takes 35  flips to come back to the original dotted triangulation.  That is, as dotted triangulations, $\mathcal{T}_0=\mathcal{T}_{35}$ and $\mathcal{T}_0 \neq \mathcal{T}_i$ for $i=1,2,\dots,34$.  The ten boundary coordinates (shown only in the first picture) are fixed, or frozen, by these flip mutations but still enter into the computations for the interior coordinates.  See Remark \ref{rem:seven-identities} for a computation.}
	\label{fig:sl3-pentagon-relation}
\end{figure}

	\subsubsection{Tropical $\mathcal{A}$-coordinate cluster transformation for a marked surface}
	\label{def:tropical-A-coordinate-cluster-transformation}
	$ $  
	
	Pausing the discussion of the topology and geometry of webs and cones for a moment, let us discuss a purely algebraic result of Fomin--Zelevinsky and Fock--Goncharov (see Theorem   \ref{prop:mutation-naturality} below).  
	
	There is a  procedure to start from a base triangulation $\mathcal{T}=\mathcal{T}_0$ and generate new dotted triangulations in a controlled way, via diagonal flips.  Indeed, if $\mathcal{T}_i$ is dotted, and if $(\Box, \mathcal{T}_i |_\Box)$ is a triangulated square in $\mathcal{T}_i$, the diagonal flip at $\Box$ induces a new dotted triangulation $\mathcal{T}_{i+1}$ (passing `in between' the two dots of the diagonal of $(\Box, \mathcal{T}_i |_\Box)$).  Note that the induced embedding of $\mathcal{T}_{i+1}$ in $\widehat{S}$ is well-defined up to isotopy in $\widehat{S} - \left\{ \textnormal{dots} \right\}$ (since the triangulated square $(\Box, \mathcal{T}_i |_\Box)$ comes with a canonical, up to isotopy, foliation by arcs starting and ending at the two boundary points of its diagonal, and including this diagonal as one of these arcs).  (For simplicity, we always assume that $\mathcal{T}_{i+1}$ is not self-folded; see \S \ref{ssec:markedsurfacesidealtriangulations}.)  
	
	Observe that the resulting sequence $\mathcal{T}=\mathcal{T}_0 \to \mathcal{T}_1 \to \mathcal{T}_2 \to \cdots$ of dotted triangulations depends on the chosen sequence of diagonal flips, so is not predetermined.   
	
	 \begin{example*}[part 3]
See for instance the passage from  the diagrams 0 to 1 or from the diagrams 1 to 2 in Figure \ref{fig:sl3-pentagon-relation}.  
\end{example*}
 
	Let $\mathcal{T}=\mathcal{T}_0$ be labeled as well, and let $\mathcal{T}_i$ and $\mathcal{T}_{i+1}$ be as above.  Then their dottings induce a `flip mutation function' 
\begin{equation*}
	\mu_{\mathcal{T}_{i+1}, \mathcal{T}_i} : \mathbb{Z}^N \rightarrow \mathbb{Z}^N
\end{equation*} 
defined in the square $\Box$ as in Definition \ref{def:mutationmap} (that is, by the formulas  \eqref{eq:boundarycoords}, \eqref{equation:mu1}, \eqref{equation:mu2}, \eqref{equation:mu3}, \eqref{equation:mu4} of Figure \ref{figure:flip}) and outside the square as the identity.  (We remind that the domain is associated to $\mathcal{T}_i$ and the codomain to $\mathcal{T}_{i+1}$.)  (In particular, the formulas are the same irrespective of whether the boundary of the square has any self-gluings.)  If 
\begin{equation*}
\label{eq:flipseq1}
\tag{$\ast$}
	\mathcal{T}=\mathcal{T}_0 \to \mathcal{T}_1 \to \mathcal{T}_2 \to \cdots \to \mathcal{T}_{m-1} \to \mathcal{T}_m=\mathcal{T}^\prime
\end{equation*} is a sequence of flips as above, ending at a dotted triangulation $\mathcal{T}_m=\mathcal{T}^\prime$, define the associated `mutation function (for this  sequence of flips)'
\begin{equation*}
	\mu_{\mathcal{T}_m, \mathcal{T}_0} : \mathbb{Z}^N \rightarrow \mathbb{Z}^N
\end{equation*}
as the function composition
\begin{equation*}
	\mu_{\mathcal{T}_m, \mathcal{T}_0} := \mu_{\mathcal{T}_m, \mathcal{T}_{m-1}} \circ \cdots \circ \mu_{\mathcal{T}_2, \mathcal{T}_1} \circ \mu_{\mathcal{T}_1, \mathcal{T}_0}.  
\end{equation*}
(Here, we have used the usual convention that $(g \circ f)(x):=g(f(x))$.)

	Now, let    in addition
	\begin{equation*}
	\label{eq:flipseq2}
\tag{$\ast\ast$}
	\mathcal{T} = \mathcal{T}_0 \to \widetilde{\mathcal{T}}_1 \to \widetilde{\mathcal{T}}_2 \to \cdots \to \widetilde{\mathcal{T}}_{\widetilde{m}-1} \to \widetilde{\mathcal{T}}_{\widetilde{m}} = \mathcal{T}^\prime
	\end{equation*}
	be another sequence of diagonal flips starting at the same base triangulation $\mathcal{T}=\mathcal{T}_0$, and ending at the same topological triangulation $\mathcal{T}^\prime$, but where the dottings of $\mathcal{T}_m = \mathcal{T}^\prime$ and $\widetilde{\mathcal{T}}_{\widetilde{m}} = \mathcal{T}^\prime$ are possibly different.  Define the associated permutation linear map
	\begin{equation*}
		\sigma_{\widetilde{\mathcal{T}}_{\widetilde{m}}, \mathcal{T}_m} : \mathbb{Z}^N \rightarrow \mathbb{Z}^N
	\end{equation*}
	as follows:  For $i\in\left\{1,2,\dots, N\right\}$, if the $i$-labeled dot is on an edge (resp. face) of $\mathcal{T}_m =\mathcal{T}^\prime$, and if the corresponding dot on the corresponding edge (resp. face) of $\widetilde{\mathcal{T}}_{\widetilde{m}} =\mathcal{T}^\prime$ is labeled $\sigma(i)$, then $\sigma_{\widetilde{\mathcal{T}}_{\widetilde{m}}, \mathcal{T}_m}$ maps the $i$-th standard basis vector $e_i$ of $\mathbb{Z}^N$ to the $\sigma(i)$-th standard basis vector $e_{\sigma(i)}$ of $\mathbb{Z}^N$. 
	
	\begin{example*}[part 4] 
 We could take $\mathcal{T}_i = \widetilde{\mathcal{T}}_i$ in Figure \ref{fig:sl3-pentagon-relation} with $m=5$ and $\widetilde{m}=35$.  Then,
\begin{equation*}
	\sigma(1)=3, \quad \sigma(2)=1, \quad \sigma(3)=4, \quad \sigma(4)=6, \quad \sigma(5)=2, \quad \sigma(6)=7, \quad \sigma(7)=5
\end{equation*}
and $\sigma(i)=i$ is the identity on all of the boundary coordinates ($i=8,9,\dots,17$).  
\end{example*}

\begin{theorem}[{\cite{FockIHES06, FominJAmerMathSoc02}}]
\label{prop:mutation-naturality}
Let $\mathcal{T}=\mathcal{T}_0$ be a labeled base dotted triangulation.  Given two flip sequences \eqref{eq:flipseq1} and \eqref{eq:flipseq2}    of dotted triangulations ending at the same topological triangulation $\mathcal{T}^\prime$, the following diagram commutes:
\begin{equation*}
  \begin{tikzcd}
    \mathbb{Z}^N \arrow{r}{
    	\mu_{\mathcal{T}_m, \mathcal{T}_0}
    } \arrow[swap]{dr}{
    		\mu_{\widetilde{\mathcal{T}}_{\widetilde{m}}, \mathcal{T}_0}
	} & \mathbb{Z}^N \arrow{d}{
    		\sigma_{\widetilde{\mathcal{T}}_{\widetilde{m}}, \mathcal{T}_m}
	} \\
     & \mathbb{Z}^N
  \end{tikzcd}
\end{equation*}  

\end{theorem}

An immediate consequence of this theorem is:

	\begin{observation}
	\label{obs:interesting-consequence}
Using the notation of  Theorem {\upshape\ref{prop:mutation-naturality}}, \eqref{eq:flipseq1} and \eqref{eq:flipseq2}:  For $\mathcal{T}^\prime=\mathcal{T}$ and $m=0$, we get $\mu_{\mathcal{T}_m, \mathcal{T}_0}=\mu_{\mathcal{T}_0, \mathcal{T}_0}=\mathrm{identity}$, hence 
\begin{equation*}
	\mu_{\widetilde{\mathcal{T}}_{\widetilde{m}}, \mathcal{T}_0}=\sigma_{\widetilde{\mathcal{T}}_{\widetilde{m}}, \mathcal{T}_0}
	\quad : \mathbb{Z}^N \rightarrow \mathbb{Z}^N
	\end{equation*} 
	is a permutation map.  
\end{observation}

\begin{example*}[part 5]
 To get a feel for the content of Theorem \ref{prop:mutation-naturality}, let us consider (similar to part 4 of this example) $\mathcal{T}_i = \widetilde{\mathcal{T}}_i$ in Figure \ref{fig:sl3-pentagon-relation} with $m=0$ and $\widetilde{m}=5$, where $\sigma_{\mathcal{T}_5, \mathcal{T}_0}$ is defined by
\begin{equation*}
	\sigma(1)=2, \quad \sigma(2)=5, \quad \sigma(3)=1, \quad \sigma(4)=3, \quad \sigma(5)=7, \quad \sigma(6)=4, \quad \sigma(7)=6
\end{equation*}
and $\sigma(i)=i$ for $i=8,9,\dots,17$.  

Then Observation \ref{obs:interesting-consequence} says $\sigma_{\mathcal{T}_5, \mathcal{T}_0}^{-1} \circ \mu_{\mathcal{T}_5, \mathcal{T}_0}$ is the identity map $\mathbb{Z}^{17} \to \mathbb{Z}^{17}$.  This is the `pentagon relation' for $\mathrm{SL}_3$.  By construction this is obvious for the 8-th through the 17-th coordinates, because these are the `frozen' coordinates on the boundary of the pentagon.  So the meat of the statement is the validity of the seven identities:
\begin{equation*}
	f_i(x_1, x_2, \dots, x_{17})=x_i
	\quad\quad  \left( i=1, 2, \dots, 7; \quad x_j \in \mathbb{Z}; \quad j=1,2,\dots,17 \right)
\end{equation*}
where $f_i : \mathbb{Z}^{17} \to \mathbb{Z}$ is defined as the $i$-th component of $\sigma_{\mathcal{T}_5, \mathcal{T}_0}^{-1} \circ \mu_{\mathcal{T}_5, \mathcal{T}_0}$.  In particular, $f_i$ is a complicated piecewise-linear function built out of the operations $+$, $-$, and $\mathrm{max}$.  

	 \end{example*}

\begin{remark}
\label{rem:seven-identities}
Of the seven identities $f_i(x_1, x_2, \dots, x_{17})=x_i$ discussed in Example (part 5), consider the case $i=5$.  	Appendix \ref{app:first-appendix} at the end of this article contains a Mathematica notebook which provides the explicit expression for $f_5(x_1, x_2, \dots, x_{17})$.  

\end{remark}

Wrapping up this digression, a well-known fact \cite{penner1987decorated} says that any two triangulations $\mathcal{T}$ and $\mathcal{T}^\prime$ are related by a sequence of diagonal flips \eqref{eq:flipseq1}.  By Theorem \ref{prop:mutation-naturality}, we immediately obtain:  

\begin{cor}
\label{cor:cluster-transformation}
Let $\mathcal{T}=\mathcal{T}_0$ be a labeled base dotted triangulation.  For any topological triangulation  $\mathcal{T}^\prime$, there is a function 
\begin{equation*}
	\mu_{\mathcal{T}^\prime, \mathcal{T}_0} : \mathbb{Z}^N \rightarrow \mathbb{Z}^N, 
\end{equation*}
defined only up to permutation of the coordinates of the codomain $\mathbb{Z}^N$, satisfying the property that if $\mathcal{T}_m=\mathcal{T}^\prime$ is related to $\mathcal{T} =\mathcal{T}_0$ by a sequence of diagonal flips \eqref{eq:flipseq1}, then
\begin{equation*}
	\mu_{\mathcal{T}^\prime, \mathcal{T}_0} = \mu_{\mathcal{T}_m, \mathcal{T}_{m-1}} \circ \cdots \circ \mu_{\mathcal{T}_2, \mathcal{T}_1} \circ \mu_{\mathcal{T}_1, \mathcal{T}_0}
\end{equation*}
where the $\mu_{\mathcal{T}_{i+1}, \mathcal{T}_i} : \mathbb{Z}^N \to \mathbb{Z}^N$ are the corresponding \textup{(}well-defined\textup{)} flip mutation functions.  \qed
\end{cor}

\begin{definition}
\label{def:cluster-transformation-for-surface}
	The (pseudo-)function $\mu_{\mathcal{T}^\prime, \mathcal{T}_0} : \mathbb{Z}^N \to \mathbb{Z}^N$ from Corollary \ref{cor:cluster-transformation} is called the \emph{tropical $\mathcal{A}$-coordinate cluster transformation} for the marked surface $\widehat{S}$ associated to the labeled base dotted triangulation $\mathcal{T}=\mathcal{T}_0$ and the topological triangulation $\mathcal{T}^\prime$.  Below, we will drop the subscript and just write $\mu_{\mathcal{T}^\prime, \mathcal{T}}$ for this function.
\end{definition}

\begin{remark}
\label{rem:FG-remark-Q}
$ $
\begin{enumerate}[label=\textnormal{(\Roman*)}]
	\item  \label{rem1:FG-remark-Q}  Throughout this sub-subsection, there has been nothing special about the integers $\mathbb{Z}$ compared to, say, the rational numbers $\mathbb{Q}$.  In particular, Theorem \ref{prop:mutation-naturality} makes sense and is true with $\mathbb{Z}$ replaced by $\mathbb{Q}$.  
	
	\item  The pentagon relation for $\mathrm{SL}_3$ (Figure \ref{fig:sl3-pentagon-relation}), equivalent to the seven identities $f_i(x_1, x_2, \dots, x_{17})=x_i$ for $i=1,2,\dots,7$ discussed in Example (part 5) above is   the main relation required to establish Theorem \ref{prop:mutation-naturality}.  
	
\end{enumerate}

\end{remark}

		\subsubsection{Naturality of the web tropical coordinate map}
		\label{sssec:naturality-of-the-web-tropical-coordinate-map}
		$ $ 
		
We are now ready to generalize Theorem \ref{thm:second-main-theorem} to any marked surface $\widehat{S}$.  

Let $\mathcal{T}=\mathcal{T}_0$ be a labeled base dotted triangulation, and let $\mathcal{T}^\prime$ be any topological triangulation; see \S \ref{sec:base-triangulations}.  Associated to this topological data is the tropical $\mathcal{A}$-coordinate cluster transformation $\mu_{\mathcal{T}^\prime, \mathcal{T}} : \mathbb{Z}^N \to \mathbb{Z}^N$, which is only defined up to permutation of the coordinates of the codomain $\mathbb{Z}^N$; see Definition \ref{def:cluster-transformation-for-surface}.  

Lastly, recall from \S \ref{sec:base-triangulations} that the labels for the dots of $\mathcal{T}=\mathcal{T}_0$ determine an inclusion $\mathcal{C}_{\mathcal{T}} \subset \mathbb{Z}_+^N$ of its KTGS cone.  Since $\mathcal{T}^\prime$ is not assumed to be dotted, the inclusion  $\mathcal{C}_{\mathcal{T}^\prime} \subset \mathbb{Z}_+^N$ of its KTGS cone is only defined up to permutation of the coordinates of $\mathbb{Z}_+^N$.   

\begin{cor}
\label{cor:second-main-theorem}
	Let $\mathcal{T}$ and $\mathcal{T}^\prime$ be triangulations, and let $\mu_{\mathcal{T}^\prime, \mathcal{T}} : \mathbb{Z}^N \to \mathbb{Z}^N$ be the corresponding tropical $\mathcal{A}$-coordinate cluster transformation, as just explained.  For the associated web tropical coordinate maps $\Phi_\mathcal{T} : \mathcal{W}_{\widehat{S}} \to \mathcal{C}_\mathcal{T} \subset \mathbb{Z}_+^{N}$ and $\Phi_{\mathcal{T}^\prime} : \mathcal{W}_{\widehat{S}} \to \mathcal{C}_{\mathcal{T}^\prime} \subset \mathbb{Z}_+^{N}$, we have
\begin{equation*}
	\mu_{\mathcal{T}^\prime, \mathcal{T}}(c) =\Phi_{\mathcal{T}^\prime} \circ \Phi_\mathcal{T}^{-1}(c)  \quad  \in  \mathcal{C}_{\mathcal{T}^\prime}  \quad\quad  \left( c \in \mathcal{C}_\mathcal{T} \right),
\end{equation*}
where this equality is only defined up to permutation of the coordinates of $\mathcal{C}_{\mathcal{T}^\prime} \subset \mathbb{Z}_+^N$.  
\end{cor}

\begin{proof}
	This is essentially an immediate consequence of Theorem \ref{thm:second-main-theorem}.  
	
	Indeed, consider a flip sequence \eqref{eq:flipseq1} of dotted triangulations $\mathcal{T}_i$, namely such that $\mathcal{T}_{i+1}$ is related to $\mathcal{T}_i$ by a single diagonal flip.  Recall (\S \ref{sec:base-triangulations}) that the dotting of $\mathcal{T}_i$ induces an inclusion $\mathcal{C}_{\mathcal{T}_i} \subset \mathbb{Z}_+^N$ of its KTGS cone, and recall (\S \ref{def:tropical-A-coordinate-cluster-transformation}) that $\mu_{\mathcal{T}_{i+1}, \mathcal{T}_i} : \mathbb{Z}^N \to \mathbb{Z}^N$ is the corresponding (well-defined) flip mutation function.  
	
	By Theorem \ref{thm:second-main-theorem}, for each  $i=0,1,...,m-1$ we have
\begin{equation*}
	\mu_{\mathcal{T}_{i+1}, \mathcal{T}_i}(c_i) =\Phi_{\mathcal{T}_{i+1}} \circ \Phi_{\mathcal{T}_i}^{-1}(c_i)   \quad \in  \mathcal{C}_{\mathcal{T}_{i+1}}
	\quad\quad  \left( c_i \in \mathcal{C}_{\mathcal{T}_i} \right).  
\end{equation*}
By iterating  this equation, for any $c_0=c \in \mathcal{C}_\mathcal{T}=\mathcal{C}_{\mathcal{T}_0}$ we obtain
\begin{align*}
	\mu_{\mathcal{T}_m, \mathcal{T}_0}(c)
&=	\mu_{\mathcal{T}_m, \mathcal{T}_{m-1}} \circ \cdots 
\circ \mu_{\mathcal{T}_3, \mathcal{T}_2}
\circ \mu_{\mathcal{T}_2, \mathcal{T}_1} \circ \mu_{\mathcal{T}_1, \mathcal{T}_0}(c_0)
\\&=	\mu_{\mathcal{T}_m, \mathcal{T}_{m-1}} \circ \cdots 
\circ \mu_{\mathcal{T}_3, \mathcal{T}_2}
\circ \mu_{\mathcal{T}_2, \mathcal{T}_1}\left(  \Phi_{\mathcal{T}_{1}} \circ \Phi_{\mathcal{T}_0}^{-1}(c_0)  \right)
\\&=	\mu_{\mathcal{T}_m, \mathcal{T}_{m-1}} \circ \cdots 
\circ \mu_{\mathcal{T}_3, \mathcal{T}_2}
\left( \Phi_{\mathcal{T}_{2}} \circ \Phi_{\mathcal{T}_1}^{-1}
\left(  \Phi_{\mathcal{T}_{1}} \circ \Phi_{\mathcal{T}_0}^{-1}(c_0)  \right)
\right)
\\&=\cdots  
\\&=  \mu_{\mathcal{T}_m, \mathcal{T}_{m-1}}
\left(  \Phi_{\mathcal{T}_{m-1}} \circ \Phi_{\mathcal{T}_0}^{-1}(c_0)  \right)
\\&=  \Phi_{\mathcal{T}_m} \circ \Phi_{\mathcal{T}_0}^{-1}(c)
\quad  \in  \mathcal{C}_{\mathcal{T}_m}.  
\end{align*}
The result follows from the defining property of the function $\mu_{\mathcal{T}^\prime, \mathcal{T}}$ (Corollary~\ref{cor:cluster-transformation} and Definition~\ref{def:cluster-transformation-for-surface}).  
\end{proof}

\begin{conceptremark}
\label{rem:equivariant}
Another way to express Corollary \ref{cor:second-main-theorem} is to say that the web tropical coordinates, determined by the maps $\{ \Phi_\mathcal{T} \}_\mathcal{T}$, are equivariant with respect to the action of the mapping class group of the marked surface $\widehat{S}$.  Said another way, they form natural coordinates for the positive tropical integer $\mathrm{PGL}_3$-points $\mathcal{A}^+_{\mathrm{PGL}_3, \widehat{S}}(\mathbb{Z}^t)$, where a point in $\mathcal{A}^+_{\mathrm{PGL}_3, \widehat{S}}(\mathbb{Z}^t)$ is thought of concretely as a reduced web $W$ in $\mathcal{W}_{\hat{S}}$.  
\end{conceptremark}

\begin{application}  
\label{app:first-application}  
Generalizing Fock-Goncharov's (bounded) $\mathrm{SL}_2$-laminations \cite[Section $12$]{FockIHES06}, Kim \cite{KimArxiv20} considers the space $\widetilde{\mathcal{W}}_{\widehat{S}}$ of `(bounded) $\mathrm{SL}_3$-laminations' (he denotes this space by $\mathcal{A}_L(\widehat{S}, \mathbb{Z})$), which extends the space $\mathcal{W}_{\widehat{S}}$ of reduced webs by allowing for negative integer weights around the peripheral loops and arcs.  He also extends the web tropical coordinate map $\Phi_\mathcal{T} : \mathcal{W}_{\widehat{S}} \to \mathcal{C}_\mathcal{T} \subset \mathbb{Z}_+^N$ of Theorem \ref{thm:main-theorem} to an injective map $\widetilde{\Phi}_\mathcal{T} : \widetilde{\mathcal{W}}_{\widehat{S}} \to \mathbb{Z}^N$, and characterizes the image as an integer lattice defined by certain `balancedness' conditions; it turns out that these conditions are equivalent to the modulo 3 congruence conditions of Definition \ref{def:KTGS-cone}.  That is, whereas the reduced webs $\mathcal{W}_{\widehat{S}}$ correspond to solutions of both the modulo 3 congruence conditions and the Knutson-Tao inequalities, the $\mathrm{SL}_3$-laminations $\widetilde{\mathcal{W}}_{\widehat{S}} \supset \mathcal{W}_{\widehat{S}}$ correspond to solutions of only the modulo 3 congruence conditions.  By \cite[Proposition 3.35]{KimArxiv20}, which generalizes Corollary \ref{cor:second-main-theorem}, the lamination tropical coordinates $\{ \widetilde{\Phi}_\mathcal{T} \}_\mathcal{T}$ are also natural, thereby constituting an explicit model for the tropical integer $\mathrm{PGL}_3$-points $\mathcal{A}_{\mathrm{PGL}_3, \widehat{S}}(\mathbb{Z}^t)$; compare Remark \ref{rem:equivariant} and see also Remark \ref{rem:realsolutionsareinteger}\eqref{rem2:realsolutionsareinteger}.  

Kim's proof of \cite[Proposition 3.35]{KimArxiv20} uses Corollary   \ref{cor:second-main-theorem}.  One way to think about upgrading the naturality statement from webs to laminations is in terms of  the proof strategy of Theorem   \ref{thm:second-main-theorem}; see Section \ref{ssec:proofofnaturality}.  Indeed, since Lemma \ref{lemma:add} works as well for corner arcs with integer coefficients, the proof of Theorem   \ref{thm:second-main-theorem} works more generally for the laminations $\widetilde{\mathcal{W}}_{\widehat{S}}$.  (See also \cite{KimArxiv21}.)  
\end{application}

\begin{application}
The same strategy used in the proof of Corollary \ref{cor:second-main-theorem} provides an alternative, geometric topological proof of Theorem \ref{prop:mutation-naturality} (see also Remark \ref{rem:FG-remark-Q}\eqref{rem1:FG-remark-Q}), but only valid on the restricted cone domain $\mathcal{C}_{\mathcal{T}} \subset \mathbb{Z}_+^N \subset \mathbb{Q}_+^N$ (or, by applying Kim's result for $\mathrm{SL}_3$-laminations $\widetilde{\mathcal{W}}_{\widehat{S}}$ from Application \ref{app:first-application}, on the restricted lattice domain $\widetilde{\Phi}_{\mathcal{T}}(\widetilde{W}_{\widehat{S}}) \subset \mathbb{Z}^N \subset \mathbb{Q}^N$).  
\end{application}

\section{KTGS cone for the square:  Hilbert basis}
\label{section:Hsq}

In the remaining two sections, we study the structure of the Knutson-Tao-Goncharov-Shen cone $\mathcal{C}_\mathcal{T} \subset \mathbb{Z}_+^N$ associated to an ideal triangulation $\mathcal{T}$ of a marked surface $\widehat{S}$ (Definition \ref{def:KTGS-cone}  and Proposition \ref{prop:KTGS-is-positive-integer-cone}) when $\widehat{S}=\Box$ is an ideal square.  In this case, an ideal triangulation $\mathcal{T}$ is simply a choice of a diagonal of $\Box$.   (In this section, we will use some of the terminology and results of Appendix \ref{sec:cones}.)

\subsection{Hilbert basis of the KTGS cone for the triangle and the square}
\label{ssec:hilbert-basis-of-the-ktgs-cone-for-the-square}
	    
\subsubsection{Hilbert basis for the triangle}
\label{sssec:hilbert-basis-for-the-triangle}
	   
We begin by recalling from \cite[Section $6$]{DouglasArxiv20} the case of a single ideal triangle $\widehat{S}=\mathcal{T}=\Delta$.  Let $\mathcal{C}_\Delta \subset \mathbb{Z}_+^7$ be the corresponding KTGS positive integer cone.  
	
Recall the eight `irreducible' webs $L_a, R_a, L_b, R_b, L_c, R_c, T_{in}, T_{out}$ in $\mathcal{W}_\Delta$ defined in Section \ref{sssec:definitionofthetropicalwebcoordinates}.  For each such web $W^H$, its 7 tropical coordinates $\Phi_\Delta(W^H) \in \mathcal{C}_\Delta \subset \mathbb{Z}_+^7$ are provided in Figure \ref{figure:triangle}.    

\begin{prop}
\label{prop:hilb-base-triangle}
The $8$-element subset
\begin{equation*}
\mathcal{H}_\Delta=\{ \Phi_\Delta(W^H); 
  W^H=L_a, R_a, L_b, R_b, L_c, R_c, T_{in}, T_{out}  \}
  \subset    \mathcal{C}_\Delta
\end{equation*}
is the Hilbert basis \textup{(}Definition {\upshape\ref{def:hilbert-basis}}\textup{)} of the KTGS cone $\mathcal{C}_\Delta \subset \mathbb{Z}_+^7$ for the triangle.  
\end{prop}
	
\begin{proof}
This is a consequence of \cite[Proposition 6.6]{DouglasArxiv20} and its proof.  		

Indeed, we need to show that $\mathcal{H}_\Delta$ is the set of irreducible elements.  To start, any such $\Phi_\Delta(W^H)$ is nonzero.    By the last sentence of \cite[Proposition 6.6]{DouglasArxiv20}, we have that $\mathcal{H}_\Delta$ is a $\mathbb{Z}_+$-spanning set for $\mathcal{C}_\Delta$.  One checks by hand that no single element of $\mathcal{H}_\Delta$ can be written as a $\mathbb{Z}_+$-linear combination of other elements of $\mathcal{H}_\Delta$.  (This last property is particularly clear when viewed in the isomorphic cone $\mathcal{C} \subset \mathbb{Z}_+^6 \times \mathbb{Z} \subset \mathbb{Z}^7$, namely the image of $\mathcal{C}_\Delta$ under a certain linear isomorphism $\mathbb{R}^7 \to \mathbb{R}^7$ of geometric origin; for details, see the proof of \cite[Proposition 6.6]{DouglasArxiv20}.  In \S \ref{section:ssq}, we generalize this linear isomorphism to the square case.)  The result follows by Proposition \ref{prop:irreducible-elements}.
\end{proof}
	
\begin{remark}
\label{rem:word-of-caution}
As a caution, it is not implied that an element of $\mathcal{C}_\Delta$ has a unique decomposition as a sum of  Hilbert basis elements.  Indeed, in $\mathcal{C}_\Delta$, we have the relation $\Phi_\Delta(T_{in})+\Phi_\Delta(T_{out})=\Phi_\Delta(L_a)+\Phi_\Delta(L_b)+\Phi_\Delta(L_c)$.  See also Section \ref{ssec:relations-in-the-ktgs-cone-for-the-square}.

It is also not true that if $\Phi_\Delta(W^\prime) \leq \Phi_\Delta(W) \in \mathcal{C}_\Delta \subset \mathbb{Z}_+^7$, in the sense that the inequality holds for each coordinate, then $W^\prime$ is topologically `contained in' $W$.  Indeed, in the above example, we have $\Phi_\Delta(T_{in})$ or $\Phi_\Delta(T_{out}) \leq  \Phi_\Delta(L_a)+\Phi_\Delta(L_b)+\Phi_\Delta(L_c)$ in $\mathcal{C}_\Delta$.    An even simpler example is $\Phi_\Delta(L_a) \leq \Phi_\Delta(R_b)+\Phi_\Delta(R_c)$.
\end{remark}

\subsubsection{Hilbert basis for the square}
\label{sssec:hilbert-basis-for-the-square-statement}
	   
We turn to the square $\Box$, which for the rest of this section is  equipped with an ideal triangulation $\mathcal{T}$, namely a choice of diagonal of $\Box$.  

Recall  the 8 oriented corner arcs in the square $\Box$ (Definition \ref{def:corner-arcs}); these are the `irreducible' webs (1)-(8) in  $\mathcal{W}_\Box$ depicted in Figure \ref{figure:square1}.  The triangulation $\mathcal{T}$ determines 14 more `irreducible' webs in $\Box$, namely the webs (9)-(22) in $\mathcal{W}_\Box$ depicted in Figure \ref{figure:square2}.  The bracket notation used in Figure \ref{figure:square2} is explained in the caption of the figure.  In sum, let us denote these 22 `irreducible' webs by $W^H_i \in \mathcal{W}_\Box$ for $i=1,2,\dots,22$.  

Let $\Phi_\mathcal{T} : \mathcal{W}_\Box \to \mathcal{C}_\mathcal{T}$ be the associated web tropical coordinate map.  For each web $W_i^H$, its 12 tropical coordinates $\Phi_\mathcal{T}(W_i^H) \in \mathcal{C}_\mathcal{T} \subset \mathbb{Z}_+^{12}$ are also provided in Figure \ref{figure:square2}.    

\begin{theorem}
\label{theorem:basis}
For the  webs $\{ W_i^H \}_{i=1,2,\dots,22}$ in $\mathcal{W}_\Box$ displayed in Figures {\upshape\ref{figure:square1}} and {\upshape\ref{figure:square2}},  the subset
\begin{equation*}
\mathcal{H}_{(\Box, \mathcal{T})}=\{ \Phi_\mathcal{T}(W_i^H) ; 
  i=1,2,\dots,22  \}
 \subset \mathcal{C}_\mathcal{T}
\end{equation*}
is the Hilbert basis of the KTGS cone $\mathcal{C}_\mathcal{T} \subset \mathbb{Z}_+^{12}$ for the triangulated square $(\Box, \mathcal{T})$.  
\end{theorem}
	
\begin{remark}
\label{rem:hilbert-basis-depends-on-triangulation}
Note that if the other triangulation $\mathcal{T}^\prime$ of $\Box$ had been chosen, then only the webs $W^H_1, \dots, W^H_{16}$ and $W^H_{19}, \dots, W^H_{22}$ would appear among the 22 `irreducible' webs $W^{\prime H}_i$ corresponding to $\mathcal{T}^\prime$.  In other words, the set of webs corresponding to the Hilbert basis $\mathcal{H}_{(\Box, \mathcal{T})}$ of $\mathcal{C}_\mathcal{T}$ depends on which triangulation $\mathcal{T}$ of the square is chosen.
\end{remark}

We will need a little bit of preparation before proving the theorem.  

Let $\Delta$ and $\Delta^\prime$ be the two triangles appearing in the split triangulation $\mathcal{T}$ of $\Box$ (Section \ref{sssec:splitidealtriangulationsandgoodposition}).   Say, $\Delta$ is the top triangle on the left hand side of Figure \ref{figure:flip}, and $\Delta^\prime$ is the bottom triangle.  In particular, neither $\Delta$ nor $\Delta^\prime$ include the intermediate bigon.  If $W$ is a reduced web in $\Box$ in good position with respect to the split ideal triangulation $\mathcal{T}$, then the restrictions $W|_\Delta$ and $W|_{\Delta^\prime}$ are in good position in their respective triangles (by definition of good position of $W$ with respect to $\mathcal{T}$).  At the level of coordinates, this induces two projections $\pi_\Delta : \mathcal{C}_\mathcal{T} \to \mathcal{C}_\Delta$ and $\pi_{\Delta^\prime} : \mathcal{C}_{\mathcal{T}} \to \mathcal{C}_{\Delta^\prime}$ defined by $\pi_\Delta(\Phi_\mathcal{T}(W))=\Phi_\Delta(W|_\Delta)$ and $\pi_{\Delta^\prime}(\Phi_\mathcal{T}(W))=\Phi_{\Delta^\prime}(W|_{\Delta^\prime})$.  Compare Figure \ref{fig:coordinates-example}.  

\begin{lem}
\label{proposition:h}
For a reduced web $W$ in $\mathcal{W}_\Box$, suppose its image $\Phi_{\mathcal{T}}(W)$ is an  irreducible element of $\mathcal{C}_\mathcal{T}$.  Then, the projections $\pi_\Delta(\Phi_{\mathcal{T}}(W))$ and $\pi_{\Delta^\prime}(\Phi_{\mathcal{T}}(W))$ are, respectively, in the Hilbert bases $\mathcal{H}_\Delta$ and $\mathcal{H}_{\Delta^\prime}$ of the cones $\mathcal{C}_\Delta$ and $\mathcal{C}_{\Delta^\prime}$.  

Consequently, the set of irreducible elements of $\mathcal{C}_\mathcal{T}$ is finite \textup{(}thus forming a Hilbert basis\textup{)} and is a subset of $\mathcal{H}_{(\Box, \mathcal{T})}$, as defined in Theorem {\upshape\ref{theorem:basis}}.  \textup{(}This is because $\mathcal{H}_{(\Box, \mathcal{T})}$ is formed by taking all possible gluings across the bigon of irreducible elements in the two triangles.\textup{)}
\end{lem}

\begin{proof}
Assuming the first statement, the second statement immediately follows by Definition \ref{def:hilbert-basis}, Proposition \ref{prop:hilb-base-triangle}, and the construction of the 22 element set $\mathcal{H}_{(\Box, \mathcal{T})} \subset \mathcal{C}_\mathcal{T}$.  

To establish the first statement, assume $W$ is in good position with respect to $\mathcal{T}$.  It suffices to show that if $\pi_\Delta(\Phi_{\mathcal{T}}(W)) = \Phi_\Delta(W|_\Delta)\in \mathcal{C}_\Delta$ is reducible, then $\Phi_\mathcal{T}(W) \in \mathcal{C}_\mathcal{T}$ is reducible.  So assume that there are nonempty reduced webs $A_1$ and $A_2$ in $\mathcal{W}_\Delta$ such that $\Phi_\Delta(W|_\Delta)=\Phi_\Delta(A_1)+\Phi_\Delta(A_2)$ in $\mathcal{C}_\Delta$.  (At this point, one should be mindful of Remark \ref{rem:word-of-caution}.)   We explicitly construct nonempty reduced webs $W_1$ and $W_2$ in $\mathcal{W}_\Box$ such that 
\begin{equation*}\tag{i}
\label{eq:contradiction-hilbert-basis}
\Phi_\mathcal{T}(W)=\Phi_\mathcal{T}(W_1)+\Phi_\mathcal{T}(W_2)   \in \mathcal{C}_\mathcal{T}.
\end{equation*} 

Let $E$ (resp. $E^\prime$) denote the bigon edge intersecting $\Delta$ (resp. $\Delta^\prime$).   Let $n$ and $m$ (resp. $n_i$ and $m_i$ for $i=1,2$) be, respectively, the number of out- and in-strand-ends of $W|_\Delta$ (resp. $A_i$) on $E$; similarly, let $n^\prime$ and $m^\prime$ be, respectively, the number of out- and in-strand-ends of $W|_{\Delta^\prime}$ on $E^\prime$.  Note $n^\prime = m$ and $m^\prime = n$.  

By \cite[Definition 5.1, property 2]{DouglasArxiv20}, which says that the two edge coordinates on $E$ uniquely determine the number of out- and in-strand-ends on $E$ (this is a simple linear algebra calculation), we must have $n=n_1+n_2$ and $m=m_1+m_2$.  We gather $n^\prime=m_1+m_2$ and $m^\prime=n_1+n_2$.  

Now, recall from Section \ref{sssec:splitidealtriangulationsandgoodposition} that a reduced web in a triangle consists of a honeycomb (possibly empty) together with corner arcs (possibly none).  For each $i=1,2$, arbitrarily choose $m_i$ out-strand-ends and $n_i$ in-strand-ends of $W|_{\Delta^\prime}$ on $E^\prime$, which we call `$i$-strand-ends' of $W|_{\Delta^\prime}$.  Let us say that a component $C^\prime$ of $W|_{\Delta^\prime}$ is `$A_i$-connecting' if at least one of its strand-ends on $E^\prime$ is an $i$-strand-end; note that (1) a corner arc $C^\prime$ is  $A_i$-connecting for at most one $i$ (possibly none, when $C^\prime$ is on the corner opposite $E^\prime$), and (2)  a honeycomb $C^\prime$ is $A_i$-connecting for at least one $i$, and may be both $A_1$- and $A_2$-connecting.  

Let $h^\prime \in \mathbb{Z}_+$ be the size of the honeycomb $H^\prime$ of $W|_{\Delta^\prime}$, and let $h^{\prime (i)}$ be the number of $i$-strand-ends of $H^\prime$; note that $h^\prime=h^{\prime (1)}+h^{\prime (2)}$.  For each $i=1,2$, define $A^\prime_i$ to be the reduced web in $\mathcal{W}_{\Delta^\prime}$ consisting of the $A_i$-connecting corner arc components $C^\prime$ of $W|_{\Delta^\prime}$  together with a honeycomb of size $h^{\prime (i)}$ oriented as $H^\prime$ (and we  can include the non-$A_i$-connecting components $C^\prime$ into $A^\prime_1$, say); note in particular that  $\Phi_{\Delta^\prime}(W|_{\Delta^\prime})=\Phi_{\Delta^\prime}(A^\prime_1)+\Phi_{\Delta^\prime}(A^\prime_2) \in \mathcal{C}_{\Delta^\prime}$.  

Lastly, for each $i=1,2$, define $W_i$ in $\mathcal{W}_\Box$ to be the unique nonempty reduced web in the square obtained from the triangle webs $A_i \in \mathcal{W}_\Delta$ and $A^\prime_i \in \mathcal{W}_{\Delta^\prime}$ by gluing across the bigon in the usual way (as in Figure \ref{figure:bigon}).  (Technically, it is the class of $W_i$ in $\mathcal{W}_\Box$ that is unique, and $W_i$ is determined up to corner arc permutations).  By construction, \eqref{eq:contradiction-hilbert-basis} holds.
\end{proof}

\begin{proof}[Proof of Theorem   {\upshape \ref{theorem:basis}}] 
By the last paragraph of Lemma \ref{proposition:h}, it remains to show that each element of $\mathcal{H}_{(\Box, \mathcal{T})}$ is irreducible in $\mathcal{C}_\mathcal{T}$.  This property can be checked by hand. (The irreducibility becomes clearer in light of the linear map $\theta_\mathcal{T} : \mathbb{R}^{12} \to \mathbb{R}^{18}$ of Section \ref{sssec:a-linear-isomorphism}, where the image $\theta_\mathcal{T}(\mathcal{H}_{(\Box, \mathcal{T})}) \subset \mathbb{Z}_+^{18}$ is written explicitly).  
\end{proof}

\begin{remark}
\label{rem:commentsabouthilbertbasisutility}
It follows by Proposition \ref{prop:irreducible-elements} that the Hilbert basis of the KTGS cone $\mathcal{C}_\mathcal{T} \subset \mathbb{Z}_+^{12}$ for the triangulated square $(\Box, \mathcal{T})$ appearing in Theorem \ref{theorem:basis} spans $\mathcal{C}_\mathcal{T}$ over $\mathbb{Z}_+$.  Actually, we will prove this finite generation of $\mathcal{C}_\mathcal{T}$ directly in Section \ref{section:ssq} as a consequence of Theorem \ref{thm:main-theorem}; see the proof of Theorem \ref{theorem:decomp} in Section \ref{sssec:proof-of-main-theorem-2}.  Strictly speaking then, Theorem \ref{theorem:basis} is not required at the technical level in what follows; the computations throughout this section will be used, however.

Note, in particular, that for these reasons the positivity of the KTGS cone $\mathcal{C}_\mathcal{T} \subset \mathbb{Z}_+^{12}$, while possibly conceptually interesting, plays a complementary role; see also Remark \ref{rem:spanningpropertiesofsubmonoids}.
\end{remark}

\begin{figure}[t]
\includegraphics[width=\textwidth]{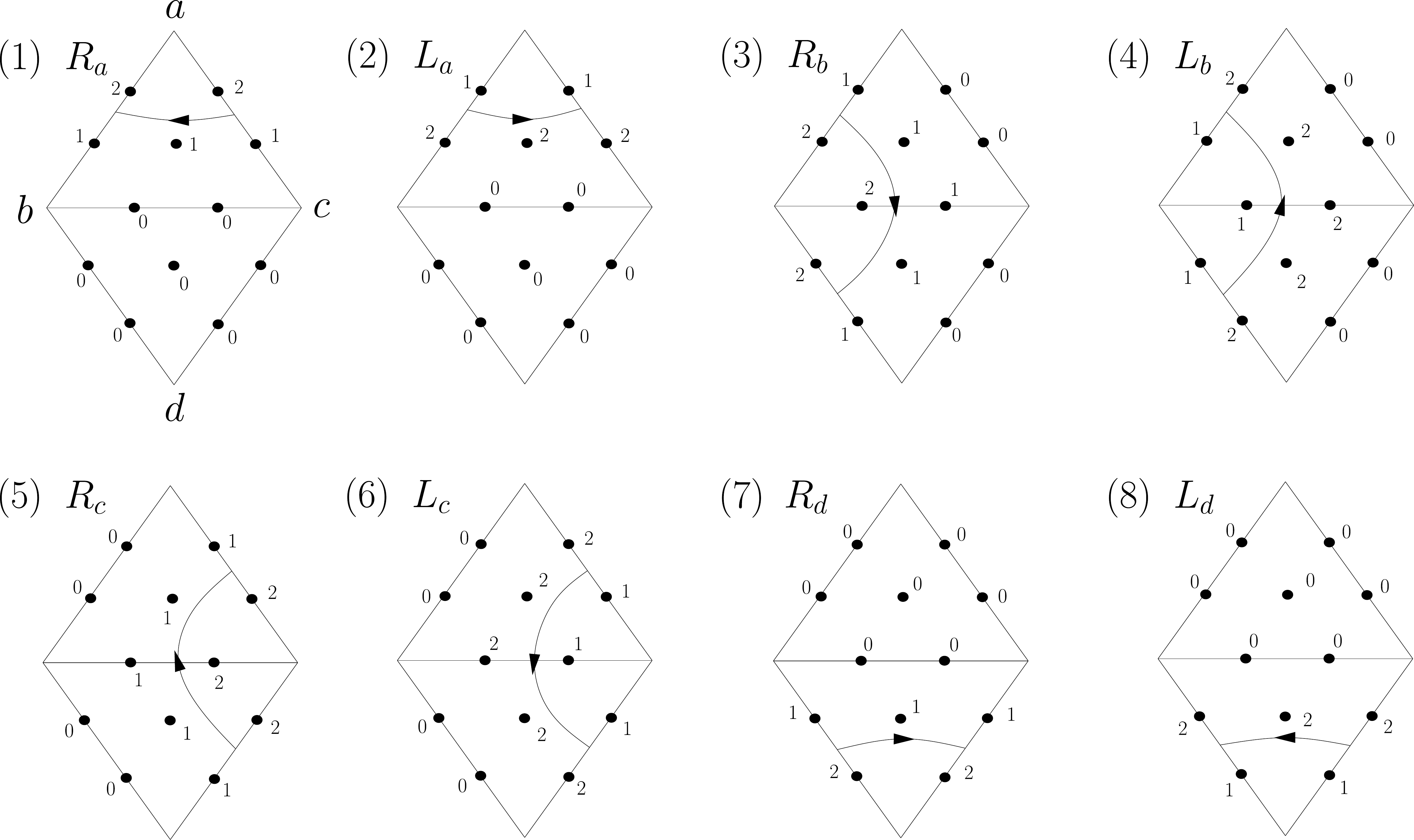}
\caption{     (See also Figure \ref{figure:square2}.)  The first 8 elements of the 22 element Hilbert basis for the KTGS cone $\mathcal{C}_\mathcal{T}$ of the triangulated square $(\Box, \mathcal{T})$, pictured via the corresponding `irreducible' reduced webs $\{ W^H_i \}_{i=1,2,\dots,22}$.   }
\label{figure:square1}
\end{figure}

\begin{figure}[t]
\includegraphics[width=.9\textwidth]{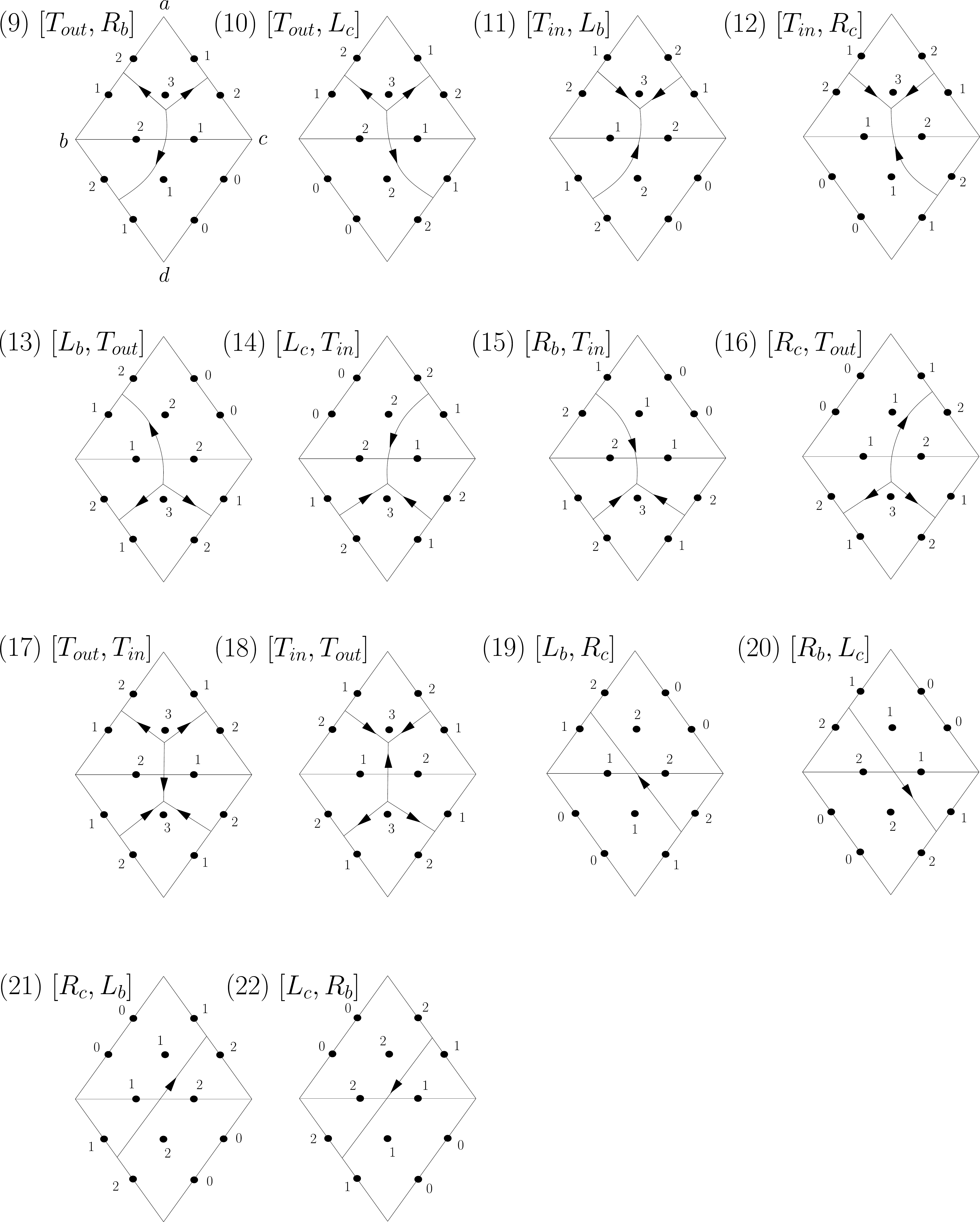}
\caption{     (See also Figure \ref{figure:square1}.)  The last 14 elements of the 22 element Hilbert basis for the KTGS cone $\mathcal{C}_\mathcal{T}$ of the triangulated square $(\Box, \mathcal{T})$, pictured via the corresponding `irreducible' reduced webs $\{ W^H_i \}_{i=1,2,\dots,22}$.  The square bracket is a purely notational device for webs (9)-(22); the first entry of $[\cdot, \cdot]$ corresponds to the top triangle, and the second entry  to the bottom triangle.}
\label{figure:square2}
\end{figure}

\subsubsection{Two linear isomorphisms:  first isomorphism $\theta_\mathcal{T}$ by rhombus numbers}
\label{sssec:a-linear-isomorphism}
		   
Recall (Definition \ref{def:KTGS-cone}) that the KTGS cone $\mathcal{C}_\mathcal{T}$ for any triangulated marked surface $(\widehat{S}, \mathcal{T})$ is defined as the points in $\mathbb{Z}^N$ satisfying two conditions per rhombus, where there are three rhombi per pointed ideal triangle $\Delta$ of $\mathcal{T}$   (Section \ref{ssec:markedsurfacesidealtriangulations}).  Both conditions involve the quantity $3\beta=a+b-c-d$ associated to the rhombus; the first being that $3\beta = a+b-c-d \geq 0$, and the second that $\beta = (a+b-c-d)/3 \in \mathbb{Z}$ is an integer. (Recall  $d=0$ if the rhombus is a corner rhombus; Section \ref{ssec:markedsurfacesidealtriangulations}.) Let $\{ \beta_i \}_i$ denote these `rhombus numbers', varying over all the rhombi of $\mathcal{T}$.

It will be convenient in the remainder of the paper to talk about real vector spaces $\mathbb{R}^N$, which we think of as containing $\mathbb{Z}^N$, in particular the KTGS cone $\mathcal{C}_\mathcal{T}$, as a subset. 

In this sub-subsection, for the triangulated ideal square $(\Box, \mathcal{T})$ we define a linear isomorphism $\theta_\mathcal{T}$ of real 12-dimensional vector spaces, which is mentioned in the proof of Theorem \ref{theorem:basis} and used in Section \ref{section:ssq}. Here, 12 is the number of tropical coordinates for the square.  Note that the triangulated square  has 18 rhombi $\{ \beta_i \}_{i=1,2,\dots,18}$, as displayed in Figure \ref{fig:tropical-X-coords} in Section \ref{section:ssq}.  
		
\begin{definition}
Let $(\Box, \mathcal{T})$ be the triangulated square, whose coordinates are labeled as in  the left hand side of Figure \ref{figure:flip}.  Define a linear map 
\begin{equation*}
\theta_\mathcal{T} : \mathbb{R}^{12} \to \mathbb{R}^{18}
\end{equation*} 
by the formula
\begin{equation*}
\theta_\mathcal{T}(x_1, x_2, \dots, x_8, y_1, \dots, y_4) = (\beta_1, \beta_2, \beta_3, \beta_4, \beta_5, \beta_6, \beta_7, \beta_8, \beta_9, \beta_{10}, \beta_{11}, \beta_{12}, \beta_{13}, \beta_{14}, \beta_{15}, \beta_{16}, \beta_{17}, \beta_{18})
\end{equation*}
where the $\{ \beta_i \}_{i=1,2,\dots,18}$ are the 18 rhombus numbers defined above.  
\end{definition}

For example, the images under $\theta_\mathcal{T}$ of the 22-element Hilbert basis $\mathcal{H}_{(\Box, \mathcal{T})} \subset \mathcal{C}_\mathcal{T}$ of Theorem \ref{theorem:basis} are calculated from Figures \ref{figure:square1}, \ref{figure:square2}, and \ref{fig:tropical-X-coords} to be:
\begin{enumerate}[label=(\arabic*)]
\item $\theta_\mathcal{T}(\Phi_\mathcal{T}([R_a]))=(1, 0, 0, 0, 0, 0, 0, 0, 0, 0, 0, 0, 0, 0, 0, 0, 0, 0) $;
\item $\theta_\mathcal{T}(\Phi_\mathcal{T}([L_a]))=(0, 1, 1, 0, 0, 0, 0, 0, 0, 0, 0, 0, 0, 0, 0, 0, 0, 0) $;
\item $\theta_\mathcal{T}(\Phi_\mathcal{T}([R_b]))=(0, 0, 0, 1, 0, 0, 0, 0, 0, 0, 0, 0, 1, 0, 0, 0, 0, 0) $;
\item $\theta_\mathcal{T}(\Phi_\mathcal{T}([L_b]))=(0, 0, 0, 0, 1, 1, 0, 0, 0, 0, 0, 0, 0, 1, 1, 0, 0, 0) $;
\item $\theta_\mathcal{T}(\Phi_\mathcal{T}([R_c]))=(0, 0, 0, 0, 0, 0, 1, 0, 0, 0, 0, 0, 0, 0, 0, 1, 0, 0) $;
\item $\theta_\mathcal{T}(\Phi_\mathcal{T}([L_c]))=(0, 0, 0, 0, 0, 0, 0, 1, 1, 0, 0, 0, 0, 0, 0, 0, 1, 1) $;
\item $\theta_\mathcal{T}(\Phi_\mathcal{T}([R_d]))=(0, 0, 0, 0, 0, 0, 0, 0, 0, 1, 0, 0, 0, 0, 0, 0, 0, 0) $;
\item $\theta_\mathcal{T}(\Phi_\mathcal{T}([L_d]))=(0, 0, 0, 0, 0, 0, 0, 0, 0, 0, 1, 1, 0, 0, 0, 0, 0, 0) $;
\item $\theta_\mathcal{T}(\Phi_\mathcal{T}([T_{out}, R_b]))=(0, 0, 1, 0, 0, 1, 0, 0, 1, 0, 0, 0, 1, 0, 0, 0, 0, 0) $;
\item $\theta_\mathcal{T}(\Phi_\mathcal{T}([T_{out}, L_c]))=(0, 0, 1, 0, 0, 1, 0, 0, 1, 0, 0, 0, 0, 0, 0, 0, 1, 1) $;
\item $\theta_\mathcal{T}(\Phi_\mathcal{T}([T_{in}, L_b]))=(0, 1, 0, 0, 1, 0, 0, 1, 0, 0, 0, 0, 0, 1, 1, 0, 0, 0) $;
\item $\theta_\mathcal{T}(\Phi_\mathcal{T}([T_{in}, R_c]))=(0, 1, 0, 0, 1, 0, 0, 1, 0, 0, 0, 0, 0, 0, 0, 1, 0, 0) $;
\item $\theta_\mathcal{T}(\Phi_\mathcal{T}([L_b, T_{out}]))=(0, 0, 0, 0, 1, 1, 0, 0, 0, 0, 0, 1, 0, 0, 1, 0, 0, 1) $;
\item $\theta_\mathcal{T}(\Phi_\mathcal{T}([L_c, T_{in}]))=(0, 0, 0, 0, 0, 0, 0, 1, 1, 0, 1, 0, 0, 1, 0, 0, 1, 0) $;
\item $\theta_\mathcal{T}(\Phi_\mathcal{T}([R_b, T_{in}]))=(0, 0, 0, 1, 0, 0, 0, 0, 0, 0, 1, 0, 0, 1, 0, 0, 1, 0) $;
\item $\theta_\mathcal{T}(\Phi_\mathcal{T}([R_c, T_{out}]))=(0, 0, 0, 0, 0, 0, 1, 0, 0, 0, 0, 1, 0, 0, 1, 0, 0, 1) $;
\item $\theta_\mathcal{T}(\Phi_\mathcal{T}([T_{out}, T_{in}]))=(0, 0, 1, 0, 0, 1, 0, 0, 1, 0, 1, 0, 0, 1, 0, 0, 1, 0) $;
\item $\theta_\mathcal{T}(\Phi_\mathcal{T}([T_{in},T_{out}]))=(0, 1, 0, 0, 1, 0, 0, 1, 0, 0, 0, 1, 0, 0, 1, 0, 0, 1) $;
\item $\theta_\mathcal{T}(\Phi_\mathcal{T}([L_b,R_c]))=(0, 0, 0, 0, 1, 1, 0, 0, 0, 0, 0, 0, 0, 0, 0, 1, 0, 0) $;
\item $\theta_\mathcal{T}(\Phi_\mathcal{T}([R_b,L_c]))=(0, 0, 0, 1, 0, 0, 0, 0, 0, 0, 0, 0, 0, 0, 0, 0, 1, 1) $;
\item $\theta_\mathcal{T}(\Phi_\mathcal{T}([R_c,L_b]))=(0, 0, 0, 0, 0, 0, 1, 0, 0, 0, 0, 0, 0, 1, 1, 0, 0, 0) $;
\item $\theta_\mathcal{T}(\Phi_\mathcal{T}([L_c, R_b]))=(0, 0, 0, 0, 0, 0, 0, 1, 1, 0, 0, 0, 1, 0, 0, 0, 0, 0) $.
\end{enumerate}

When there is no confusion, we also let $\beta_i$ denote the general $i$-th coordinate of $\mathbb{R}^{18}$.  

Consider the subspace $V_\mathcal{T} \subset \mathbb{R}^{18}$ defined by
\begin{gather*}\tag{j}
\label{eq:tropical-x-coordinates}
V_\mathcal{T} = \{ (\beta_i)_i \in \mathbb{R}_{18};  
X_1 := \beta_3 - \beta_2 = \beta_6 - \beta_5 = \beta_9 - \beta_8, 
X_2 := \beta_4 - \beta_{13} = \beta_{17} - \beta_9, 
\\ \notag	X_3 := \beta_{12} - \beta_{11} = \beta_{15}-\beta_{14} = \beta_{18} - \beta_{17}, 
X_4 := \beta_{16} - \beta_7 = \beta_5 - \beta_{15}  \}.
\end{gather*}
See  Section \ref{sssec:second-linear-isomorphism} for a discussion of the geometric meaning of the subspace $V_\mathcal{T}$ and the quantities $X_i$.
	
\begin{prop}
\label{obs:theta-is-injective}
The linear map $\theta_\mathcal{T}:\mathbb{R}^{12} \to \mathbb{R}^{18}$ is an isomorphism of $\mathbb{R}^{12}$ onto $V_\mathcal{T}$.  That is, $\theta_\mathcal{T}$ is injective, and the image of $\theta_\mathcal{T}$ is equal to $V_\mathcal{T}$.  In particular, $V_\mathcal{T}$ is $12$-dimensional.
\end{prop}

\begin{proof}
That $\theta_\mathcal{T}(\mathbb{R}^{12}) \subset V_\mathcal{T}$ follows from the definition of the rhombus numbers $\{ \beta_i \}_{i=1,2,\dots,18}$; compare Figure \ref{fig:tropical-X-coords}.    The remainder of the proof is elementary.  
\end{proof}

\begin{conceptremark}
\label{rem:sec6conceptremark}
Recall from Remark \ref{rem:conceptual-remarks} that we view the positive integer cone $\mathcal{C}_\mathcal{T} \cong -3 \mathcal{A}_{\mathrm{PGL}_3, \Box}^+(\mathbb{Z}^t)_\mathcal{T}$ as a $\mathcal{T}$-chart for the positive tropical integer points $\mathcal{A}_{\mathrm{PGL}_3, \Box}^+(\mathbb{Z}^t) \subset \mathcal{A}_{\mathrm{PGL}_3, \Box}(\mathbb{R}^t) = \mathcal{A}_{\mathrm{SL}_3, \Box}(\mathbb{R}^t)$, with one tropical $\mathcal{A}$-coordinate $-3(A_{a,b,c}^{i,j,k})^t$ per dot of $\mathcal{T}$.
			
We think of $\mathbb{R}^{12} \cong \mathcal{A}_{\mathrm{SL}_3, \Box}(\mathbb{R}^t)_\mathcal{T} $ as the coordinate chart of $\mathcal{A}_{\mathrm{SL}_3, \Box}(\mathbb{R}^t)$ associated to the ideal triangulation $\mathcal{T}$.  
			
We view the rhombus numbers $\beta_i$ for $i=1,2,\dots,18$ as the tropicalizations $(\alpha_{a;b,c}^{i,j,k})^t$ of the rhombus functions $\alpha_{a;b,c}^{i,j,k}$ on the moduli space $\mathcal{A}_{\mathrm{PGL}_3, \Box}$.    By Proposition \ref{obs:theta-is-injective}, we can also think of the rhombus numbers $( \beta_i )_{i} \in V_\mathcal{T} \subset \mathbb{R}^{18}$ as providing an alternative coordinate chart for $\mathcal{A}_{\mathrm{PGL}_3, \Box}(\mathbb{R}^t)$ via the isomorphism $\theta_\mathcal{T}$, that is,
\begin{equation*}
\mathcal{A}_{\mathrm{PGL}_3, \Box}(\mathbb{R}^t)_\mathcal{T} \cong V_\mathcal{T} \overset{\theta_\mathcal{T}}{\cong} \mathbb{R}^{12} \cong   \mathcal{A}_{\mathrm{SL}_3, \Box}(\mathbb{R}^t)_\mathcal{T}.
\end{equation*}  
\end{conceptremark}

\subsection{Tropical skein relations in the KTGS cone for the square}
\label{ssec:relations-in-the-ktgs-cone-for-the-square}
	   
We end this section with a noteworthy observation, which will not be needed later.

We saw in Remark \ref{rem:word-of-caution} that there are interesting relations in the KTGS cone $\mathcal{C}_\Delta \subset \mathbb{Z}_+^7$ for the triangle.  In fact, there is essentially only one relation (in the sense analogous to Proposition \ref{prop:relations-in-the-square}).  The intuitive reason there is only 1 relation for the triangle is because the Hilbert basis for $\mathcal{C}_\Delta$ has 8 elements, whereas there are   only 7 Fock-Goncharov coordinates.  
	
We now describe all of the relations in the KTGS cone $\mathcal{C}_\mathcal{T} \subset \mathbb{Z}_+^{12}$ for the square.  Intuitively, there are 10 relations because the Hilbert basis for $\mathcal{C}_\mathcal{T}$ has 22 elements, whereas there are only 12 Fock-Goncharov coordinates.  

\begin{prop}
\label{prop:relations-in-the-square}
The following $10$ linear relations are independent and complete among the $22$ elements of the Hilbert basis $\mathcal{H}_{(\Box, \mathcal{T})} \subset \mathcal{C}_\mathcal{T}$ for the KTGS cone for the square:
\begin{enumerate}[label=\textnormal{(\arabic*)}]
\item $\Phi_\mathcal{T}([T_{in},L_b])+ \Phi_\mathcal{T}([T_{out},L_c])=\Phi_\mathcal{T}([L_a])+\Phi_\mathcal{T}([L_b])+\Phi_\mathcal{T}([L_c]) $;
\item $\Phi_\mathcal{T}([T_{out},L_c])+ \Phi_\mathcal{T}([T_{in},R_c])=\Phi_\mathcal{T}([L_a])+\Phi_\mathcal{T}([L_c])+\Phi_\mathcal{T}([L_b,R_c]) $;
\item $\Phi_\mathcal{T}([L_c, T_{in}])+ \Phi_\mathcal{T}([L_b, T_{out}])=\Phi_\mathcal{T}([L_b])+\Phi_\mathcal{T}([L_c])+\Phi_\mathcal{T}([L_d]) $;
\item $\Phi_\mathcal{T}([L_b,T_{out}])+ \Phi_\mathcal{T}([R_b, T_{in}])+\Phi_\mathcal{T}([T_{out},R_b])
=\Phi_\mathcal{T}([L_b])+\Phi_\mathcal{T}([R_b])+\Phi_\mathcal{T}([L_d])+\Phi_\mathcal{T}([T_{out},L_c]) $;
\item $\Phi_\mathcal{T}([R_c,T_{out}])+ \Phi_\mathcal{T}([L_b, R_c])=\Phi_\mathcal{T}([L_b,T_{out}])+\Phi_\mathcal{T}([R_c]) $;
\item $\Phi_\mathcal{T}([T_{out},T_{in}])+ \Phi_\mathcal{T}([L_b, T_{out}])=\Phi_\mathcal{T}([L_b])+\Phi_\mathcal{T}([L_d])+\Phi_\mathcal{T}([T_{out}, L_c]) $;
\item $\Phi_\mathcal{T}([T_{in},T_{out}])+ \Phi_\mathcal{T}([T_{out},L_c])=\Phi_\mathcal{T}([L_a])+\Phi_\mathcal{T}([L_c])+\Phi_\mathcal{T}([L_b, T_{out}]) $;
\item $\Phi_\mathcal{T}([T_{out},R_b])+ \Phi_\mathcal{T}([R_b,L_c])=\Phi_\mathcal{T}([R_b])+\Phi_\mathcal{T}([T_{out}, L_c]) $;
\item $\Phi_\mathcal{T}([L_b,R_c])+ \Phi_\mathcal{T}([R_c,L_b])=\Phi_\mathcal{T}([L_b])+\Phi_\mathcal{T}([R_c]) $;
\item $\Phi_\mathcal{T}([T_{out},L_c])+ \Phi_\mathcal{T}([L_c,R_b])=\Phi_\mathcal{T}([T_{out},R_b])+\Phi_\mathcal{T}([L_c]) $.
\end{enumerate}
\end{prop}

\begin{proof}
More precisely, what is meant by the statement of the proposition is the following.  Let $f : \mathbb{R}^{22} \to \mathbb{R}^{12}$ be the linear map $ f(\lambda_1, \lambda_2, \dots, \lambda_{22}) =  \sum_{i=1}^{22} \lambda_i \Phi_\mathcal{T}(W_i^H)  \in \mathbb{R}^{12}, $ where the webs $W_i^H$ are as in Theorem \ref{theorem:basis}.    Each of the 10 relations above determines an element $r_j$ of $\mathbb{R}^{22}$.    Let $V \subset \mathbb{R}^{22}$ be the kernel of $f$.   The claim is that  the elements $\{ r_j \}_{j=1,2,\dots,10}$ form a basis of $V$.	 The remainder of the proof is elementary.  
\end{proof}

\begin{remark}
The  relations of Proposition \ref{prop:relations-in-the-square} can be viewed as `tropical $\mathrm{SL}_3$ skein relations'.  Indeed, they can be `predicted' as the result of resolving the overlapping webs in the square (corresponding to a side of a given relation in the cone) by  the  Kuperberg $\mathrm{SL}_3$ skein relation \cite[Section $4$, $q=1$]{KuperbergCommMathPhys96} (one resolution per crossing in the picture).  See also \cite{XieArxiv13}.    
\end{remark}

\section{KTGS cone for the square:  sector decomposition}
\label{section:ssq}

In Section \ref{section:Hsq}, we saw that the Knutson-Tao-Goncharov-Shen cone $\mathcal{C}_\mathcal{T} \subset \mathbb{Z}_+^{12}$ for the triangulated square $(\Box, \mathcal{T})$ has a Hilbert basis $\mathcal{H}_{(\Box, \mathcal{T})} \subset \mathcal{C}_\mathcal{T}$ consisting of 22 elements (Figures \ref{figure:square1} and \ref{figure:square2}).  There are many linear dependence relations in $\mathbb{R}^{12}$ among these Hilbert basis elements; see Proposition \ref{prop:relations-in-the-square}.  In this last section, we study certain linearly independent subsets of the Hilbert basis $\mathcal{H}_{(\Box, \mathcal{T})}$ that have  topological interpretations in terms of webs.  

More specifically, we show that each of the 42 web families $\mathcal{W}_i \subset \mathcal{W}_\Box$ (Section \ref{ssec:42-reduced-web-families-in-the-square} and Figure \ref{figure:9cases}) corresponds to a 12-dimensional subcone $\mathcal{C}_\mathcal{T}^i \subset \mathcal{C}_\mathcal{T}$ (called a sector) generated by 12 Hilbert basis elements.  Moreover, every point in the KTGS cone $\mathcal{C}_\mathcal{T}$ lies in such a sector $\mathcal{C}_\mathcal{T}^i$.  These sectors have a geometric description in terms of  tropical integer $\mathcal{X}$-coordinates (Figure \ref{fig:tropical-X-coords}) for reduced webs $W \in \mathcal{W}_\Box$, which are functions of the corresponding positive tropical integer $\mathcal{A}$-coordinates (Figure \ref{fig:coordinates-example}); we already encountered some of these ideas in Section \ref{sssec:a-linear-isomorphism}.   

In summary, this analysis gives us a deeper understanding of the combinatorial, geometric, and topological properties of the KTGS cone $\mathcal{C}_\mathcal{T} \subset \mathbb{Z}_+^{12}$ for the square; see Figure \ref{figure:wallscross}.   (In this section, we will use some of the terminology and results of Appendix \ref{sec:cones}.)

\subsection{Web families in the square}
\label{ssec:42-reduced-web-families-in-the-square}
	  
Recall the notion of a web schematic; see Section \ref{sssec:splitidealtriangulationsandgoodposition} and Remark \ref{rem:cornerarcschematics}.  Recall also Definitions \ref{def:corner-arcs} and \ref{def:cornerless-webs-in-the-square}, for the notions of corner webs $W=W_r \in \mathcal{R}$ and cornerless webs $W=W_c$.

\begin{prop}
\label{lem:42families}
We can write the reduced webs in the triangulated square as a union 
\begin{equation*}
\mathcal{W}_\Box = \cup_{i=1}^{42} \mathcal{W}_i
\end{equation*}
of $42$ families $\mathcal{W}_i \subset \mathcal{W}_\Box$ of reduced webs, where by definition $W \in \mathcal{W}_i$ if its cornerless part $W_c$ can be represented by the `$i$-th cornerless schematic',  $9$ of which are shown in Figure {\upshape\ref{figure:9cases}}; in fact, up to rotation, reflection, and orientation-reversing symmetry \textup{(}see the caption of Figure {\upshape\ref{figure:9cases}}\textup{)}, every family $\mathcal{W}_i$ falls into one of these $9$ cases.  
\end{prop}

\begin{proof}
This is a direct combinatorial count, done by hand.  We note that the number of possibilities is restricted by the topology of web good positions; see Observation \ref{lemma:bigon}.  
\end{proof}

\begin{notation}
\label{not:web-families}
Completely arbitrarily, the index $i_j$ for the family  $\mathcal{W}_{i_j}$ whose cornerless schematic  is labeled $(j)$ in Figure \ref{figure:9cases} $(j=1,2,\dots,9)$ is
\begin{equation*}
i_1 = 29, \,\, i_2=30, \,\, i_3 =42, \,\, i_4=17, \,\, i_5=5, \,\, i_6=6, \,\, i_7=2, \,\, i_8=1, \,\, i_9=33.
\end{equation*}
See also Remark  \ref{rem:symmetry-groupings} just below.  

As we will see later in this section, the family $\mathcal{W}_i$ corresponds to the $i$-th sector shown in Figure \ref{figure:wallscross}.
\end{notation}

\begin{remark}
\label{rem:symmetry-groupings}
If we define an equivalence relation on the 42 families $\mathcal{W}_i$ by rotation, reflection, and orientation-reversing symmetry, then (using Notation \ref{not:web-families}) the symmetry class of:
\begin{enumerate}[label=(\arabic*)]
\item $\mathcal{W}_{i_1}$ has four members,     $\mathcal{W}_{i_1}=\mathcal{W}_{29}$, $\mathcal{W}_{21},\mathcal{W}_{24},\mathcal{W}_{32}$;
\item  $\mathcal{W}_{i_2}$ has four members,      $\mathcal{W}_{i_2}=\mathcal{W}_{30}$, $\mathcal{W}_{23},\mathcal{W}_{22},\mathcal{W}_{31}$;
\item  $\mathcal{W}_{i_3}$ has eight members,      $\mathcal{W}_{i_3}=\mathcal{W}_{42}$, $\mathcal{W}_{36},\mathcal{W}_{37},\mathcal{W}_{38},\mathcal{W}_{39},\mathcal{W}_{40},\mathcal{W}_{41},\mathcal{W}_{35}$;
\item  $\mathcal{W}_{i_4}$ has eight members,      $\mathcal{W}_{i_4}=\mathcal{W}_{17}$, $\mathcal{W}_{18},\mathcal{W}_{19},\mathcal{W}_{20},\mathcal{W}_{25},\mathcal{W}_{26},\mathcal{W}_{27},\mathcal{W}_{28}$;
\item  $\mathcal{W}_{i_5}$ has four members,      $\mathcal{W}_{i_5}=\mathcal{W}_{5}$, $\mathcal{W}_{8},\mathcal{W}_{13},\mathcal{W}_{16}$;
\item  $\mathcal{W}_{i_6}$ has four members,      $\mathcal{W}_{i_6}=\mathcal{W}_{6}$, $\mathcal{W}_{7},\mathcal{W}_{14},\mathcal{W}_{15}$;
\item  $\mathcal{W}_{i_7}$ has four members,      $\mathcal{W}_{i_7}=\mathcal{W}_{2}$, $\mathcal{W}_{3},\mathcal{W}_{10},\mathcal{W}_{11}$;
\item  $\mathcal{W}_{i_8}$ has four members,      $\mathcal{W}_{i_8}=\mathcal{W}_{1}$, $\mathcal{W}_{4},\mathcal{W}_{9},\mathcal{W}_{12}$;
\item  $\mathcal{W}_{i_9}$ has two members,      $\mathcal{W}_{i_9}=\mathcal{W}_{33}$, $\mathcal{W}_{34}$.
\end{enumerate}
\end{remark}

We emphasize that each schematic in Figure \ref{figure:9cases} represents a subset $\mathcal{W}_i \subset \mathcal{W}_\Box$ of reduced webs in the square.  Note these subsets are not disjoint.  Indeed, each intersection $\mathcal{W}_i \cap \mathcal{W}_j$ is at least `8-dimensional', in an appropriate sense (see later in this section), since the set of corner webs $\mathcal{R}$ is contained in each family $\mathcal{W}_i$.  This intersection can contain more than just the corner webs.  For instance, the intersection $\mathcal{W}_{29} \cap \mathcal{W}_{30}$, corresponding to schematics (1) and (2) in Figure \ref{figure:9cases}, is `11-dimensional' (thus, in Figure \ref{figure:wallscross}, sectors 29 and 30 are separated by a wall); the last, 12-th, dimension comes from the source or sink labeled with the weight $x \in \mathbb{Z}_+$ in schematics (1) and (2).  As another example, $\mathcal{W}_{29} \cap \mathcal{W}_{33}$, corresponding to schematics (1) and (9) in Figure \ref{figure:9cases}, is `10-dimensional' (thus, in Figure \ref{figure:wallscross}, sectors 29 and 33 are not separated by a wall).  In fact, each family $\mathcal{W}_i$ is `12-dimensional' (intuitively, this is because the square has 12 Fock-Goncharov coordinates): 8 dimensions come from the corner part $W_r$, and 4 dimensions come from the cornerless part $W_c$.  Correspondingly, each cornerless schematic in Figure \ref{figure:9cases} has four weights $x,y,z,t \in \mathbb{Z}_+$.  

We remind (Remark \ref{rem:cornerarcschematics}) that, in schematics (1) and (2) in Figure \ref{figure:9cases}, we could have reversed the orientations of the two arc components, without changing the class of the web in $\mathcal{W}_\Box$.  On the other hand, the orientation of the weight $x$ component distinguishes schematic (1) from (2); note the caption of Figure \ref{figure:9cases}.  Also, the $t$ and $z$ strands in schematic (3), for example, do not cross in the upper triangle, for otherwise the web would have an external H-4-face (Section \ref{sssec:splitidealtriangulationsandgoodposition}) on the boundary, violating the reduced property.  

Lastly, we remark that the web families $\mathcal{W}_i$ depend on the choice of triangulation $\mathcal{T}$ of the square $\Box$; compare Example (family ($7$)) in Appendix \ref{ssec:flip-examples}, in particular the difference between the cases $z\geq t$ and $z\leq t$, the former case which is demonstrated in Figures \ref{figure:flip7}-\ref{fig:case7flip}.

\begin{figure}[t]
\includegraphics[scale=.57]{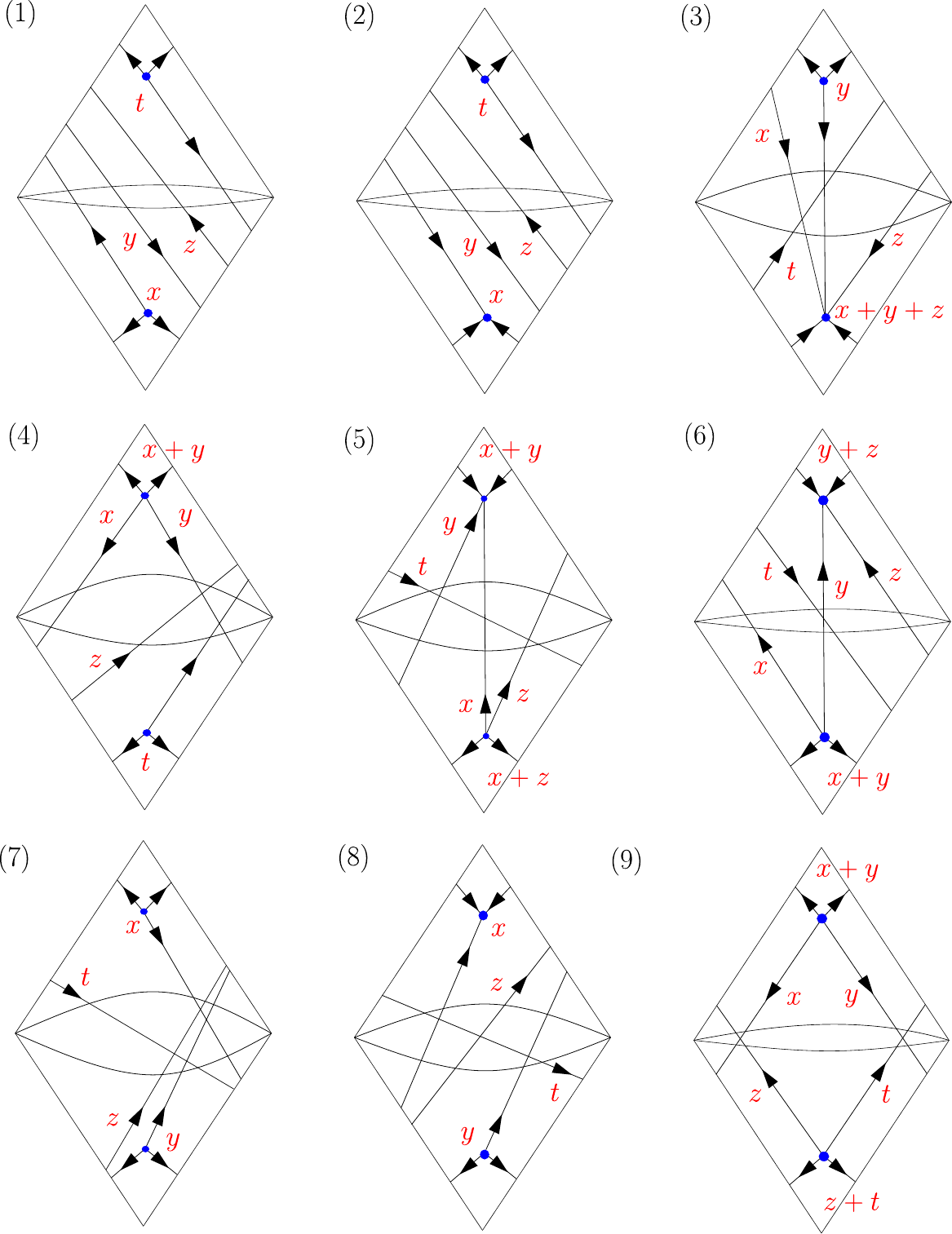}
\caption{     Families (1)-(9).    Schematics for cornerless webs $W=W_c$.  There are $9$ reduced web families up to rotation, reflection, and orientation-reversing symmetry.  (Note that orientation-reversing symmetry means simultaneously reversing the orientations of all components of the web.)}
\label{figure:9cases}
\end{figure}

\subsection{Sector decomposition of the KTGS cone for the triangle and the square}
\label{ssec:sector-decomposition-of-the-KTGS-cone-for-the-triangle-and-the-square}
	   
Recall the notion of a sector decomposition $\{ C_i \}_i$ of a full cone $C \subset \mathbb{R}^k$, and of a wall between two full cones; see Definitions \ref{def:cone-defs} and \ref{def:sector-decomp}.  
		
\begin{definition}
\label{def:completed-ktgs-cone}
Let $\widehat{S}$ be a marked surface, either the ideal triangle or the ideal square, and let $\mathcal{T}$ be an ideal triangulation of $\widehat{S}$.  The \emph{completed Knutson-Tao-Goncharov-Shen (KTGS) cone} $C_\mathcal{T}$ is the completion $C_\mathcal{T}= \overline{\mathcal{C}}_\mathcal{T} \subset \mathbb{R}_+^N$ of the KTGS cone $\mathcal{C}_\mathcal{T} \subset \mathbb{Z}_+^N$; see Definitions \ref{def:KTGS-cone}, \ref{def:completion},  Propositions \ref{prop:KTGS-is-positive-integer-cone}, \ref{prop:hilb-base-triangle}, and Theorem \ref{theorem:basis} (note however Remark \ref{rem:commentsabouthilbertbasisutility}).  
\end{definition}

\subsubsection{Sector decomposition for the triangle}
\label{sssec:sector-decomposition-for-the-triangle}
	   
Let $S=\Delta$ be the ideal triangle.  We will use the notation of Section \ref{sssec:hilbert-basis-for-the-triangle}.  
	
\begin{prop}
\label{prop:sector-decomp-triangle}
The completed KTGS cone $C_\Delta \subset \mathbb{R}_+^7$ is  $7$-dimensional.  Putting
\begin{gather*}
\mathcal{C}_\Delta^{out}=\mathrm{span}_{\mathbb{Z}_+}( \{ \Phi_\Delta(W^H); 
W^H = L_a, R_a, L_b, R_b, L_c, R_c, T_{out} \} )
\subset \mathcal{C}_\Delta\subset \mathbb{Z}_+^7,
\\	\mathcal{C}_\Delta^{in}=\mathrm{span}_{\mathbb{Z}_+}( \{ \Phi_\Delta(W^H); 
W^H = L_a, R_a, L_b, R_b, L_c, R_c, T_{in} \} )
\subset \mathcal{C}_\Delta\subset \mathbb{Z}_+^7,
\\  C_\Delta^{out} = \overline{\mathcal{C}}_\Delta^{out},\,\,
C_\Delta^{in} = \overline{\mathcal{C}}_\Delta^{in}
 \subset C_\Delta  \subset \mathbb{R}_+^7
\end{gather*}
yields a sector decomposition $\{ C_\Delta^{out}, C_\Delta^{in} \}$ of $C_\Delta$.  Moreover, $C_\Delta^{out} \cap C_\Delta^{in}$ is a $6$-dimensional wall, generated by the cone points $\Phi_\Delta(W^H)$ corresponding to the $6$ corner arcs in $\mathcal{W}_\Delta$.  

Moreover, the rank \textup{(}Definition {\upshape\ref{def:cone-defs}}\textup{)} of the completed KTGS cone $C_\Delta \subset \mathbb{R}_+^7$ is $8$.
\end{prop}
	
\begin{proof}
This is a consequence of \cite[Proposition 6.6]{DouglasArxiv20} and its proof.  		

	Indeed, by  \cite[Proposition 6.6]{DouglasArxiv20} we have that $\mathcal{C}_\Delta = \mathcal{C}_\Delta^{out} \cup \mathcal{C}_\Delta^{in}$.  Thus, $C_\Delta = C_\Delta^{out} \cup C_\Delta^{in}$ by Lemma \ref{lem:actual-geometry}.  Also, again by  \cite[Proposition 6.6]{DouglasArxiv20}   (or elementary linear algebra), the cones $C_\Delta$, $C_\Delta^{out}$, $C_\Delta^{in} \subset \mathbb{R}_+^7$ are full; in particular, $C_\Delta^{out}$ and $C_\Delta^{in}$ are sectors.  
	
	  As mentioned in passing during the proof of Proposition \ref{prop:hilb-base-triangle}, by the proof of  \cite[Proposition 6.6]{DouglasArxiv20} there is a linear isomorphism $f : \mathbb{R}^7 \to \mathbb{R}^7$ satisfying:    
	\begin{enumerate}[label=\textnormal{(\Roman*)}]
	\item $f(C_\Delta) \subset \mathbb{R}_+^6 \times \mathbb{R}$;
	\item $f$ sends the corner arcs points $\Phi_\Delta(W^H)$ for $W^H=L_a, R_a, L_b, R_b, L_c, R_c$ to the first six standard basis elements $e_i$ of $\mathbb{R}^6 \times \mathbb{R}$ for $i=1,2,\dots,6$;
	\item  $f(\Phi_\Delta(T_{in}))=(0,0,0,0,0,0;1)$ and  $f(\Phi_\Delta(T_{out}))=(0,1,0,1,0,1;-1)$.  
	\end{enumerate}
	It follows that  $f(C_\Delta^{out}) \cap f(C_\Delta^{in}) = \mathbb{R}_+^6 \times \left\{ 0 \right\}$ has empty interior; moreover, it is a 6-dimensional wall.  As these properties are preserved by isomorphisms of $\mathbb{R}^7$, we conclude all but the last sentence of the result.

	It remains to show the rank of $C_\Delta$ is $8$.  By Proposition \ref{prop:hilb-base-triangle}, the Hilbert basis $\mathcal{H}_\Delta$ of the positive integer cone $\mathcal{C}_\Delta \subset \mathbb{Z}_+^7$ has 8 elements.  It follows by Observation \ref{obs:rank-hilbert-basis-estimate} that $\mathrm{rank}(C_\Delta) \leq 8$.

	By the first part of the proposition already proved, $\mathrm{rank}(C_\Delta) \geq 7$.  We show there is no generating set with 7 elements.

We again work in the isomorphic  cone $C := f(C_\Delta) \subset \mathbb{R}_+^6 \times \mathbb{R}$ from above.  Suppose $\left\{ c_i \right\}_{i=1,2,\dots,7}$ is a generating set of $C$.  Let $\pi_6 : \mathbb{R}^6 \times \mathbb{R} \to \mathbb{R}^6$ be the natural projection.  Since the standard basis element $e_i$ is in $C$ for $i=1,2,\dots,6$, and since $\pi_6(C)\subset\text{(in fact, equals) }\mathbb{R}_+^6$, after possibly changing indices we can assume, for $i=1,2,\dots,6$, that $c_i = (e^{(6)}_i; \alpha_i)$ where $e^{(6)}_i$ is the $i$-th standard basis element of $\mathbb{R}^6$ and $\alpha_i \in \mathbb{R}$.  

Since, by above, $C$ has a generating set where just the one generator $(0,1,0,1,0,1;-1)$ has a negative last coordinate, it follows that any point of $C$ with a negative last coordinate has  positive second, fourth, and sixth coordinates.  Thus,  $\alpha_i \geq 0$ for $i=1,2,\dots,6$.  

The remaining generator $c_7$ must therefore have a  negative last coordinate.  We conclude  $(0,0,0,0,0,0;1) \in C$ cannot be generated with $\left\{ c_i \right\}_{i=1,2,\dots,7}$.  
\end{proof}

\subsubsection{Sector decomposition for the square}
\label{sssec:sector-decomposition-for-the-square}
	   
Let $S=\Box$ be the ideal square, equipped with an ideal triangulation $\mathcal{T}$, namely a choice of diagonal.  We will use the notation of Section \ref{sssec:hilbert-basis-for-the-square-statement}.  In particular, recall the 22 Hilbert basis webs $W^H_j$ in $\mathcal{W}_\Box$ ($j=1,2,\dots,22$) associated to $\mathcal{T}$; see Figures \ref{figure:square1} and \ref{figure:square2}.  
	
We define 42 subcones $C_\mathcal{T}^i \subset C_\mathcal{T}$ ($i=1,2,\dots,42$)  of the completed KTGS cone $C_\mathcal{T} \subset \mathbb{R}_+^{12}$ as follows.  First, define 42 web subsets $\mathcal{Q}_i \subset \mathcal{W}_\Box$ ($i=1,2,\dots,42$), each consisting of four Hilbert basis webs $W^H$, by:
\begin{enumerate}[label=(\arabic*)]
\item $\mathcal{Q}_1=\{[T_{in}, L_b],[R_c, T_{out}],[R_c, L_b],[R_b, L_c]\} $;
\item $ \mathcal{Q}_2=\{[T_{out}, L_c],[R_c, T_{out}],[R_c, L_b],[R_b, L_c]\} $;
\item $ \mathcal{Q}_3=\{[T_{in}, L_b],[R_b, L_c],[R_b, T_{in}],[R_c, L_b]\} $;
\item $ \mathcal{Q}_4=\{[T_{out}, L_c],[R_b, L_c],[R_b, T_{in}],[R_c, L_b]\} $;
\item $ \mathcal{Q}_5=\{[T_{in}, T_{out}],[T_{in}, L_b],[R_c, T_{out}],[R_b, L_c]\} $;
\item $ \mathcal{Q}_6=\{[L_b, T_{out}],[T_{in}, T_{out}],[T_{in}, R_c],[R_b, L_c]\} $;
\item $ \mathcal{Q}_7=\{[T_{in}, L_b],[L_c, R_b],[T_{in}, T_{out}],[R_c, T_{out}]\} $;
\item $ \mathcal{Q}_8=\{[T_{in}, T_{out}],[L_c,R_b],[L_b, T_{out}],[T_{in}, R_c]\} $;
\item $ \mathcal{Q}_9=\{[T_{out}, R_b],[L_c,R_b],[L_c, T_{in}],[L_b, R_c]\} $;
\item $ \mathcal{Q}_{10}=\{[T_{in}, R_c],[L_c,R_b],[L_c, T_{in}],[L_b, R_c]\} $;
\item $ \mathcal{Q}_{11}=\{[T_{out}, R_b],[L_c,R_b],[L_b, T_{out}],[L_b, R_c]\} $;
\item $ \mathcal{Q}_{12}=\{[T_{in}, R_c],[L_c,R_b],[L_b, T_{out}],[L_b, R_c]\} $;
\item $ \mathcal{Q}_{13}=\{[T_{out},T_{in}],[T_{out}, R_b],[L_c, T_{in}],[L_b, R_c]\} $;
\item $ \mathcal{Q}_{14}=\{[T_{out},L_c],[R_b, T_{in}],[T_{out},T_{in}],[L_b, R_c]\} $;
\item $ \mathcal{Q}_{15}=\{[T_{out},T_{in}],[T_{out}, R_b],[L_c, T_{in}],[R_c, L_b]\} $;
\item $ \mathcal{Q}_{16}=\{[T_{out},L_c],[R_b, T_{in}],[T_{out},T_{in}],[R_c, L_b]\} $;
\item $ \mathcal{Q}_{17}=\{[T_{out}, R_b],[T_{out},L_c],[R_c, L_b],[R_c,T_{out}]\} $;
\item $ \mathcal{Q}_{18}=\{[T_{in},L_b],[R_b, T_{in}],[L_c,T_{in}],[R_c, L_b]\} $;
\item $ \mathcal{Q}_{19}=\{[T_{in},L_b],[L_c,R_b],[L_c,T_{in}],[T_{in}, R_c]\} $;
\item $ \mathcal{Q}_{20}=\{[T_{out}, R_b],[L_c,R_b],[L_b, T_{out}],[R_c,T_{out}]\} $;
\item $ \mathcal{Q}_{21}=\{[T_{out}, R_b],[L_c,R_b],[R_c,T_{out}],[R_c,L_b]\} $;
\item $ \mathcal{Q}_{22}=\{[T_{out}, R_b],[L_c,R_b],[L_c,T_{in}],[R_c,L_b]\} $;
\item $ \mathcal{Q}_{23}=\{[T_{in}, L_b],[L_c,R_b],[R_c,T_{out}],[R_c,L_b]\} $;
\item $ \mathcal{Q}_{24}=\{[T_{in}, L_b],[L_c,R_b],[L_c,T_{in}],[R_c,L_b]\} $;
\item $ \mathcal{Q}_{25}=\{[T_{out},L_c],[T_{out}, R_b],[L_b, T_{out}],[L_b, R_c]\} $;
\item $ \mathcal{Q}_{26}=\{[T_{in}, R_c],[R_b,T_{in}],[L_c, T_{in}],[L_b, R_c]\} $;
\item $ \mathcal{Q}_{27}=\{[T_{in}, L_b],[R_b, L_c],[R_b, T_{in}],[T_{in}, R_c]\} $;
\item $ \mathcal{Q}_{28}=\{[T_{out},L_c],[R_b, L_c],[L_b, T_{out}],[R_c, T_{out}]\} $;
\item $ \mathcal{Q}_{29}=\{[L_b, T_{out}],[R_b, L_c],[L_b,R_c],[T_{out},L_c]\} $;
\item $ \mathcal{Q}_{30}=\{[R_b, T_{in}],[R_b, L_c],[L_b,R_c],[T_{out},L_c]\} $;
\item $ \mathcal{Q}_{31}=\{[T_{in},R_c],[R_b, L_c],[L_b, T_{out}],[L_b,R_c]\} $;
\item $ \mathcal{Q}_{32}=\{[T_{in},R_c],[R_b, L_c],[R_b, T_{in}],[L_b,R_c]\} $;
\item $ \mathcal{Q}_{33}=\{[T_{out}, R_b],[T_{out},L_c],[L_b, T_{out}],[R_c, T_{out}]\} $;
\item $ \mathcal{Q}_{34}=\{[T_{in},L_b],[R_b, T_{in}],[L_c,T_{in}],[T_{in}, R_c]\} $;
\item $ \mathcal{Q}_{35}=\{[T_{in},T_{out}],[T_{in}, L_b],[T_{in}, R_c],[R_b, L_c]\} $;
\item $ \mathcal{Q}_{36}=\{[T_{in},T_{out}],[R_b, L_c],[L_b, T_{out}],[R_c, T_{out}]\} $;
\item $ \mathcal{Q}_{37}=\{[T_{in}, L_b],[L_c,R_b],[T_{in}, T_{out}],[T_{in}, R_c]\} $;
\item $ \mathcal{Q}_{38}=\{[T_{in}, T_{out}],[L_c,R_b],[L_b, T_{out}],[R_c,T_{out}]\} $;
\item $ \mathcal{Q}_{39}=\{[T_{out},L_c],[T_{out}, R_b],[T_{out}, T_{in}],[L_b, R_c]\} $;
\item $ \mathcal{Q}_{40}=\{[T_{out}, T_{in}],[R_b,T_{in}], [L_c, T_{in}],[L_b, R_c]\} $;
\item $ \mathcal{Q}_{41}=\{[T_{out},L_c],[T_{out}, R_b],[T_{out}, T_{in}],[R_c,L_b]\} $;
\item $ \mathcal{Q}_{42}=\{[R_b,T_{in}],[T_{out}, T_{in}], [L_c, T_{in}],[R_c,L_b]\} $.	
\end{enumerate}
	
The $i$-th web subset $\mathcal{Q}_i \subset \mathcal{W}_i \subset \mathcal{W}_\Box$ is moreover a subset of the $i$-th web family $\mathcal{W}_i$ (Section \ref{ssec:42-reduced-web-families-in-the-square}).  More precisely, each of the four Hilbert basis webs $W^H \in \mathcal{Q}_i$ is determined by the schematic picture for the web family $\mathcal{W}_i$ (as in Figure \ref{figure:9cases}) by putting all but one of the variables $x,y,z,t$ to 0 and the remaining variable to 1.  Recall that the 9 specific web families denoted $(j)$ in Figure \ref{figure:9cases} are  the families $\mathcal{W}_{i_j}$ as explained in Notation \ref{not:web-families}.  

\begin{definition} 
\label{def:topological type}
For $i=1,2,\dots,42$, let $\mathcal{Q}_i \subset \mathcal{W}_i \subset \mathcal{W}_\Box$ be the set  of four webs defined just above, and recall that $W_j^H$ for $j=1,2,\dots,8$ are the 8 corner arcs in the square.  

Define the \emph{$i$-th completed KTGS subcone} $C_\mathcal{T}^i \subset C_\mathcal{T} \subset \mathbb{R}_+^{12}$ as the completion
\begin{equation*}
C_\mathcal{T}^i = \overline{\mathcal{C}}_\mathcal{T}^i
\end{equation*}
of the \emph{$i$-th KTGS submonoid} $\mathcal{C}_\mathcal{T}^i \subset \mathcal{C}_\mathcal{T} \subset \mathbb{Z}_+^{12}$, defined by
\begin{equation*}
\mathcal{C}_\mathcal{T}^i = \mathrm{span}_{\mathbb{Z}_+}(\{ \Phi_\mathcal{T}(W_1^H), \Phi_\mathcal{T}(W_2^H), \dots, \Phi_\mathcal{T}(W_8^H) \} \cup \{ \Phi_\mathcal{T}(W^H);  W^H \in \mathcal{Q}_i \} )
\end{equation*}
where $\Phi_\mathcal{T}(W^H) \in \mathcal{C}_\mathcal{T} \subset \mathbb{Z}_+^{12}$ is the point in the KTGS positive integer cone $\mathcal{C}_\mathcal{T}$ assigned to $W^H$ by the web tropical coordinate map $\Phi_\mathcal{T} : \mathcal{W}_\Box \to \mathcal{C}_\mathcal{T}$.  

The set  $\mathcal{Q}_i$ of four webs is called the \emph{topological type} of the completed KTGS subcone  $C_\mathcal{T}^i$.  
\end{definition}

By construction of the web tropical coordinate map $\Phi_\mathcal{T}$, we can immediately say:

\begin{observation}
\label{obs:families-and-their-cones}
For $i=1,2,\dots,42$, we have the image $\Phi_\mathcal{T}(\mathcal{W}_i)=\mathcal{C}_\mathcal{T}^i \subset \mathbb{Z}_+^{12}$.  \qed
\end{observation}

The main result of the second half of the paper is:

\begin{theorem}
\label{theorem:decomp}
Consider the completed KTGS cone $C_\mathcal{T} \subset \mathbb{R}_+^{12}$ for the triangulated square $(\Box, \mathcal{T})$; see Definition {\upshape\ref{def:completed-ktgs-cone}}.  Then the following properties hold.
\begin{enumerate}[label=\textnormal{(\Roman*)}]
\item  $C_\mathcal{T}$ is $12$-dimensional.  Namely, $C_\mathcal{T}$ is a full cone; see Definition {\upshape\ref{def:cone-defs}}.	
\item  	The completed KTGS subcones $C_\mathcal{T}^i \subset C_\mathcal{T}$ are full sectors forming a sector decomposition  $\{ C^i_\mathcal{T} \}_{i=1,2,\dots,42}$  of $C_\mathcal{T}$; see Definition {\upshape\ref{def:sector-decomp}}.    
\item  	The intersection $C_\mathcal{T}^i \cap C_\mathcal{T}^\ell$ is a wall if and only if $\mathcal{Q}_i \cap \mathcal{Q}_\ell$ has $3$ elements; that is, if and only if the topological types of $C_\mathcal{T}^i$ and $C_\mathcal{T}^\ell$ differ by a single web.  In this case, 
\begin{equation*}\tag{k}
\label{eq:theorem-equation}
C_\mathcal{T}^i \cap C_\mathcal{T}^\ell = \mathrm{span}_{\mathbb{R}_+}(\{ \Phi_\mathcal{T}(W_1^H), \Phi_\mathcal{T}(W_2^H), \dots, \Phi_\mathcal{T}(W_8^H) \} \cup \{ \Phi_\mathcal{T}(W^H);  W^H \in \mathcal{Q}_i \cap \mathcal{Q}_\ell \} )  \subset \mathbb{R}_+^{12}.  
\end{equation*}

Moreover, for each web $W^H \in \mathcal{Q}_i$, there exists a unique index $i^\ast(i,W^H) \in \{ 1, 2, \dots, 42 \}$ such that 
\begin{equation*}
\mathcal{Q}_i \cap \mathcal{Q}_{i^\ast(i,W^H)} = \mathcal{Q}_i - \{ W^H \};
\end{equation*}
that is,  such that there is a web $W^\ast(i,W^H) \in \mathcal{Q}_{i^\ast(i,W^H)}$ distinct from $W^H$ satisfying the property that the topological type  $\mathcal{Q}_{i^\ast(i,W^H)}$ of $C_\mathcal{T}^{i^\ast(i,W^H)}$ is obtained from the topological type  $\mathcal{Q}_i$ of $C_\mathcal{T}^i$ by swapping $W^H$ with $W^\ast(i,W^H)$.   
	
In particular, each sector $C^i_\mathcal{T}$ has $4$ walls.  See Figure {\upshape\ref{figure:wallscross}}.  
\end{enumerate}
\end{theorem}

\begin{example}
\label{ex:example-of-web-mutation}
As an example of the second paragraph of the third item of Theorem \ref{theorem:decomp}, consider $i=i_1=29$, corresponding to family (1) in Figure \ref{figure:9cases}.  If $W^H =   [L_b, T_{out}] \in \mathcal{Q}_{29}$, then $i^\ast(29,W^H)=i_2=30$ corresponding to family (2)   in Figure \ref{figure:9cases}, and $W^\ast(29,W^H) = [R_b, T_{in}] \in \mathcal{Q}_{30}$.   One similarly checks that $i^\ast(29, [L_b, R_c])=28$ and $W^\ast(29, [L_b, R_c])=[R_c, T_{out}] \in \mathcal{Q}_{28}$; that $i^\ast(29, [R_b, L_c])=25$ and $W^\ast(29, [R_b, L_c])=[T_{out}, R_b] \in \mathcal{Q}_{25}$; and, that $i^\ast(29, [T_{out}, L_c])=31$ and $W^\ast(29, [T_{out}, L_c])=[T_{in}, R_c] \in \mathcal{Q}_{31}$.  

Note that Figure \ref{figure:wallscross} provides some, but not  all, of this topological information; the full information is contained in the definition of the subsets $\mathcal{Q}_i$ above.  
\end{example}

\begin{question}
By Theorem \ref{theorem:basis}, the Hilbert basis $\mathcal{H}_{(\Box, \mathcal{T})}$ of the positive integer cone $\mathcal{C}_\mathcal{T} \subset \mathbb{Z}_+^{12}$ has 22 elements.  It follows by Observation \ref{obs:rank-hilbert-basis-estimate} that $\mathrm{rank}(C_\mathcal{T}) \leq 22$.   By Theorem  \ref{theorem:decomp}, $\mathrm{rank}(C_\mathcal{T}) \geq 12$.  What is the rank of $C_\mathcal{T}$?  (Compare the last paragraph of Proposition \ref{prop:sector-decomp-triangle}.)
\end{question}

\begin{remark}
\label{rem:topwallcrossphen}
The authors enjoy imagining Theorem \ref{theorem:decomp} as expressing a kind of `topological wall-crossing phenomenon', where we interpret the swapping of topological types upon crossing a wall as a kind of `web mutation'.  Investigating how this phenomenon relates to other wall-crossing phenomena appearing in cluster geometry \cite{kontsevich2008stability} could be of potential interest.  In particular, there should be a relationship with the so-called $D_4$ cluster complex (see, for example, \cite{FominArxiv17}).  This is also related to Remark \ref{rem:hilbert-basis-depends-on-triangulation} and the last paragraph of Section \ref{ssec:42-reduced-web-families-in-the-square}.  
\end{remark}

\subsection{Proof of Theorem \ref{theorem:decomp}}
\label{ssec:proof-of-decomp-theorem}

We make some preparations before proving the theorem.  
	
\subsubsection{Two linear isomorphisms:  second isomorphism $\phi_\mathcal{T}$ by tropical $\mathcal{X}$-coordinates}
\label{sssec:second-linear-isomorphism}
		   
In Section \ref{sssec:a-linear-isomorphism}, to each ideal triangulation $\mathcal{T}$ of the square $\Box$ we constructed a linear isomorphism $\theta_\mathcal{T} : \mathbb{R}^{12} \to V_\mathcal{T} \subset \mathbb{R}^{18}$.  This map sends 12 real numbers $A_1, A_2, \dots, A_{12}$, called the `(real) tropical $\mathcal{A}$-coordinates', to their 18 rhombus numbers $\beta_1, \beta_2, \dots, \beta_{18}$.  There are 6 relations (see \eqref{eq:tropical-x-coordinates} in Section \ref{sssec:a-linear-isomorphism})  defining the 12-dimensional subspace $V_\mathcal{T} \subset \mathbb{R}^{18}$, which determine four real numbers $X_1, X_2, X_3, X_4$, called the `(real) tropical $\mathcal{X}$-coordinates':  they are four numbers assigned to any 18-tuple in $V_\mathcal{T}$ of rhombus numbers.  See Figure \ref{fig:tropical-X-coords}.  See also \cite{XieArxiv13}.    
		
\begin{remark}  
The tropical $\mathcal{X}$-coordinates originate in Fock-Goncharov theory as tropicalized  double and triple ratios \cite{FockIHES06, FockAdvMath07}, and can be thought of in the following geometric way.  For $X_1$, say, consider the hexagon in the top triangle in the top left square of Figure \ref{fig:tropical-X-coords}.  There are six tropical $\mathcal{A}$-coordinates assigned to the vertices of this hexagon.  Then $X_1$ is the signed sum of these coordinates, as indicated in the figure.  Similarly for $X_2, X_3, X_4$.  
\end{remark}
		
\begin{definition}
Let $V_\mathcal{T} \subset \mathbb{R}^{18}$ be the 12-dimensional subspace just discussed.  Define a linear map 
\begin{equation*}
\phi_\mathcal{T} : V_\mathcal{T} \to \mathbb{R}^{8} \times \mathbb{R}^4
\end{equation*}
by 
\begin{equation*}
\phi_\mathcal{T}(\beta_1, \beta_2, \beta_3, \dots, \beta_{18})=(\beta_1, \beta_2, \beta_4, \beta_5, \beta_7, \beta_8, \beta_{10}, \beta_{11};  X_1, X_2, X_3, X_4).  
\end{equation*}
\end{definition}

See Figure \ref{fig:tropical-X-coords}, where the  eight rhombi appearing in the first eight coordinates of the  image of $\phi_\mathcal{T}$ are colored green.  

For example, the images under $\phi_\mathcal{T}$ of $\theta_\mathcal{T}$ applied to the 22 Hilbert basis elements, $\theta_\mathcal{T}(\Phi_\mathcal{T}(W_j^H)) \in V_\mathcal{T}$, can be computed from Figures \ref{figure:square1} and \ref{figure:square2}, or from the computations in Section \ref{sssec:a-linear-isomorphism}, to be:
\begin{enumerate}[label=(\arabic*)]
\item $\phi_\mathcal{T}(\theta_\mathcal{T}(\Phi_\mathcal{T}([R_a])))=(1,0,0,0,0,0,0,0;0,0,0,0)	 $;
\item $\phi_\mathcal{T}(\theta_\mathcal{T}(\Phi_\mathcal{T}([L_a])))=(0,1,0,0,0,0,0,0;0,0,0,0) $;
\item $\phi_\mathcal{T}(\theta_\mathcal{T}(\Phi_\mathcal{T}([R_b])))=(0,0,1,0,0,0,0,0;0,0,0,0) $;
\item $\phi_\mathcal{T}(\theta_\mathcal{T}(\Phi_\mathcal{T}( [L_b])))=(0,0,0,1,0,0,0,0;0,0,0,0) $;
\item $\phi_\mathcal{T}(\theta_\mathcal{T}(\Phi_\mathcal{T}( [R_c])))=(0,0,0,0,1,0,0,0;0,0,0,0) $;
\item $\phi_\mathcal{T}(\theta_\mathcal{T}(\Phi_\mathcal{T}( [L_c])))=(0,0,0,0,0,1,0,0;0,0,0,0) $;
\item $\phi_\mathcal{T}(\theta_\mathcal{T}(\Phi_\mathcal{T}(  [R_d])))=(0,0,0,0,0,0,1,0;0,0,0,0) $;
\item $\phi_\mathcal{T}(\theta_\mathcal{T}(\Phi_\mathcal{T}( [L_d])))=(0,0,0,0,0,0,0,1;0,0,0,0) $;
\item $\phi_\mathcal{T}(\theta_\mathcal{T}(\Phi_\mathcal{T}([T_{out}, R_b])))=(0,0,0,0,0,0,0,0;		1,-1,0,0)   $;
\item $\phi_\mathcal{T}(\theta_\mathcal{T}(\Phi_\mathcal{T}([T_{out}, L_c])))=(0,0,0,0,0,0,0,0;		1,0,0,0)	 $;
\item $\phi_\mathcal{T}(\theta_\mathcal{T}(\Phi_\mathcal{T}([T_{in}, L_b])))=(0,1,0,1,0,1,0,0;		-1,0,0,0)	 $;
\item $\phi_\mathcal{T}(\theta_\mathcal{T}(\Phi_\mathcal{T}( [T_{in}, R_c])))=(0,1,0,1,0,1,0,0;		-1,0,0,1)	 $;
\item $\phi_\mathcal{T}(\theta_\mathcal{T}(\Phi_\mathcal{T}( [L_b, T_{out}])))=(0,0,0,1,0,0,0,0;		0,0,1,0)	 $;
\item $\phi_\mathcal{T}(\theta_\mathcal{T}(\Phi_\mathcal{T}(  [L_c, T_{in}])))=(0,0,0,0,0,1,0,1;		0,0,-1,0)	 $;
\item $\phi_\mathcal{T}(\theta_\mathcal{T}(\Phi_\mathcal{T}(  [R_b, T_{in}])))=(0,0,1,0,0,0,0,1;		0,1,-1,0)	 $;
\item $\phi_\mathcal{T}(\theta_\mathcal{T}(\Phi_\mathcal{T}( [R_c, T_{out}])))=(0,0,0,0,1,0,0,0;		0,0,1,-1)	 $;
\item $\phi_\mathcal{T}(\theta_\mathcal{T}(\Phi_\mathcal{T}(  [T_{out}, T_{in}])))=(0,0,0,0,0,0,0,1;		1,0,-1,0)	 $;
\item $\phi_\mathcal{T}(\theta_\mathcal{T}(\Phi_\mathcal{T}( [T_{in}, T_{out}])))=(0,1,0,1,0,1,0,0;		-1,0,1,0)	 $;
\item $\phi_\mathcal{T}(\theta_\mathcal{T}(\Phi_\mathcal{T}( [L_b, R_c])))=(0,0,0,1,0,0,0,0;			0,0,0,1)	 $;
\item $\phi_\mathcal{T}(\theta_\mathcal{T}(\Phi_\mathcal{T}( [R_b, L_c])))=(0,0,1,0,0,0,0,0;			0,1,0,0)		 $;
\item $\phi_\mathcal{T}(\theta_\mathcal{T}(\Phi_\mathcal{T}( [R_c, L_b])))=(0,0,0,0,1,0,0,0;			0,0,0,-1)	 $;
\item $\phi_\mathcal{T}(\theta_\mathcal{T}(\Phi_\mathcal{T}( [L_c, R_b])))=(0,0,0,0,0,1,0,0;			0,-1,0,0) $.
\end{enumerate}

\begin{prop}
\label{prop:second-isomorphism}
The linear map $\phi_\mathcal{T} : V_\mathcal{T} \to \mathbb{R}^{8} \times \mathbb{R}^4$ is an isomorphism.  Consequently, letting $\theta_\mathcal{T} : \mathbb{R}^{12} \to V_\mathcal{T}$ be the isomorphism from {\upshape Section \ref{sssec:a-linear-isomorphism}}, we have that the composition 
\begin{equation*}
\phi_\mathcal{T} \circ \theta_\mathcal{T} : \mathbb{R}^{12} \overset{\sim}{\to} \mathbb{R}^{8} \times \mathbb{R}^4
\end{equation*}
is a linear isomorphism.  
\end{prop}

\begin{proof}
Since $V_\mathcal{T}$ is 12-dimensional (Proposition \ref{obs:theta-is-injective}), it suffices to show that the image of $\phi_\mathcal{T}$ spans $\mathbb{R}^{8} \times \mathbb{R}^4$.  Indeed, one checks that the above 22 images $\{ \phi_\mathcal{T}(\theta_\mathcal{T}(\Phi_\mathcal{T}(W_j^H)))\}_{j=1,2,\dots,22} \subset \mathbb{R}_+^8 \times \mathbb{R}^4$ span $\mathbb{R}^{8} \times \mathbb{R}^4$. 
\end{proof}

\begin{definition}
\label{def:isomorphic-cone}
The linear isomorphism $\phi_\mathcal{T} \circ \theta_\mathcal{T} : \mathbb{R}^{12} \to \mathbb{R}^{8} \times \mathbb{R}^4$ of Proposition \ref{prop:second-isomorphism} maps the completed KTGS cone $C_\mathcal{T} \subset \mathbb{R}_+^{12}$ to the \emph{isomorphic cone} 
\begin{equation*}
C:=\phi_\mathcal{T}(\theta_\mathcal{T}(C_\mathcal{T}))
 \subset \mathbb{R}_+^8 \times \mathbb{R}^4.
\end{equation*}
\end{definition}

Note that $C$ indeed lies in $\mathbb{R}_+^8 \times \mathbb{R}^4$ because $C_\mathcal{T} \subset \mathbb{R}_+^{12}$ is the completion of the KTGS cone $\mathcal{C}_\mathcal{T} \subset \mathbb{Z}_+^{12}$, which by definition has all nonnegative (integer) rhombus numbers $\{ \beta_i \}_{i=1,2,\dots,18}$.
	
Observe, in particular, that the 8 corner arcs $W^H_1, W^H_2, \dots, W^H_8$ of  Figure \ref{figure:square1} correspond via $\phi_\mathcal{T} \circ \theta_\mathcal{T} \circ \Phi_\mathcal{T}$ to the first 8 standard basis elements $e_i$ of $\mathbb{R}^8 \times \mathbb{R}^4$.

\begin{figure}[t]
\includegraphics[scale=.51]{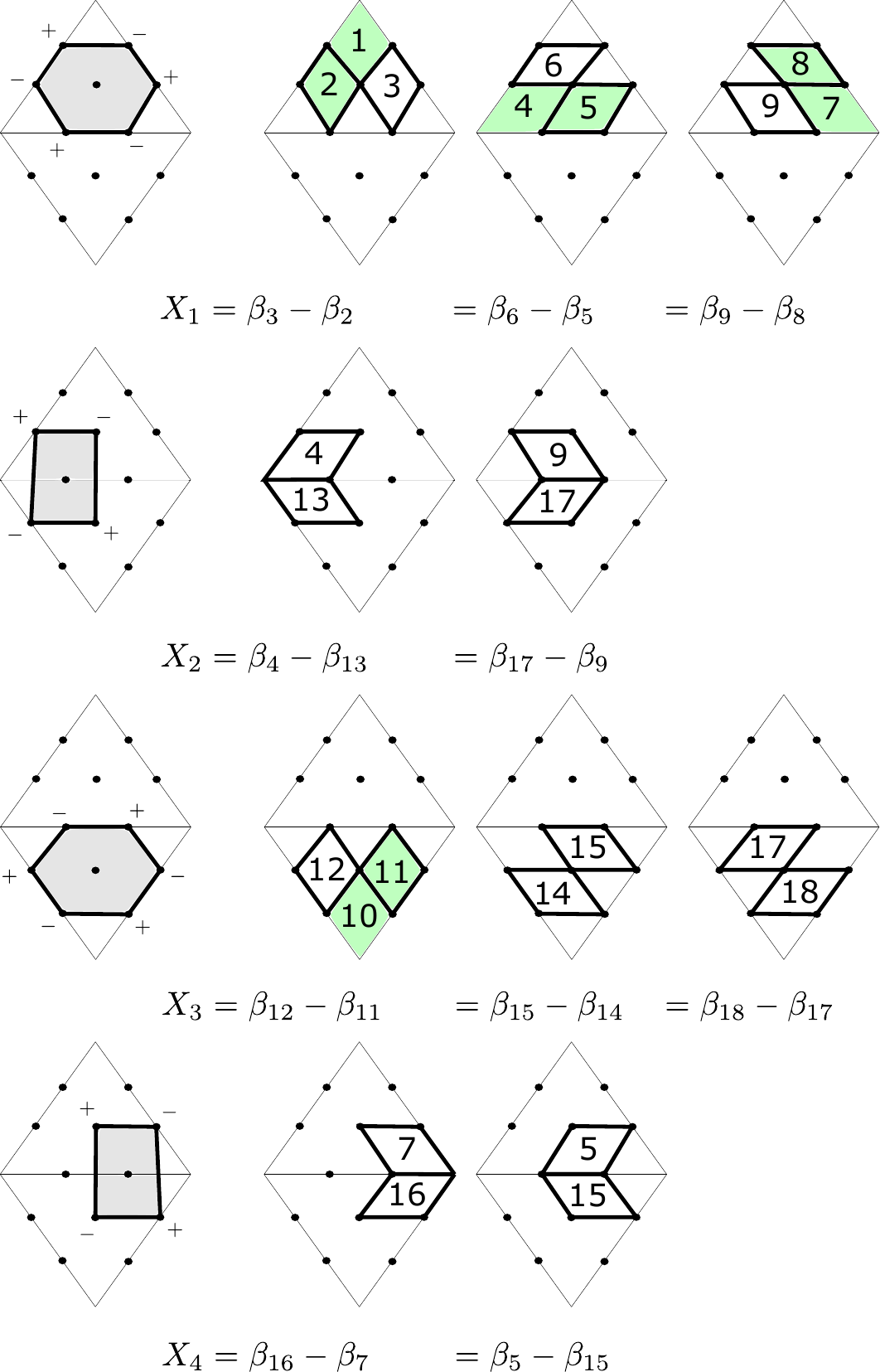}
\caption{     Shown are the 18 rhombus numbers $\{ \beta_i \}_{i=1,2,\dots,18}$ for the square, and the associated 4 tropical integer $\mathcal{X}$-coordinates $\{ X_i \}_{i=1,2,3,4}$.  The latter can be computed either as differences of rhombus numbers, or as alternating sums of the 12 positive tropical integer $\mathcal{A}$-coordinates $\{ A_i \}_{i=1,2,\dots,12}$ around the polygons displayed on the left.  The rhombi colored green are those involved in the first 8 coordinates of the isomorphism $\phi_\mathcal{T}: V_\mathcal{T} \to \mathbb{R}^{12}$.}
\label{fig:tropical-X-coords}
\end{figure}
	
\subsubsection{Sector decomposition of $\mathbb{R}^4$ via the isomorphism $\phi_\mathcal{T} \circ \theta_\mathcal{T}$}
\label{sssec:sector-decomposition-of-r4-via-the-ktgs-cone-for-the-square}
	   
Let $C \subset \mathbb{R}_+^8 \times \mathbb{R}^4$ be the cone defined in Definition \ref{def:isomorphic-cone}, which is isomorphic, via the isomorphism $\phi_\mathcal{T} \circ \theta_\mathcal{T}: \mathbb{R}^{12} \to \mathbb{R}^{8} \times \mathbb{R}^4$, to the completed KTGS cone $C_\mathcal{T} \subset \mathbb{R}_+^{12}$ for the triangulated square $(\Box, \mathcal{T})$.  
	
\begin{notation}
\label{not:relating-to-cone-lemmas}
Put $k=8$, $n=4$, $m=14$ ($=22-8$), and $p=42$ (compare Section \ref{ssec:technical-statements}).  
\end{notation}
	
For $j=1,2,\dots,m$, define cone points $x_j \in C$ by
\begin{equation*}
x_j = \phi_\mathcal{T}(\theta_\mathcal{T}(\Phi_\mathcal{T}(W^H_{k+j})))  \in C
\end{equation*}
where $\Phi_\mathcal{T}(W^H_{j^\prime})$ is the $j^\prime$-th Hilbert basis element for the triangulated square; see Figures \ref{figure:square1} and \ref{figure:square2}.  Note that the  points  $\{x_j\}_{i=1,2,\dots,m} \subset C$ are displayed explicitly in Section \ref{sssec:second-linear-isomorphism}.  
	
For $i=1,2,\dots,p$, define index sets $J_i \subset \{ 1,2,\dots,m \}$ of constant size $n$ as follows.  Given $i$, consider the topological type $\mathcal{Q}_i \subset \mathcal{W}_i \subset \mathcal{W}_\Box$ of the completed KTGS subcone $C_\mathcal{T}^i \subset \mathbb{R}_+^{12}$ (Definition \ref{def:topological type}).  By definition of the topological type $\mathcal{Q}_i$, there are four Hilbert basis webs $W^{H}_{k+j^{(i)}_1}, W^{H}_{k+j^{(i)}_2}, W^{H}_{k+j^{(i)}_3}, W^{H}_{k+j^{(i)}_4}$, where  the indices $j_r^{(i)} \in \{ 1,2,\dots,m \}$, such that  $\mathcal{Q}_i = \{ W^{H}_{k+j^{(i)}_r} \}_{r=1,2,3,4}$.  We then define the index set $J_i$ by
\begin{equation*}
J_i = \{ j^{(i)}_1, j^{(i)}_2, j^{(i)}_3, j^{(i)}_4 
\}  \subset \{ 1,2,\dots, m \}.  
\end{equation*}

\begin{definition}
\label{def:sectors-in-Rn}
Recalling Notation \ref{not:relating-to-cone-lemmas}: for each $i=1,2,\dots,42$, define subcones $D_i \subset \mathbb{R}^4$ by (compare Section \ref{ssec:technical-statements})
\begin{equation*}
D_i = \mathrm{span}_{\mathbb{R}_+}(\{ \pi_4(x_j); j \in J_i \} )  \subset \mathbb{R}^4.
\end{equation*} 
Here,  $\pi_4 : \mathbb{R}^8 \times \mathbb{R}^4 \to \mathbb{R}^4$ is the natural projection.  
	 
Just as for the subcones $C_\mathcal{T}^i \subset C_\mathcal{T}$, we call $\mathcal{Q}_i$ the \emph{topological type} of the subcone $D_i \subset \mathbb{R}^4$. 
\end{definition}

\begin{remark}
\label{rem:14-vectors-distinct}
Note, by the calculations of Section \ref{sssec:second-linear-isomorphism}, that the 14 vectors $\{ \pi_4(x_j) \}_{j=1,2,\dots,14} \subset \mathbb{R}^4$ are nonzero and distinct.
\end{remark}

\begin{prop}
\label{prop:subcones-sect-decomp-r4}
The subcones $D_i \subset \mathbb{R}^4$ are full sectors  forming a sector decomposition $\{ D_i \}_{i=1,2,\dots,42}$ of $\mathbb{R}^4$ \textup{(}Definition {\upshape\ref{def:sector-decomp}}\textup{)}.  
\end{prop}

\begin{proof}
Let us begin by giving some examples of how to describe the subcones $D_i$.  Specifically, we will describe those 9 subcones $D_{i_j} \subset \mathbb{R}^4 $ corresponding to the  topological types $\mathcal{Q}_{i_j} \subset \mathcal{W}_{i_j} \subset \mathcal{W}_\Box$, which in particular are subsets of the $(j)$ web families $\mathcal{W}_{i_j}$ displayed in Figure  \ref{figure:9cases}; see Notation \ref{not:web-families} and Remark \ref{rem:symmetry-groupings}.  

\begin{remark}
\label{rem:remark-about-rows}
In each of the 9 examples just below, note that the ordering of the rows of the matrix, chosen to match Figure \ref{figure:9cases}, does not affect the description of  the subcone $D_i \subset \mathbb{R}^4$.
\end{remark}

\begin{example*}[$i_1=29$]  For $\mathcal{Q}_{29}$, let us write the four vectors $\pi_4 (\phi_\mathcal{T} \circ \theta_\mathcal{T} \circ \Phi_{\mathcal{T}}(\mathcal{Q}_{29})) \subset \mathbb{R}^4$ in rows to form a $4\times 4$ matrix $M_{29}=\left(
\begin{matrix}
0&0&1&0 \\
0&1&0&0 \\ 
0&0&0&1 \\
1&0&0&0 
\end{matrix} 
\right)$. Then for real numbers $x,y,z,t \geq 0$, we get
\begin{equation*}
\begin{pmatrix}   x & y & z & t  \end{pmatrix}
M_{29}     =\begin{pmatrix}
t&y&x&z
\end{pmatrix}.
\end{equation*}
Thus $D_{29}=\{(X_1, X_2,X_3,X_4)\in \mathbb{R}_{+}\times \mathbb{R}_{+} \times \mathbb{R}_{+} \times \mathbb{R}_{+}\}$.  
\end{example*}

\begin{example*}[$i_2=30$]  Similarly, for $\mathcal{Q}_{30}$, writing the four vectors $\pi_4 (\phi_\mathcal{T}\circ \theta_\mathcal{T}\circ \Phi_{\mathcal{T}}(\mathcal{Q}_{30}))$ in rows, we get a $4\times 4$ matrix $M_{30}=\left(
\begin{matrix}
0&1&-1&0\\
0&1&0&0\\ 
0&0&0&1\\ 
1&0&0&0
\end{matrix} 
\right)$. Then for $x,y,z,t \geq 0$, we get
\begin{equation*}
\begin{pmatrix}   x & y & z & t  \end{pmatrix} 
M_{30}     =\begin{pmatrix}
t&x+y&-x&z
\end{pmatrix}.
\end{equation*}
Thus $D_{30}=\{(X_1, X_2,X_3,X_4)\in \mathbb{R}_{+}\times \mathbb{R}_{+} \times \mathbb{R}_{-} \times \mathbb{R}_{+} ; X_2+X_3\geq 0\}$.
\end{example*}

\begin{example*}[$i_3=42$]  For $\mathcal{Q}_{42}$, writing the four vectors $\pi_4 (\phi_\mathcal{T}\circ \theta_\mathcal{T}\circ \Phi_{\mathcal{T}}(\mathcal{Q}_{42}))$ in rows, we get a $4\times 4$ matrix $M_{42}=\left(
\begin{matrix}
0&1&-1&0 \\
1&0&-1&0\\ 
0&0&-1&0\\
0&0&0&-1 
\end{matrix} 
\right)$. Then for $x,y,z,t \geq 0$, we get
\begin{equation*}
\begin{pmatrix}   x & y & z & t  \end{pmatrix} 
M_{42}     =\begin{pmatrix}
y & x & -x -y-z & -t
\end{pmatrix}.
\end{equation*}
Thus $D_{42}=\{(X_1, X_2,X_3,X_4)\in \mathbb{R}_{+}\times \mathbb{R}_{+} \times \mathbb{R}_{-} \times \mathbb{R}_{-}; X_1 + X_2 + X_3 \leq 0\}$.
\end{example*}

\begin{example*}[$i_4=17$]  For $\mathcal{Q}_{17}$, writing the four vectors $\pi_4 (\phi_\mathcal{T}\circ \theta_\mathcal{T}\circ \Phi_{\mathcal{T}}(\mathcal{Q}_{17}))$ in rows, we get  a $4\times 4$ matrix $M_{17}=\left(
\begin{matrix}
1&-1&0&0\\
1&0&0&0\\ 
0&0&0&-1\\
0&0&1&-1 
\end{matrix}
 \right)$. Then for $x,y,z,t \geq 0$, we get
\begin{equation*}
\begin{pmatrix}   x & y & z & t  \end{pmatrix} 
M_{17}     =\begin{pmatrix}
x + y & -x & t & -z-t
\end{pmatrix}.
\end{equation*}
Thus $D_{17}=\{(X_1, X_2,X_3,X_4)\in \mathbb{R}_{+}\times \mathbb{R}_{-} \times \mathbb{R}_{+} \times \mathbb{R}_{-}; X_1 + X_2 \geq 0 \text{ and }  X_3 + X_4 \leq 0\}$.
\end{example*}

\begin{example*}[$i_5=5$]  For $\mathcal{Q}_{5}$, writing the four vectors $\pi_4 (\phi_\mathcal{T}\circ \theta_\mathcal{T}\circ \Phi_{\mathcal{T}}(\mathcal{Q}_{5}))$ in rows, we get  a $4\times 4$ matrix $M_{5}=\left(
\begin{matrix}
-1&0&1&0\\
-1&0&0&0\\ 
0&0&1&-1\\
0&1&0&0 
\end{matrix} 
\right)$. Then for $x,y,z,t \geq 0$, we get
\begin{equation*}
\begin{pmatrix}   x & y & z & t  \end{pmatrix} 
M_{5}     =\begin{pmatrix}
-x-y & t & x+z & -z
\end{pmatrix}.
\end{equation*}
Thus $D_{5}=\{(X_1, X_2,X_3,X_4)\in \mathbb{R}_{-}\times \mathbb{R}_{+} \times \mathbb{R}_{+} \times \mathbb{R}_{-};  -X_1 \geq X_3+X_4 \geq 0\}$.
\end{example*}

\begin{example*}[$i_6=6$]  For $\mathcal{Q}_{6}$, writing the four vectors $\pi_4 (\phi_\mathcal{T}\circ \theta_\mathcal{T}\circ \Phi_{\mathcal{T}}(\mathcal{Q}_{6}))$ in rows, we get  a $4\times 4$ matrix $M_{6}=\left(
\begin{matrix}
0&0&1&0\\
-1&0&1&0\\ 
-1&0&0&1\\ 
0&1&0&0
\end{matrix} 
\right)$. Then for $x,y,z,t \geq 0$, we get
\begin{equation*}
\begin{pmatrix}   x & y & z & t  \end{pmatrix} 
M_{6}     =\begin{pmatrix}
-y-z & t & x+y & z 
\end{pmatrix}.
\end{equation*}
Thus $D_{6}=\{(X_1, X_2,X_3,X_4)\in \mathbb{R}_{-}\times \mathbb{R}_{+} \times \mathbb{R}_{+} \times \mathbb{R}_{+}; -X_3 \leq X_1 + X_4 \leq 0\}$.
\end{example*}

\begin{example*}[$i_7=2$]  For $\mathcal{Q}_{2}$, writing the four vectors $\pi_4 (\phi_\mathcal{T}\circ \theta_\mathcal{T}\circ \Phi_{\mathcal{T}}(\mathcal{Q}_{2}))$ in rows, we get a $4\times 4$ matrix $M_{2}=\left(
\begin{matrix}
1&0&0&0\\
0&0&1&-1\\ 
0&0&0&-1\\ 
0&1&0&0
\end{matrix} 
\right)$. Then for $x,y,z,t \geq 0$, we get
\begin{equation*}
\begin{pmatrix}   x & y & z & t  \end{pmatrix} 
M_{2}     =\begin{pmatrix}
x & t & y & -y - z 
\end{pmatrix}.
\end{equation*}
Thus $D_{2}=\{(X_1, X_2,X_3,X_4)\in \mathbb{R}_{+}\times \mathbb{R}_{+} \times \mathbb{R}_{+} \times \mathbb{R}_{-}; X_3 + X_4 \leq 0\}$.
\end{example*}

\begin{example*}[$i_8=1$]  For $\mathcal{Q}_1$, writing the four vectors $\pi_4 (\phi_\mathcal{T}\circ \theta_\mathcal{T}\circ \Phi_{\mathcal{T}}(\mathcal{Q}_1))$ in rows, we get  a $4\times 4$ matrix $M_1=\left(
\begin{matrix}
-1&0&0&0\\
0&0&1&-1\\ 
0&0&0&-1\\ 
0&1&0&0
\end{matrix} 
\right)$. Then for $x,y,z,t \geq 0$, we get
\begin{equation*}
\begin{pmatrix}  x & y & z & t \end{pmatrix} 
M_1     =\begin{pmatrix}
-x & t & y & -y-z
\end{pmatrix}.
\end{equation*}
Thus $D_1=\{(X_1, X_2,X_3,X_4)\in \mathbb{R}_{-}\times \mathbb{R}_{+} \times \mathbb{R}_{+} \times \mathbb{R}_{-}; X_3+X_4\leq 0\}$.
\end{example*}

\begin{example*}[$i_9=33$]  For $\mathcal{Q}_{33}$, writing the four vectors $\pi_4 (\phi_\mathcal{T}\circ \theta_\mathcal{T}\circ \Phi_{\mathcal{T}}(\mathcal{Q}_{33}))$ in rows, we get  a $4\times 4$ matrix $M_{33}=\left(
\begin{matrix}
1&-1&0&0\\
1&0&0&0\\ 
0&0&1&0\\ 
0&0&1&-1
\end{matrix} 
\right)$. Then for $x,y,z,t \geq 0$, we get
\begin{equation*}
\begin{pmatrix}  x & y & z & t \end{pmatrix} 
M_{33}     =\begin{pmatrix}
x+y & -x & z+t&-t
\end{pmatrix}.
\end{equation*}
Thus $D_{33}=\{(X_1, X_2,X_3,X_4)\in \mathbb{R}_{+}\times \mathbb{R}_{-} \times \mathbb{R}_{+} \times \mathbb{R}_{-};  X_1 + X_2 \geq 0 \text{ and } X_3 + X_4 \geq 0  \}$.
\end{example*}

In the same way as the 9 examples just demonstrated, we compute   directly by hand the subcones $D_i \subset \mathbb{R}^4 $ for $i=1,2,\dots,42$ as follows:  
\begin{enumerate}[label=(\arabic*)]
\item  $D_1=\{(X_1, X_2,X_3,X_4)\in \mathbb{R}_{-}\times \mathbb{R}_{+} \times \mathbb{R}_{+} \times \mathbb{R}_{-}; X_3+X_4\leq 0\} $;
\item  $ D_2=\{(X_1, X_2,X_3,X_4)\in \mathbb{R}_{+}\times \mathbb{R}_{+} \times \mathbb{R}_{+} \times \mathbb{R}_{-}; X_3+X_4\leq 0\} $;
\item  $ D_3=\{(X_1, X_2,X_3,X_4)\in \mathbb{R}_{-}\times \mathbb{R}_{+} \times \mathbb{R}_{-} \times \mathbb{R}_{-}; X_2+X_3\geq 0\} $;
\item  $ D_4=\{(X_1, X_2,X_3,X_4)\in \mathbb{R}_{+}\times \mathbb{R}_{+} \times \mathbb{R}_{-} \times \mathbb{R}_{-}; X_2+X_3\geq 0\} $;
\item  $ D_5=\{(X_1, X_2,X_3,X_4)\in \mathbb{R}_{-}\times \mathbb{R}_{+} \times \mathbb{R}_{+} \times \mathbb{R}_{-}; -X_1\geq X_3+X_4\geq 0\} $;
\item  $ D_6=\{(X_1, X_2,X_3,X_4)\in \mathbb{R}_{-}\times \mathbb{R}_{+} \times \mathbb{R}_{+} \times \mathbb{R}_{+}; -X_3 \leq X_1+X_4\leq 0\} $;
\item  $ D_7=\{(X_1, X_2,X_3,X_4)\in \mathbb{R}_{-}\times \mathbb{R}_{-} \times \mathbb{R}_{+} \times \mathbb{R}_{-}; -X_1 \geq  X_3+X_4\geq 0\} $;
\item  $ D_8=\{(X_1, X_2,X_3,X_4)\in \mathbb{R}_{-}\times \mathbb{R}_{-} \times \mathbb{R}_{+} \times \mathbb{R}_{+};-X_3 \leq X_1+X_4\leq 0\} $;
\item  $ D_9=\{(X_1, X_2,X_3,X_4)\in \mathbb{R}_{+}\times \mathbb{R}_{-} \times \mathbb{R}_{-} \times \mathbb{R}_{+}; X_1+X_2\leq 0\} $;
\item  $ D_{10}=\{(X_1, X_2,X_3,X_4)\in \mathbb{R}_{-}\times \mathbb{R}_{-} \times \mathbb{R}_{-} \times \mathbb{R}_{+}; X_1+X_4\geq 0\} $;
\item  $ D_{11}=\{(X_1, X_2,X_3,X_4)\in \mathbb{R}_{+}\times \mathbb{R}_{-} \times \mathbb{R}_{+} \times \mathbb{R}_{+}; X_1+X_2\leq 0\} $;
\item  $ D_{12}=\{(X_1, X_2,X_3,X_4)\in \mathbb{R}_{-}\times \mathbb{R}_{-} \times \mathbb{R}_{+} \times \mathbb{R}_{+}; X_1+X_4\geq 0\} $;
\item  $ D_{13}=\{(X_1, X_2,X_3,X_4)\in \mathbb{R}_{+}\times \mathbb{R}_{-} \times \mathbb{R}_{-} \times \mathbb{R}_{+}; -X_3\geq X_1+X_2\geq 0\} $;
\item  $ D_{14}=\{(X_1, X_2,X_3,X_4)\in \mathbb{R}_{+}\times \mathbb{R}_{+} \times \mathbb{R}_{-} \times \mathbb{R}_{+}; -X_1\leq X_2+X_3\leq 0\} $;
\item  $ D_{15}=\{(X_1, X_2,X_3,X_4)\in \mathbb{R}_{+}\times \mathbb{R}_{-} \times \mathbb{R}_{-} \times \mathbb{R}_{-}; -X_3\geq X_1+X_2\geq 0\} $;
\item  $ D_{16}=\{(X_1, X_2,X_3,X_4)\in \mathbb{R}_{+}\times \mathbb{R}_{+} \times \mathbb{R}_{-} \times \mathbb{R}_{-}; -X_1\leq  X_2+X_3\leq 0\} $;
\item  $ D_{17}=\{(X_1, X_2,X_3,X_4)\in \mathbb{R}_{+}\times \mathbb{R}_{-} \times \mathbb{R}_{+} \times \mathbb{R}_{-}; X_1+X_2\geq 0, X_3+X_4\leq 0\} $;
\item  $ D_{18}=\{(X_1, X_2,X_3,X_4)\in \mathbb{R}_{-}\times \mathbb{R}_{+} \times \mathbb{R}_{-} \times \mathbb{R}_{-}; X_2+X_3\leq 0\} $;
\item  $ D_{19}=\{(X_1, X_2,X_3,X_4)\in \mathbb{R}_{-}\times \mathbb{R}_{-} \times \mathbb{R}_{-} \times \mathbb{R}_{+}; X_1+X_4\leq 0\} $;
\item  $ D_{20}=\{(X_1, X_2,X_3,X_4)\in \mathbb{R}_{+}\times \mathbb{R}_{-} \times \mathbb{R}_{+} \times \mathbb{R}_{-}; X_1+X_2\leq 0, X_3+X_4\geq 0\} $;
\item  $ D_{21}=\{(X_1, X_2,X_3,X_4)\in \mathbb{R}_{+}\times \mathbb{R}_{-} \times \mathbb{R}_{+} \times \mathbb{R}_{-}; X_1+X_2\leq 0, X_3+X_4\leq 0\} $;
\item  $ D_{22}=\{(X_1, X_2,X_3,X_4)\in \mathbb{R}_{+}\times \mathbb{R}_{-} \times \mathbb{R}_{-} \times \mathbb{R}_{-}; X_1+X_2\leq 0\} $;
\item  $ D_{23}=\{(X_1, X_2,X_3,X_4)\in \mathbb{R}_{-}\times \mathbb{R}_{-} \times \mathbb{R}_{+} \times \mathbb{R}_{-}; X_3+X_4\leq 0\} $;
\item  $ D_{24}=\{(X_1, X_2,X_3,X_4)\in \mathbb{R}_{-}\times \mathbb{R}_{-} \times \mathbb{R}_{-} \times \mathbb{R}_{-}\} $;
\item  $ D_{25}=\{(X_1, X_2,X_3,X_4)\in \mathbb{R}_{+}\times \mathbb{R}_{-} \times \mathbb{R}_{+} \times \mathbb{R}_{+}; X_1+X_2\geq 0\} $;
\item  $ D_{26}=\{(X_1, X_2,X_3,X_4)\in \mathbb{R}_{-}\times \mathbb{R}_{+} \times \mathbb{R}_{-} \times \mathbb{R}_{+}; X_1+X_4\geq 0, X_2+X_3\leq 0\} $;
\item  $ D_{27}=\{(X_1, X_2,X_3,X_4)\in \mathbb{R}_{-}\times \mathbb{R}_{+} \times \mathbb{R}_{-} \times \mathbb{R}_{+}; X_1+X_4\leq 0, X_2+X_3\geq 0\} $;
\item  $ D_{28}=\{(X_1, X_2,X_3,X_4)\in \mathbb{R}_{+}\times \mathbb{R}_{+} \times \mathbb{R}_{+} \times \mathbb{R}_{-}; X_3+X_4\geq 0\} $;
\item  $ D_{29}=\{(X_1, X_2,X_3,X_4)\in \mathbb{R}_{+}\times \mathbb{R}_{+} \times \mathbb{R}_{+} \times \mathbb{R}_{+}\} $;
\item  $ D_{30}=\{(X_1, X_2,X_3,X_4)\in \mathbb{R}_{+}\times \mathbb{R}_{+} \times \mathbb{R}_{-} \times \mathbb{R}_{+}; X_2+X_3\geq 0\} $;
\item  $ D_{31}=\{(X_1, X_2,X_3,X_4)\in \mathbb{R}_{-}\times \mathbb{R}_{+} \times \mathbb{R}_{+} \times \mathbb{R}_{+}; X_1+X_4\geq 0\} $;
\item  $ D_{32}=\{(X_1, X_2,X_3,X_4)\in \mathbb{R}_{-}\times \mathbb{R}_{+} \times \mathbb{R}_{-} \times \mathbb{R}_{+}; X_1+X_4\geq 0, X_2+X_3\geq 0\} $;
\item  $ D_{33}=\{(X_1, X_2,X_3,X_4)\in \mathbb{R}_{+}\times \mathbb{R}_{-} \times \mathbb{R}_{+} \times \mathbb{R}_{-}; X_1+X_2\geq 0, X_3+X_4\geq 0\} $;
\item  $ D_{34}=\{(X_1, X_2,X_3,X_4)\in \mathbb{R}_{-}\times \mathbb{R}_{+} \times \mathbb{R}_{-} \times \mathbb{R}_{+}; X_1+X_4\leq 0, X_2+X_3\leq 0\} $;
\item  $ D_{35}=\{(X_1, X_2,X_3,X_4)\in \mathbb{R}_{-}\times \mathbb{R}_{+} \times \mathbb{R}_{+} \times \mathbb{R}_{+}; X_1+X_3+X_4\leq 0\} $;
\item  $ D_{36}=\{(X_1, X_2,X_3,X_4)\in \mathbb{R}_{-}\times \mathbb{R}_{+} \times \mathbb{R}_{+} \times \mathbb{R}_{-}; X_1+X_3+X_4\geq 0\} $;
\item  $ D_{37}=\{(X_1, X_2,X_3,X_4)\in \mathbb{R}_{-}\times \mathbb{R}_{-} \times \mathbb{R}_{+} \times \mathbb{R}_{+}; X_1+X_3+X_4\leq 0\} $;
\item  $ D_{38}=\{(X_1, X_2,X_3,X_4)\in \mathbb{R}_{-}\times \mathbb{R}_{-} \times \mathbb{R}_{+} \times \mathbb{R}_{-}; X_1+X_3+X_4\geq 0\} $;
\item  $ D_{39}=\{(X_1, X_2,X_3,X_4)\in \mathbb{R}_{+}\times \mathbb{R}_{-} \times \mathbb{R}_{-} \times \mathbb{R}_{+}; X_1+X_2+X_3\geq 0\} $;
\item  $ D_{40}=\{(X_1, X_2,X_3,X_4)\in \mathbb{R}_{+}\times \mathbb{R}_{+} \times \mathbb{R}_{-} \times \mathbb{R}_{+}; X_1+X_2+X_3\leq 0\} $;
\item  $ D_{41}=\{(X_1, X_2,X_3,X_4)\in \mathbb{R}_{+}\times \mathbb{R}_{-} \times \mathbb{R}_{-} \times \mathbb{R}_{-}; X_1+X_2+X_3\geq 0\} $;
\item  $D_{42}=\{(X_1, X_2,X_3,X_4)\in \mathbb{R}_{+}\times \mathbb{R}_{+} \times \mathbb{R}_{-} \times \mathbb{R}_{-}; X_1+X_2+X_3\leq 0\} $.
\end{enumerate}

In particular, each subcone $D_i \subset \mathbb{R}^4 $ has dimension 4, since its explicit description via inequalities shows that it has nonempty interior in $\mathbb{R}^4$.  Equivalently, one can check  directly by hand that the corresponding $4 \times 4$ matrix, such as in the 9 examples above, has rank 4.  

Since, by Definition \ref{def:sectors-in-Rn}, the  subcones $D_i \subset \mathbb{R}^4 $ are generated by 4 elements, we gather that each $D_i$ is a full subsector of $\mathbb{R}^4$.  

One checks directly by hand  that the 16 orthants $\mathbb{R}_\pm \times \mathbb{R}_\pm \times \mathbb{R}_\pm \times \mathbb{R}_\pm \subset \mathbb{R}^4$ decompose as:
\begin{enumerate}[label=(\arabic*)]
\item  $\mathbb{R}_{-}\times \mathbb{R}_{-} \times \mathbb{R}_{-} \times \mathbb{R}_{-} = D_{24} $;
\item  $ \mathbb{R}_{+}\times \mathbb{R}_{+} \times \mathbb{R}_{+} \times \mathbb{R}_{+} = D_{29} $;
\item  $ \mathbb{R}_{+}\times \mathbb{R}_{+} \times \mathbb{R}_{+} \times \mathbb{R}_{-} = D_{2}\cup D_{28} $;
\item  $\mathbb{R}_{-}\times \mathbb{R}_{+} \times \mathbb{R}_{-} \times \mathbb{R}_{-} = D_{3}\cup D_{18} $;
\item  $ \mathbb{R}_{-}\times \mathbb{R}_{+} \times \mathbb{R}_{+} \times \mathbb{R}_{+} = D_{6}\cup D_{31}\cup D_{35} $;
\item  $ \mathbb{R}_{-}\times \mathbb{R}_{-} \times \mathbb{R}_{+} \times \mathbb{R}_{-} = D_{7}\cup D_{23}\cup D_{38} $;
\item  $ \mathbb{R}_{-}\times \mathbb{R}_{-} \times \mathbb{R}_{-} \times \mathbb{R}_{+} = D_{10}\cup D_{19} $;
\item  $ \mathbb{R}_{+}\times \mathbb{R}_{-} \times \mathbb{R}_{+} \times \mathbb{R}_{+} = D_{11}\cup D_{25} $;
\item  $ \mathbb{R}_{+}\times \mathbb{R}_{+} \times \mathbb{R}_{-} \times \mathbb{R}_{+} = D_{14}\cup D_{30}\cup D_{40} $;
\item  $ \mathbb{R}_{+}\times \mathbb{R}_{-} \times \mathbb{R}_{-} \times \mathbb{R}_{-} = D_{15}\cup D_{22}\cup D_{41} $;
\item  $ \mathbb{R}_{-}\times \mathbb{R}_{+} \times \mathbb{R}_{+} \times \mathbb{R}_{-} = D_{1}\cup D_{5}\cup D_{36} $;
\item  $ \mathbb{R}_{+}\times \mathbb{R}_{+} \times \mathbb{R}_{-} \times \mathbb{R}_{-} = D_{4}\cup D_{16}\cup D_{42} $;
\item  $ \mathbb{R}_{-}\times \mathbb{R}_{-} \times \mathbb{R}_{+} \times \mathbb{R}_{+} = D_{8}\cup D_{12}\cup D_{37} $;
\item  $ \mathbb{R}_{+}\times \mathbb{R}_{-} \times \mathbb{R}_{-} \times \mathbb{R}_{+} = D_{9}\cup D_{13}\cup D_{39} $;
\item  $ \mathbb{R}_{+}\times \mathbb{R}_{-} \times \mathbb{R}_{+} \times \mathbb{R}_{-} = D_{17}\cup D_{20}\cup D_{21}\cup D_{33} $;
\item  $\mathbb{R}_{-}\times \mathbb{R}_{+} \times \mathbb{R}_{-} \times \mathbb{R}_{+} = D_{26}\cup D_{27}\cup D_{32} \cup D_{34} $.
\end{enumerate}

It follows that $\mathbb{R}^4 = \cup_{i=1}^{42} D_i$.  It remains to show that $D_i \cap D_\ell$ has empty interior for all pairwise-distinct $i, \ell$.  Since this is true if $D_i$ and $D_\ell$ lie in different orthants, we only need to check those pairs $D_i$, $D_\ell$ lying in the same orthant.  

This is  done directly by hand; however, the cases fall into only four types.  First, for orthants 1,2:  There is nothing to check.  Second, for orthants 3, 4, 7, 8:  In 3, say, the inequalities $X_3 + X_4 \leq 0$ and $X_3 + X_4 \geq 0$ have codimension 1 intersection even when defined on all of $\mathbb{R}^4$, so they do as well when restricted to the orthant.  The other cases go the same.  Third, for orthants 15, 16:  This is similar to the second type.  Lastly, for orthants 5, 6, 9, 10, 11, 12, 13, 14:  In 5, say, the sectors $D_{6}$ and $D_{31}$ similarly have codimension 1 intersection, as do the sectors $D_6$ and $D_{35}$.  This is also true for $D_{31} ( X_1+X_4\geq 0 )$ and $D_{35} ( X_1+X_3+X_4\leq 0 )$ so long as $D_{31} \cap D_{35}$ implies $X_3=0$, which it does, since it implies $X_3 \leq -X_1 - X_4 \leq 0$ whereas we are restricted to the orthant $\mathbb{R}_{-}\times \mathbb{R}_{+} \times \mathbb{R}_{+} \times \mathbb{R}_{+}$.  The other cases go the same.  
\end{proof}

We now analyze the walls (Definition \ref{def:sector-decomp}) in the sector decomposition $\{ D_i \}_{i=1,2,\dots,42}$ of $\mathbb{R}^4$.  Recall (Definition \ref{def:sectors-in-Rn}) that $\mathcal{Q}_i$ is called the topological type of the sector $D_i$.  

\begin{prop}
\label{prop:walls-for-Dis}
The third item of Theorem {\upshape\ref{theorem:decomp}} holds word-for-word, except with  $C_\mathcal{T}^i$ replaced by $D_i$, and replacing \eqref{eq:theorem-equation} by \eqref{eq:anotherequationaboutR4}.  
	
In particular, $D_i \cap D_\ell$ is a wall if and only if the intersection $J_i \cap J_\ell$ of their corresponding index sets has $3$ elements \textup{(}see the beginning of this sub-subsection, {\upshape Section \ref{sssec:sector-decomposition-of-r4-via-the-ktgs-cone-for-the-square}}\textup{)}. In this case,
\begin{equation*}\tag{l}
\label{eq:anotherequationaboutR4}
D_i \cap D_\ell = \mathrm{span}_{\mathbb{R}_+}
( 
\{ \pi_4(x_j);  j \in J_i \cap J_\ell \} 
)
  \subset \mathbb{R}^4. 
\end{equation*}
\end{prop}

\begin{proof}
Any wall $D_i \cap D_\ell$ must,  by definition, be $3$ dimensional.   This restricts which indices $\{ i, \ell \}$ can yield walls.  Through a  direct  by hand check, using   the   explicit  description   by inequalities   of the sector decomposition   $\{ D_i \}_i $ as in the proof of Proposition \ref{prop:subcones-sect-decomp-r4},  one verifies that a necessary condition for $D_i \cap D_\ell$ to be a wall is for $D_i$ and $D_\ell$ to be connected by an edge in the graph $\mathcal{G}$ depicted in Figure \ref{figure:wallscross}.   We  show this is also a sufficient condition.  

More precisely,      the goal is to show, for any two sectors $D_i$ and $D_\ell$ connected by an edge in $\mathcal{G}$, that $J_i \cap J_\ell$ has 3 elements and \eqref{eq:anotherequationaboutR4} holds.  In particular, $D_i \cap D_\ell$ is a cone of dimension $3$.  Note the inclusion $\supset$ in \eqref{eq:anotherequationaboutR4} holds  automatically; see  Definition \ref{def:sectors-in-Rn}.  

We checked this  directly   by hand.  There are two types of calculations,  depending on whether   the sectors are in different orthants or the same orthant. 

As an example where the sectors are in different orthants: We demonstrate this for $D_{29}$ and $D_{30}$, which were computed in detail in the proof of Proposition \ref{prop:subcones-sect-decomp-r4}.  There, one sees that three rows of the corresponding $4 \times 4$ matrix $M_{29}$ appear as rows in the matrix $M_{30}$ (recall  also  Remark \ref{rem:remark-about-rows}).  This means that $J_i \cap J_\ell$ has 3 elements; see Remark \ref{rem:14-vectors-distinct}.    Note that the row in $M_{29}$  that is not in $M_{30}$ corresponds to the variable $x$, and the row in $M_{30}$  that is not in $M_{29}$ corresponds to the variable $x^\prime$.  The inclusion $\subset$ in \eqref{eq:anotherequationaboutR4} is true since
\begin{equation*}
(t, y, x, z) = (t^\prime, x^\prime+y^\prime, -x^\prime, z^\prime)	 \in \mathbb{R}^4 \,\, ( x,y,z,t,x^\prime,y^\prime,z^\prime,t^\prime \geq 0 )
\end{equation*}
implies $x=x^\prime=0$.  The other different-orthant cases are similar.  

As an example where the sectors are in the same orthant: We demonstrate this for $D_5$ and $D_1$, which were also computed in detail in the proof of Proposition \ref{prop:subcones-sect-decomp-r4}.  There, one sees that three rows of the corresponding $4 \times 4$ matrix $M_{5}$ appear as rows in the matrix $M_{1}$ (recall  also Remark \ref{rem:remark-about-rows}).  This means that $J_i \cap J_\ell$ has 3 elements; see Remark \ref{rem:14-vectors-distinct}.  The row in $M_{5}$ not in $M_{1}$ corresponds to the variable $x$, and the row in $M_{1}$ not in $M_{5}$ corresponds to the variable $z^\prime$.  The inclusion $\subset$ in \eqref{eq:anotherequationaboutR4} is true since
\begin{equation*}
(-x-y, t, x+z, -z) = (-x^\prime, t^\prime, y^\prime, -y^\prime-z^\prime) \in \mathbb{R}^4 \,\, ( x,y,z,t,x^\prime,y^\prime,z^\prime,t^\prime \geq 0 )
\end{equation*}
implies, by adding the third and fourth  entries, that $x=-z^\prime$ hence $x=z^\prime=0$.  The other same-orthant cases are similar.  

We gather $D_i \cap D_\ell$ is a wall if and only if $D_i$ and $D_\ell$ are connected by an edge of the graph $\mathcal{G}$,   in which case $J_i \cap J_\ell$ has 3 elements.  In particular, since $\mathcal{G}$ is 4-valent, the last paragraph of the third item of Theorem \ref{theorem:decomp} holds (again, with $D_i$ in place of $C_\mathcal{T}^i$).  

To finish justifying the second paragraph of Proposition \ref{prop:walls-for-Dis} (equivalently, the first paragraph of the third item of Theorem \ref{theorem:decomp}, appropriately substituted), we need to show that if $J_i \cap J_\ell$ has 3 elements (equivalently, $\mathcal{Q}_i \cap \mathcal{Q}_\ell$ has 3 elements), then $D_i \cap D_\ell$ is a wall.

So far, we have exhibited, for each $i$, 4 topological types $\mathcal{Q}_\ell$ such that $\mathcal{Q}_i \cap \mathcal{Q}_\ell$ has 3 elements,  all corresponding to walls $D_i \cap D_\ell$.  We thus need to show there are no more indices $\ell$ such that $\mathcal{Q}_i \cap \mathcal{Q}_\ell$ has 3 elements.  For this, it  suffices to establish the second paragraph of the third item of Theorem \ref{theorem:decomp}, which is a purely topological  statement about webs in good position on the triangulated square; compare   Observation \ref{lemma:bigon} and Proposition \ref{lem:42families}.  We checked this directly by hand; compare Example \ref{ex:example-of-web-mutation}.  
\end{proof}

\subsubsection{Finishing}
\label{sssec:proof-of-main-theorem-2}
	  
We are now prepared to prove the theorem.  

\begin{proof}[Proof of Theorem {\upshape \ref{theorem:decomp}}]
	
Let $C \subset \mathbb{R}_+^8 \times \mathbb{R}^4$ be the cone isomorphic to the completed KTGS cone $C_\mathcal{T} \subset \mathbb{R}_+^{12}$ via  the linear isomorphism $\phi_\mathcal{T} \circ \theta_\mathcal{T} : \mathbb{R}^{12} \to \mathbb{R}^{8} \times \mathbb{R}^4$; see Definition \ref{def:isomorphic-cone}.  
	
Recall also Notation \ref{not:relating-to-cone-lemmas} from the discussion at the beginning of Section \ref{sssec:sector-decomposition-of-r4-via-the-ktgs-cone-for-the-square}, which should help  with  comparing the general lemmas of Section \ref{ssec:technical-statements}  to  the current application.  
	
By the explicit calculation of the Hilbert basis elements $\phi_\mathcal{T} \circ \theta_\mathcal{T} \circ \Phi_\mathcal{T}(W_i^H) \in C$ for $i=1,2,\dots,22$ in Section \ref{sssec:second-linear-isomorphism}, together with Proposition \ref{prop:subcones-sect-decomp-r4}, we see that $C \subset \mathbb{R}_+^8 \times \mathbb{R}^4$ satisfies the hypotheses of Lemma \ref{lem:rank-lemma}.   Indeed, $D_i \subset \pi_4(C)$ by definition, for each $i$. Therefore, $C$ is 12 dimensional, so the isomorphic  completed KTGS cone $C_\mathcal{T}$ is  12 dimensional as well.  This establishes the first item of Theorem \ref{theorem:decomp}.
		
Let us prove that $C_\mathcal{T} = \cup_{i=1}^{42} C_\mathcal{T}^i \subset \mathbb{R}_+^{12}$; see Definition \ref{def:topological type}.  This follows by Theorem \ref{thm:main-theorem}, Proposition \ref{lem:42families}, and Lemma  \ref{lem:actual-geometry}.  Indeed, by Theorem \ref{thm:main-theorem}, every point in the KTGS positive integer cone $\mathcal{C}_\mathcal{T} \subset \mathbb{Z}_+^{12}$ is equal to $\Phi_\mathcal{T}(W)$ for some reduced web $W \in \mathcal{W}_\Box$ in the square.  By Proposition \ref{lem:42families}, the web $W$ is an element of one of the 42 web families:  $W \in \mathcal{W}_i \subset \mathcal{W}_\Box$.  By Observation \ref{obs:families-and-their-cones}, we  have that $\Phi_\mathcal{T}(W) \in \mathcal{C}_\mathcal{T}^i$.  We gather that $\mathcal{C}_\mathcal{T} = \cup_{i=1}^{42} \mathcal{C}_\mathcal{T}^i \subset \mathbb{Z}_+^{12}$.   Also, the submonoids $\mathcal{C}_\mathcal{T}^i \subset \mathcal{C}_\mathcal{T} $ are finitely generated by Definition \ref{def:topological type}.  We conclude by Lemma \ref{lem:actual-geometry} that  we have the equality $C_\mathcal{T} = \cup_{i=1}^{42} C_\mathcal{T}^i \subset \mathbb{R}_+^{12}$  of completions, as desired.  
	
We return to the isomorphic cone $C \subset \mathbb{R}_+^8 \times \mathbb{R}^4$.  Let $\{ x_j \}_{j=1,2,\dots,14}$ and $\{ J_i \}_{i=1,2,\dots,42}$ be defined as in the beginning of Section \ref{sssec:sector-decomposition-of-r4-via-the-ktgs-cone-for-the-square}.  For $i=1,2,\dots,42$, let the subcones $C_i \subset C$ be defined as in the  statement  of Lemma \ref{lem:second-cone-lemma}.  Equivalently, the subcone $C_i = \phi_\mathcal{T} \circ \theta_\mathcal{T}(C_\mathcal{T}^i) \subset \mathbb{R}_+^8 \times \mathbb{R}^4$ is the isomorphic counterpart to the subcone $C_\mathcal{T}^i \subset C_\mathcal{T} \subset \mathbb{R}_+^{12}$.  It follows by the previous paragraph that $C = \cup_{i=1}^{42} C_i \subset \mathbb{R}_+^8 \times \mathbb{R}^4$ in the isomorphic cone $C$.  By this, together with another application of Proposition \ref{prop:subcones-sect-decomp-r4}, we see that the hypotheses of Lemma \ref{lem:second-cone-lemma} are satisfied.  Therefore, by Lemma \ref{lem:second-cone-lemma}, we obtain the second item of Theorem \ref{theorem:decomp}, except with $C_\mathcal{T}$ and $C_\mathcal{T}^i$ replaced by $C$ and $C_i$, respectively.  Since this property is preserved by linear isomorphisms, we conclude the second item of Theorem \ref{theorem:decomp} as stated.  
	
Lastly, by Proposition \ref{prop:walls-for-Dis}, in particular \eqref{eq:anotherequationaboutR4}, the hypothesis of Lemma \ref{lem:third-cone-lemma} is satisfied.  Therefore,  by Lemma \ref{lem:third-cone-lemma}, the set $\{ \text{walls of } \{ C_i \}_i \text { in } C \}$ is in one-to-one correspondence with the set $\{ \text{walls of } \{ D_i \}_i \text { in } \mathbb{R}^4 \}$  in the obvious way by the projection $\pi_4$.  Moreover, a given wall $C_i \cap C_\ell$ can be computed by \eqref{eq:technical-equation-intersection} in Lemma \ref{lem:third-cone-lemma}.  We conclude by the  remainder  of Proposition \ref{prop:walls-for-Dis} that the first and third paragraphs of the third item of Theorem \ref{theorem:decomp} are valid, except with $C_\mathcal{T}$, $C_\mathcal{T}^i$, and \eqref{eq:theorem-equation} replaced by $C$, $C_i$, and \eqref{eq:technical-equation-intersection}, respectively.  Since the inverse of the linear isomorphism $\phi_\mathcal{T} \circ \theta_\mathcal{T}$ preserves these properties, and maps \eqref{eq:technical-equation-intersection} to \eqref{eq:theorem-equation} (see the beginning of Section \ref{sssec:sector-decomposition-of-r4-via-the-ktgs-cone-for-the-square}),  we conclude the first and third paragraphs of the second item of Theorem \ref{theorem:decomp} as stated.  The second paragraph is a purely topological statement about  webs in the square, and was already established during the proof of Proposition \ref{prop:walls-for-Dis}. 
\end{proof}

The following consequence is immediate from the proof of Theorem \ref{theorem:decomp}.

\begin{cor}
\label{cor:canonical-map-surjection}
The function 
\begin{equation*}
\pi_4 \circ \phi_\mathcal{T} \circ \theta_\mathcal{T} : C_\mathcal{T} \twoheadrightarrow \mathbb{R}^4
\end{equation*}
is a surjection from  the completed KTGS cone $C_\mathcal{T} \subset \mathbb{R}_+^{12}$ \textup{(}in fact, from a `$4$ dimensional'  proper subset of $C_\mathcal{T}$\textup{)} onto $\mathbb{R}^4$.  \qed
\end{cor}

Recall the notion of a cornerless web $W=W_c$ in the square; see Definition \ref{def:cornerless-webs-in-the-square}.  Let $\mathcal{W}_\Box^c \subset \mathcal{W}_\Box$ denote the set of cornerless webs up to equivalence.  Note for each $i=1,2,\dots,42$ that $\mathcal{Q}_i \subset \mathcal{W}_i \cap \mathcal{W}_\Box^c$.  

Consider also the function $\pi_4 \circ \phi_\mathcal{T} \circ \theta_\mathcal{T} \circ \Phi_\mathcal{T} : \mathcal{W}_\Box \to \mathbb{Z}^4 \subset \mathbb{R}^4$ defined on $\mathcal{W}_\Box$.  See for example the nine computations at the beginning of the proof of Proposition \ref{prop:subcones-sect-decomp-r4}.  

Another consequence of the proof of Theorem \ref{theorem:decomp} is:  

\begin{cor}
\label{cor:restricted-function}
The restricted function
\begin{equation*}
\pi_4 \circ \phi_\mathcal{T} \circ \theta_\mathcal{T} \circ \Phi_\mathcal{T} : \mathcal{W}_\Box^c \twoheadrightarrow \mathbb{Z}^4
\end{equation*}
restricted to the cornerless webs $\mathcal{W}_\Box^c \subset \mathcal{W}_\Box$ is a surjection onto the integer lattice $\mathbb{Z}^4 \subset \mathbb{R}^4$.  
	
In particular, the function $\pi_4 \circ \phi_\mathcal{T} \circ \theta_\mathcal{T}$ from Corollary {\upshape\ref{cor:canonical-map-surjection}} maps \textup{(}a `$4$ dimensional'  proper subset of\textup{)} the KTGS cone $\mathcal{C}_\mathcal{T} \subset \mathbb{Z}_+^{12}$ surjectively onto $\mathbb{Z}^4$.  
\end{cor}

\begin{proof}
We know that $\mathbb{R}^4=\cup_{i=1}^{42} D_i$ and $D_i = \pi_4 \circ \phi_\mathcal{T} \circ \theta_\mathcal{T}(C_\mathcal{T}^i)$ where $C_\mathcal{T}^i \supset \mathcal{C}_\mathcal{T}^i$ (Definition \ref{def:topological type}).  We also know that the cone points $\Phi_\mathcal{T}(W^H_j)$ in $\mathcal{C}_\mathcal{T}^i$ for $j=1,2,\dots,8$, corresponding to the 8 corner arcs in the square, are sent by $\phi_\mathcal{T} \circ \theta_\mathcal{T}$ to the first 8 standard basis elements $e_j$ of $\mathbb{R}_+^8 \times \mathbb{R}^4$.  We gather $\pi_4 \circ \phi_\mathcal{T} \circ \theta_\mathcal{T}$ is still a surjection onto $D_i$ when restricted to the subset
\begin{equation*}
C_\mathcal{T}^{\prime i} := \mathrm{span}_{\mathbb{R}_+}
(  \{  \Phi_\mathcal{T}(W^H);  W^H \in \mathcal{Q}_i 
\}  )
  \subset C_\mathcal{T}^i
  \subset  \mathbb{R}_+^{12}.
\end{equation*}
Note also that (similar to Observation \ref{obs:families-and-their-cones})
\begin{equation*}
\Phi_\mathcal{T}(\mathcal{W}_i \cap \mathcal{W}_\Box^c) 
=
\mathrm{span}_{\mathbb{Z}_+}
(  \{  \Phi_\mathcal{T}(W^H);  W^H \in \mathcal{Q}_i 
\}  )
\subset C_\mathcal{T}^{\prime i}.
\end{equation*}  
It thus suffices to show: for any $c \in C_\mathcal{T}^{\prime i}$, if $\pi_4 \circ \phi_\mathcal{T} \circ \theta_\mathcal{T}(c) \in \mathbb{Z}^4 \cap D_i$, then $c \in \Phi_\mathcal{T}(\mathcal{W}_i \cap \mathcal{W}_\Box^c)$. 

Once again, we work in the isomorphic cone $C_i=\phi_\mathcal{T} \circ \theta_\mathcal{T}(C_\mathcal{T}^i) \subset \mathbb{R}_+^8 \times \mathbb{R}^4$, which projects to $D_i$ by $\pi_4$.  Put (see the beginning of Section \ref{sssec:sector-decomposition-of-r4-via-the-ktgs-cone-for-the-square})
\begin{equation*}
C_i^\prime := \phi_\mathcal{T} \circ \theta_\mathcal{T}(C_\mathcal{T}^{\prime i})=\mathrm{span}_{\mathbb{R}_+}( \{ x_j;  j \in J_i \} )
  \subset  C_i   \subset  \mathbb{R}_+^8 \times \mathbb{R}^4
\end{equation*}
and note also that
\begin{equation*}
\phi_\mathcal{T} \circ \theta_\mathcal{T} \circ \Phi_\mathcal{T}(\mathcal{W}_i \cap \mathcal{W}_\Box^c) =\mathrm{span}_{\mathbb{Z}_+}( \{ x_j;  j \in J_i \} )
  \subset  C_i^\prime.
\end{equation*}
The above property is then equivalent to showing:  for any $c \in C_i^{\prime}$ such that $\pi_4(c) \in  \mathbb{Z}^4 \cap D_i$, we have $c \in \phi_\mathcal{T} \circ \theta_\mathcal{T} \circ \Phi_\mathcal{T}(\mathcal{W}_i \cap \mathcal{W}_\Box^c)$; that is, if such a $c \in C_i^{\prime}$ is written $c=x x_{1^{(i)}} + y x_{2^{(i)}} + z x_{3^{(i)}} + t x_{4^{(i)}}$ for $x,y,z,t \geq 0$ (see the beginning of Section \ref{sssec:sector-decomposition-of-r4-via-the-ktgs-cone-for-the-square}),  we want to show $x, y, z, t \in \mathbb{Z}$.  

This is accomplished through a direct by hand check, taking advantage of the explicit description of the sectors $D_i$ provided in Section \ref{sssec:sector-decomposition-of-r4-via-the-ktgs-cone-for-the-square}.  As before, although there are 42 cases, these fall into only five types, each represented among the 9 examples demonstrated in the proof of Proposition \ref{prop:subcones-sect-decomp-r4}:  Type 1 corresponds to $i_1$; Type 2 corresponds to $i_2, i_7, i_8$; Type 3 corresponds to $i_3$; Type 4 corresponds to $i_4, i_9$; and Type 5 corresponds to $i_5, i_6$.  We will only demonstrate  the most nontrivial case, Type 5 (for $i_5$, say); the other cases are similar.  

So consider $i_5=5$, and assume $\pi_4(c) =x \pi_4(x_{1^{(5)}}) + y \pi_4(x_{2^{(5)}}) + z \pi_4(x_{3^{(5)}}) + t \pi_4(x_{4^{(5)}}) \in D_i$ is, in addition, in $\mathbb{Z}^4$ for some $x,y,z,t \geq 0$.  Note the vector $\pi_4(x_{j^{(5)}})$ is the $j$-th row of the matrix displayed in the $i_5$ example in the proof of Proposition \ref{prop:subcones-sect-decomp-r4}.  From this example, we gather that $\pi_4(c)=(-x-y, t, x+z, -z) \in \mathbb{Z}^4$.  So $t, z \in \mathbb{Z}$; implying by $x+z \in \mathbb{Z}$ that $x \in \mathbb{Z}$; implying by $-x-y \in \mathbb{Z}$ that $y \in \mathbb{Z}$, as desired.  
\end{proof}

\begin{remark}
We expect that the restricted function from Corollary \ref{cor:restricted-function} is also an injection.   We suspect that there may be a proof of this result via a conjectural generalization of \eqref{eq:theorem-equation} to higher codimension intersections.
\end{remark}

\begin{conceptremark}
\label{rem:last-conceptual-remark}
Recall Remark \ref{rem:sec6conceptremark}.  

We view the real cone $C_\mathcal{T} \subset \mathbb{R}_+^{12}$ as the isomorphic $\mathcal{T}$-chart $C_\mathcal{T}  \cong -\mathcal{A}_{\mathrm{SL}_3, \Box}^+(\mathbb{R}^t)_\mathcal{T}$.  

On the other hand, via the isomorphism $\theta_\mathcal{T} : \mathbb{R}^{12} \to V_\mathcal{T} \subset \mathbb{R}^{18}$ we may view the real cone $\theta_\mathcal{T}(C_\mathcal{T}) \subset V_\mathcal{T}$ as the isomorphic $\mathcal{T}$-chart $\theta_\mathcal{T}(C_\mathcal{T}) \cong -\mathcal{A}_{\mathrm{PGL}_3, \Box}^+(\mathbb{R}^t)_\mathcal{T}$.  

Recall, in addition to the 12 dimensional $\mathcal{A}$-moduli spaces $\mathcal{A}_{\mathrm{SL}_3, \Box}$ and $\mathcal{A}_{\mathrm{PGL}_3, \Box}$, Fock-Goncharov \cite{FockIHES06} and Goncharov-Shen \cite{GoncharovInvent15, GoncharovArxiv19} defined the $\mathcal{X}$- and $\mathcal{P}$-moduli spaces $\mathcal{X}_{\mathrm{PGL}_3, \Box}$ and $\mathcal{P}_{\mathrm{PGL}_3, \Box}$, which are 4- and 12-dimensional, respectively.  In addition, there are canonical  maps $p : \mathcal{A}_{\mathrm{SL}_3, \Box} \to \mathcal{X}_{\mathrm{PGL}_3, \Box}$ and $\overline{p}:   \mathcal{A}_{\mathrm{SL}_3, \Box}  \to \mathcal{P}_{\mathrm{PGL}_3, \Box}$ (the notation $\overline{p}$ may be nonstandard).  The tropical points $\mathcal{X}_{\mathrm{PGL}_3, \Box}(\mathbb{R}^t)$ and $\mathcal{P}_{\mathrm{PGL}_3, \Box}(\mathbb{R}^t)$ of these spaces are also defined, inducing tropicalizations $p^t : \mathcal{A}_{\mathrm{SL}_3, \Box}(\mathbb{R}^t) \to \mathcal{X}_{\mathrm{PGL}_3, \Box}(\mathbb{R}^t)$ and $\overline{p}^t:   \mathcal{A}_{\mathrm{SL}_3, \Box}(\mathbb{R}^t)  \to \mathcal{P}_{\mathrm{PGL}_3, \Box}(\mathbb{R}^t)$ of the canonical maps.  

In terms of $\mathcal{T}$-charts, we view $\mathcal{X}_{\mathrm{PGL}_3, \Box}(\mathbb{R}^t)_\mathcal{T} \cong \mathbb{R}^4$ and      $\mathcal{P}_{\mathrm{PGL}_3, \Box}(\mathbb{R}^t)_\mathcal{T} \cong \mathbb{R}^8 \times \mathbb{R}^4$.  We think of the projection $\pi_4 \circ (\phi_\mathcal{T} \circ \theta_\mathcal{T}) : \mathbb{R}^{12} \to \mathbb{R}^8 \times \mathbb{R}^4 \to \mathbb{R}^4$ as the canonical map $p^t$ written in coordinates.  We think of the isomorphism $\phi_\mathcal{T} \circ \theta_\mathcal{T} : \mathbb{R}^{12} \to \mathbb{R}^8 \times \mathbb{R}^4$ as a coordinate version of the canonical map $\overline{p}^t$.  

We interpret Corollary \ref{cor:canonical-map-surjection} as saying that, when expressed in coordinates, the canonical map $p^t ( \approx \pi_4 \circ \phi_\mathcal{T} \circ \theta_\mathcal{T} )$ also projects the subset $-\mathcal{A}_{\mathrm{SL}_3, \Box}^+(\mathbb{R}^t)_\mathcal{T}   \subset \mathcal{A}_{\mathrm{SL}_3, \Box}(\mathbb{R}^t)_\mathcal{T}    (\approx C_\mathcal{T}    \subset \mathbb{R}^{12}      )  $  of positive points onto   $\mathcal{X}_{\mathrm{PGL}_3, \Box}(\mathbb{R}^t)_\mathcal{T}    (  \approx \mathbb{R}^4   )$.  

In addition, since the positive integer cone $\mathcal{C}_\mathcal{T} \subset \mathbb{Z}_+^{12}$ is in bijection with the set of reduced webs $\mathcal{W}_\Box$ via the web tropical coordinate map $\Phi_\mathcal{T}$, and since $\mathcal{C}_\mathcal{T} \cong -3 \mathcal{A}_{\mathrm{PGL}_3, \Box}^+(\mathbb{Z}^t)_\mathcal{T}$ by Remark \ref{rem:conceptual-remarks},  we can interpret Corollary \ref{cor:restricted-function} as saying that, in coordinates, the canonical map $p^t ( \approx \pi_4 \circ \phi_\mathcal{T} \circ \theta_\mathcal{T} )$ projects a proper subset of $-3 \mathcal{A}_{\mathrm{PGL}_3, \Box}^+(\mathbb{Z}^t)_\mathcal{T} \subset \mathcal{A}_{\mathrm{SL}_3, \Box}(\mathbb{Z}^t)_\mathcal{T} ( \approx \mathcal{C}_\mathcal{T} \subset \mathbb{Z}^{12})$ onto $\mathcal{X}_{\mathrm{PGL}_3, \Box}(\mathbb{Z}^t)_\mathcal{T}    (  \approx \mathbb{Z}^4   )$.  
\end{conceptremark}

\appendix

\section{Background on Fock--Goncharov--Shen theory}
\label{sec:preliminary-for-tropical-points}

We briefly summarize some of the related concepts from Fock-Goncharov-Shen theory \cite{FockIHES06, GoncharovInvent15}.  Details can be found in the other references, for example, \cite[Section $4$]{HuangArxiv19}.  (This appendix assumes the terminology of Section \ref{ssec:markedsurfacesidealtriangulations}.)

\begin{remark}
Although not strictly required for the main theorems of the article, the material of this appendix is intended to emphasize the important conceptual concepts guiding  the rest of the paper.  See, in particular, the Conceptual Remarks \ref{rem:conceptual-remarks}, \ref{rem:equivariant}, \ref{rem:sec6conceptremark}, \ref{rem:last-conceptual-remark}.
\end{remark}

\subsection{\texorpdfstring{$\mathrm{SL}_3$}{SL3}-decorated local systems}

Let $E$ be a 3-dimensional vector space equipped with a volume form $\Omega$.  Let $\mathcal{A}$ denote the collection of decorated (complete) flags in $E$.  

Let $\widehat{S}$ be a marked surface.  Fix a base point $x_0$ in $\widehat{S}$, henceforth suppressed in the notation.  For each puncture $p_i \in m_p$ let $\alpha_i$ be an oriented peripheral closed curve around the puncture.  An $\mathrm{SL}_3$-decorated local system on $\widehat{S}$ determines a pair $(\rho,\xi)$ consisting of:
\begin{enumerate}[label=\textnormal{(\Roman*)}]
\item  a surface group representation $\rho\in\mathrm{Hom}(\pi_1(\widehat{S}),\mathrm{SL}_3)$ with unipotent monodromy along each peripheral curve $\alpha_i$; 
\item  a flag map $\xi:m_b\cup m_p \rightarrow \mathcal{A}$ such that each peripheral monodromy $\rho(\alpha_i)$ fixes the decorated flag $\xi(p_i)\in\mathcal{A}$.  (More precisely, the flag map $\xi$ should be defined equivariantly at the level of the universal cover $\widetilde{S}$.)
\end{enumerate}
A point of the Fock-Goncharov moduli space $\mathcal{A}_{\mathrm{SL}_3,\widehat{S}}$ determines an $\mathrm{SL}_3$-decorated local system up to suitable equivalence.  

Let $V_{\mathcal{T}}$ (resp. $V_{\mathcal{T}_3}$) be the set of vertices of an ideal triangulation $\mathcal{T}$ (resp. ideal $3$-triangulation $\mathcal{T}_3$) of $\widehat{S}$.  Note that $V_{\mathcal{T}}=m_b \cup m_p \subset V_{\mathcal{T}_3}$.  Put
\begin{equation*}
I_3:=
\{
V\in V_{\mathcal{T}_3} - V_{\mathcal{T}}
;
V\text{ lies on an edge of } \mathcal{T}
\}\text{ and }
J_3:= V_{\mathcal{T}_3} - ( V_{\mathcal{T}} \cup I_3).
\end{equation*}
We denote a vertex $V\in I_3 \cup J_3$ on a triangle $(a,b,c)$ oriented counterclockwise by $v^{i,j,k}_{a,b,c}$ where the three nonnegative integers $i,j,k$ summing to $3$ are the least number of edges of $\mathcal{T}_3$ from $V$ to $\overline{bc}$, from $V$ to $\overline{ac}$, and from $V$ to $\overline{ab}$, respectively, where $\overline{bc}$, $\overline{ac}$, $\overline{ab}$ denote the unoriented edges of $\mathcal{T}$ (see Figure \ref{figure:acoor}). 

Consider a vertex $V \in I_3 \cup J_3$ contained in a counterclockwise oriented ideal triangle $\Delta=(a,b,c)$.  For an $\mathrm{SL}_3$-decorated local system $(\rho,\xi)$ in $\mathcal{A}_{\mathrm{SL}_3,\widehat{S}}$ let 
\begin{equation*}
(a_1,a_2,a_{3}), (b_1,b_2,b_3), (c_1,c_2,c_3)  \in E^3
\end{equation*}
be bases adapted to the decorated flags $\xi(a)$, $\xi(b)$, $\xi(c) \in \mathcal{A}$ with respect to the volume form $\Omega$.  The Fock-Goncharov $\mathcal{A}$-coordinate at $V=v_{a,b,c}^{i,j,k} \in I_3 \cup J_3$ is 
\begin{equation*}  
A_{V}(\rho,\xi):=A_{V}:=A_{a,b,c}^{i,j,k}:=\Omega(a^i \wedge b^j \wedge c^k)
\end{equation*}
where $a^1=a_1$, $a^2=a_1 \wedge a_2$, etc.  This is independent of the choice of bases.  (Also put $A_{a,b,c}^{3,0,0} = \Omega(a^3) = \Omega(a_1 \wedge a_2 \wedge a_3) = 1$ and, similarly, $A_{a,b,c}^{0,3,0}=A_{a,b,c}^{0,0,3}=1$.)  

Given the quiver defined with respect to the orientation of the surface as in Figure \ref{figure:acoor}, let
\begin{equation*}
\varepsilon_{VW} =  |\{\text{arrows from } V  \text{ to }  W\}| -|\{\text{arrows from } W  \text{ to }  V\}|.
\end{equation*} 
For any $V \in I_3 \cup J_3$, $V \notin \partial \widehat{S}$, the Fock-Goncharov $\mathcal{X}$-coordinate at $V$ is
\begin{equation*}
X_{V}(\rho,\xi):=X_{V}:= \prod_{W \in I_3 \cup J_3} A_{W}^{\varepsilon_{VW}}.
\end{equation*}

\begin{remark}
By \cite{FockIHES06}, the moduli space $\mathcal{A}_{\mathrm{SL}_3,\widehat{S}}$ has a cluster algebraic structure \cite{FominJAmerMathSoc02} described by quivers on the surface, such as that in Figure \ref{figure:acoor}.  In particular, each $\mathcal{A}$-coordinate transition map between different triangulations $\mathcal{T}$ and $\mathcal{T}^\prime$ is positive rational.
\end{remark}

Let $D=\{(2,1,0),(1,2,0),(1,1,1)\}.$  Suppose the, now pointed, ideal triangle $\Delta=(a;b,c)$ is counterclockwise oriented, as in Figure \ref{figure:acoor1}.  For $(i,j,k)\in D$, the monomials
\begin{equation*}
\alpha_{a;b,c}^{i,j,k} := 
A_{a,b,c}^{i-1,j,k+1} A_{a,b,c}^{i+1,j-1,k}/A_{a,b,c}^{i,j,k} A_{a,b,c}^{i,j-1,k+1}
\end{equation*}
correspond to the three rhombi in Figure \ref{figure:acoor1}.  Define 
\begin{equation*}
P(\Delta):=P(a;b,c):=\alpha_{a;b,c}^{2,1,0}+\alpha_{a;b,c}^{1,2,0}+\alpha_{a;b,c}^{1,1,1}.
\end{equation*}
Let $\Theta$ be the collection of counterclockwise oriented pointed ideal triangles of $\mathcal{T}$.  The Goncharov-Shen potential is 
\begin{equation*}
P=\sum_{\Delta\in \Theta} P(\Delta).
\end{equation*}
Note for a given ideal triangulation $\mathcal{T}$, the Goncharov-Shen potential is a positive Laurent polynomial in the $\mathcal{A}$-coordinates for $\mathcal{A}_{\mathrm{SL}_3, \widehat{S}}$.  

\begin{remark}
The Goncharov-Shen potential is mapping class group equivariant and defines a rational positive function on the moduli space $\mathcal{A}_{\mathrm{SL}_3, \widehat{S}}$.  In \cite{GoncharovInvent15}, the Goncharov-Shen potentials are related to the mirror Landau-Ginzburg potentials.  In \cite[Section $4$]{HuangArxiv19}, the Goncharov-Shen potentials are related to generalized horocycle lengths. 
\end{remark}

The points $\mathcal{A}_{\mathrm{SL}_3, \widehat{S}}(\mathbb{P})$ are defined over any semifield $\mathbb{P}$.  The tropical semifield $\mathbb{R}^t = (\mathbb{R}, \otimes, \oplus)$ is defined by $x \otimes y =  x+y$ and  $x \oplus y = \max\{x,y\}$. The isomorphism $x\rightarrow -x$ sends $(\mathbb{R}^t, +, \max)$ to $(\mathbb{R}^t, +, \min)$. In this appendix, we use $(\mathbb{R}^t, +, \min)$.  The tropical semifields $\mathbb{Z}^t$ and $(1/3) \mathbb{Z}^t$ are similarly defined.  To each  ideal triangulation $\mathcal{T}$ there is associated a $\mathcal{T}$-chart $\mathcal{A}_{\mathrm{SL}_3, \widehat{S}}(\mathbb{P})_\mathcal{T}$.  We have $\mathcal{A}_{\mathrm{SL}_3, \widehat{S}}(\mathbb{R}^t)_\mathcal{T}\cong\mathbb{R}^N$ and $\mathcal{A}_{\mathrm{SL}_3, \widehat{S}}(\mathbb{Z}^t)_\mathcal{T}\cong\mathbb{Z}^N$ and $\mathcal{A}_{\mathrm{SL}_3, \widehat{S}}((1/3) \mathbb{Z}^t)_\mathcal{T}\cong((1/3) \mathbb{Z})^N$ where $N$  is the number of $\mathcal{A}$-coordinates.  

A positive Laurent polynomial $f$  has tropicalization
\begin{equation*}
f^t(x_1,\cdots, x_k)= \lim_{C\rightarrow -\infty} \log f(e^{C x_1}, \cdots, e^{C x_k})/C.
\end{equation*}
The tropical $\mathcal{A}$-coordinates are denoted $A_V^t$ for $V \in I_3 \cup J_3$.  (Note $(A_{a,b,c}^{3,0,0})^t = 0$ since $A_{a,b,c}^{3,0,0}=1$.)  So
\begin{equation*}
(\alpha_{a;b,c}^{i,j,k})^t= 
(A_{a,b,c}^{i-1,j,k+1})^t+ (A_{a,b,c}^{i+1,j-1,k})^t - (A_{a,b,c}^{i,j,k})^t-( A_{a,b,c}^{i,j-1,k+1})^t
\end{equation*}
and the tropicalized Goncharov-Shen potential is
\begin{equation*}
P^t=\min\{(\alpha_{a;b,c}^{i,j,k})^t\}_{\text{all } \alpha_{a;b,c}^{i,j,k} \text{ of }P}
\end{equation*}
where the minimum is taken over all rhombi in all pointed ideal triangles of $\mathcal{T}$.    

The condition $P^t\geq0$ on $\mathcal{A}_{\mathrm{SL}_3, \widehat{S}}(\mathbb{R}^t)$ determines the space $\mathcal{A}_{\mathrm{SL}_3,\widehat{S}}^+(\mathbb{R}^t)$ of positive real tropical points as well as its $\mathcal{T}$-chart $\mathcal{A}_{\mathrm{SL}_3,\widehat{S}}^+(\mathbb{R}^t)_\mathcal{T}$.  The spaces $\mathcal{A}_{\mathrm{SL}_3,\widehat{S}}^+(\mathbb{Z}^t)$ and $\mathcal{A}_{\mathrm{SL}_3,\widehat{S}}^+((1/3)\mathbb{Z}^t)$ as well as their $\mathcal{T}$-charts $\mathcal{A}_{\mathrm{SL}_3,\widehat{S}}^+(\mathbb{Z}^t)_\mathcal{T}$ and $\mathcal{A}_{\mathrm{SL}_3,\widehat{S}}^+((1/3)\mathbb{Z}^t)_\mathcal{T}$ are similarly defined.  

\begin{figure}[t]
\includegraphics[scale=0.3]{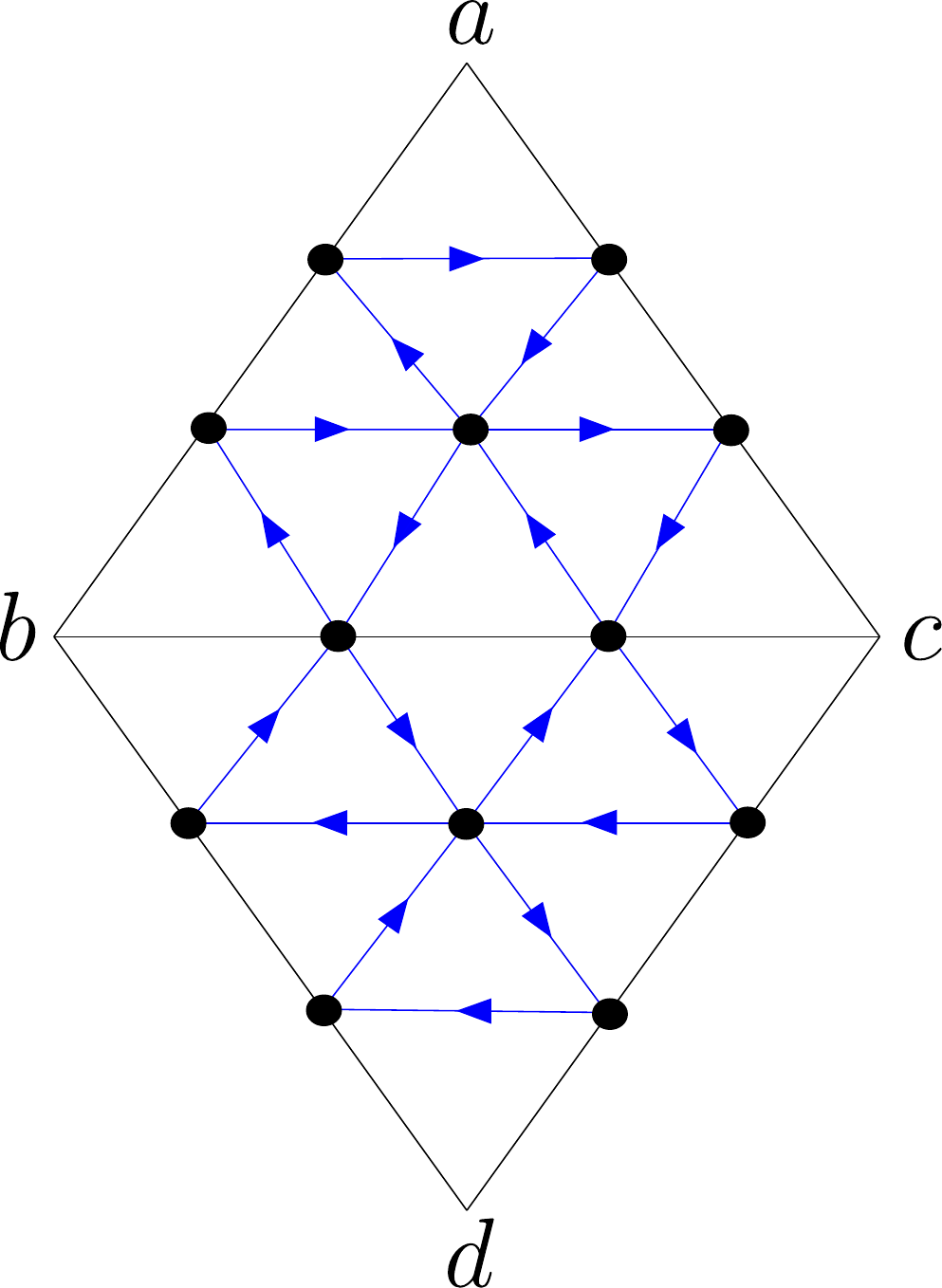}
\caption{     $3$-triangulation with quiver.}
\label{figure:acoor}
\end{figure}

\begin{figure}[t]
\includegraphics[scale=0.38]{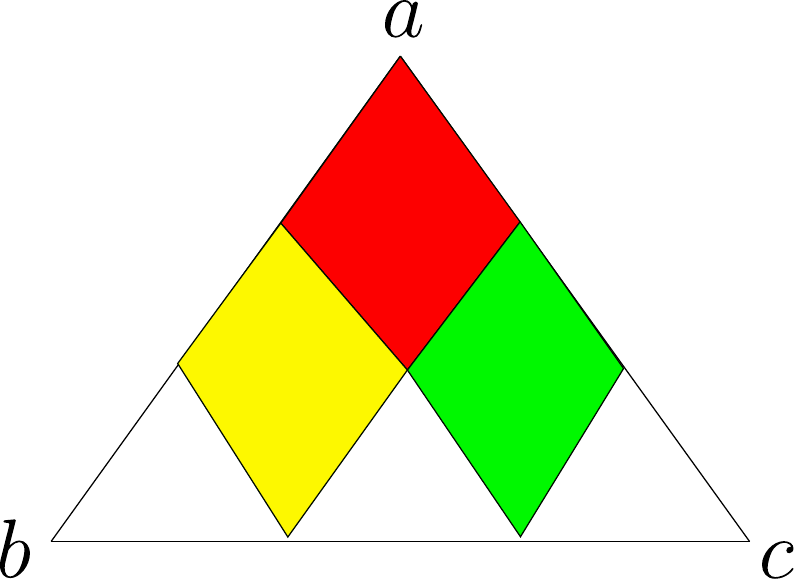}
\caption{     Red rhombus for $\alpha_{a;b,c}^{2,1,0}$, yellow rhombus for $\alpha_{a;b,c}^{1,2,0}$, and green rhombus for $\alpha_{a;b,c}^{1,1,1}$ in a pointed ideal triangle.}
\label{figure:acoor1}
\end{figure}

\subsection{\texorpdfstring{$\mathrm{PGL}_3$}{PGL3}-decorated local systems}

The moduli space $\mathcal{A}_{\mathrm{PGL}_3,\widehat{S}}$ is defined analogously to $\mathcal{A}_{\mathrm{SL}_3,\widehat{S}}$.  Although the $\mathcal{A}$-coordinates $A_{a,b,c}^{i,j,k}$ are no longer defined for $\mathcal{A}_{\mathrm{PGL}_3,\widehat{S}}$ (as they depend on the choices of scale), the rhombus numbers $\alpha_{a;b,c}^{i,j,k}$, and so the Goncharov-Shen potential $P$, are defined (the choices of scale cancel out in the ratio defining $\alpha_{a;b,c}^{i,j,k}$).  So, the tropical points $\mathcal{A}_{\mathrm{PGL}_3,\widehat{S}}(\mathbb{P})$ and their positive parts $\mathcal{A}^+_{\mathrm{PGL}_3,\widehat{S}}(\mathbb{P})$ with respect to the tropicalized Goncharov-Shen potential $P^t$ are defined as well.  We refer to the space $\mathcal{A}_{\mathrm{PGL}_3,\widehat{S}}(\mathbb{Z}^t)$ as the space of lamination-tropical points, and we refer to the space $\mathcal{A}^+_{\mathrm{PGL}_3,\widehat{S}}(\mathbb{Z}^t)$ as the space of web-tropical points.  

The $\mathcal{T}$-chart $\mathcal{A}_{\mathrm{PGL}_3,\widehat{S}}(\mathbb{Z}^t)_\mathcal{T}$ is determined by imposing the condition $(\alpha_{a;b,c}^{i,j,k})^t \in \mathbb{Z}$ for all $(\alpha_{a;b,c}^{i,j,k})^t$ of $P^t$ on $\mathcal{A}_{\mathrm{SL}_3,\widehat{S}}((1/3) \mathbb{Z}^t)_\mathcal{T}  \cong  ((1/3) \mathbb{Z})^N$, and the $\mathcal{T}$-chart $\mathcal{A}^+_{\mathrm{PGL}_3,\widehat{S}}(\mathbb{Z}^t)_\mathcal{T}$ is determined by imposing the condition $(\alpha_{a;b,c}^{i,j,k})^t \in \mathbb{Z}_+$ for all $(\alpha_{a;b,c}^{i,j,k})^t$ on $\mathcal{A}_{\mathrm{SL}_3,\widehat{S}}((1/3) \mathbb{Z}^t)_\mathcal{T}  \cong  ((1/3) \mathbb{Z})^N$.  

In summary, we have the following relations among $\mathcal{T}$-charts:
\begin{gather*}
\mathbb{Z}^N\cong\mathcal{A}_{\mathrm{SL}_3,\widehat{S}}(\mathbb{Z}^t)_\mathcal{T} \subset \mathcal{A}_{\mathrm{PGL}_3,\widehat{S}}(\mathbb{Z}^t)_\mathcal{T} \subset \mathcal{A}_{\mathrm{SL}_3,\widehat{S}}((1/3)\mathbb{Z}^t)_\mathcal{T}\cong((1/3) \mathbb{Z})^N,\\
\mathcal{A}_{\mathrm{SL}_3,\widehat{S}}^+(\mathbb{Z}^t)_\mathcal{T} \subset \mathcal{A}_{\mathrm{PGL}_3,\widehat{S}}^+(\mathbb{Z}^t)_\mathcal{T} \subset \mathcal{A}_{\mathrm{SL}_3,\widehat{S}}^+((1/3)\mathbb{Z}^t)_\mathcal{T},\\
\mathcal{A}_{\mathrm{PGL}_3, \widehat{S}}(\mathbb{R}^t)_\mathcal{T} = \mathcal{A}_{\mathrm{SL}_3, \widehat{S}}(\mathbb{R}^t)_\mathcal{T}\cong\mathbb{R}^N,\,\,
\mathcal{A}_{\mathrm{PGL}_3,\widehat{S}}^+(\mathbb{R}^t)_\mathcal{T}=\mathcal{A}_{\mathrm{SL}_3,\widehat{S}}^+(\mathbb{R}^t)_\mathcal{T}.
\end{gather*}

\begin{remark}
\label{rem:realsolutionsareinteger} 
$ $
\begin{enumerate}[label=\textnormal{(\Roman*)}]
\item  \label{item:annoyingminusisgn}
By \cite[Remark 6.5]{DouglasArxiv20}, we equally well could have imposed the conditions defining $\mathcal{A}_{\mathrm{PGL}_3,\widehat{S}}(\mathbb{Z}^t)_\mathcal{T}$ and $\mathcal{A}_{\mathrm{PGL}_3,\widehat{S}}^+(\mathbb{Z}^t)_\mathcal{T}$ on $\mathcal{A}_{\mathrm{SL}_3,\widehat{S}}( \mathbb{R}^t)_\mathcal{T} \cong \mathbb{R}^N$ rather than on $\mathcal{A}_{\mathrm{SL}_3,\widehat{S}}((1/3) \mathbb{Z}^t)_\mathcal{T}  \cong  ((1/3) \mathbb{Z})^N$.  That is, all real solutions are, in fact, one third integer solutions.  Moreover, in the case of $\mathcal{A}_{\mathrm{PGL}_3,\widehat{S}}^+(\mathbb{Z}^t)_\mathcal{T}$, these one third integer solutions are nonpositive (this is because the tropicalized rhombus numbers $(\alpha_{a;b,c}^{i,j,k})^t$ are defined with the opposite sign compared to those appearing in \cite{DouglasArxiv20}).  (Note that less confusing conventions should be possible without significant effort.)
\item  \label{rem2:realsolutionsareinteger}  
The space $\mathcal{A}_{\mathrm{PGL}_3,\widehat{S}}(\mathbb{Z}^t)$ of lamination-tropical points is called the space of `balanced points' in \cite{KimArxiv20}.
\item  By \cite{GoncharovInvent15}, when $\widehat{S}$ is a disk with three marked points on its boundary, the positive integer tropical points are identified with the Knutson-Tao hive  \cite{KnutsonJAmerMathsoc99}.
\end{enumerate}
\end{remark}

\section{Cones}
\label{sec:cones}

\begin{remark}
Some of our terminology might be non-standard, adapted for the purposes of this paper.  For example, using the terminology of \cite{maclagan2021introduction}, a `polyhedral cone' is what we call a cone; a `simplicial polyhedral cone' is what we call a sector; and, a `$k$-dimensional pure simplicial polyhedral fan' determines what we call a sector decomposition of a cone.
\end{remark}

\subsection{Positive integer cones and Hilbert bases}
\label{ssec:hilbert-bases}
		   
Recall that $\mathbb{Z}_+$ denotes the set of nonnegative integers.  
			
\begin{definition}
\label{def:submonoid}
A subset $\mathcal{M} \subset \mathbb{Z}^k$ (or $\subset \mathbb{R}^k$) is a \emph{submonoid} if $\mathcal{M}$ is closed under addition and contains 0.  
\end{definition}
			
\begin{definition}
Let $\mathcal{M} \subset \mathbb{Z}^k$ (or $\subset \mathbb{R}^k$) be a submonoid.  An element $x \in \mathcal{M}$ is \emph{irreducible} if $x$ is nonzero, and $x$ cannot be written as the sum of two nonzero elements of $\mathcal{M}$.  
			
We denote by $\mathcal{H} \subset \mathcal{M}$ the set of irreducible elements of $\mathcal{M}$.  
			
A subset $\mathcal{D} \subset \mathcal{M}$ is:
\begin{enumerate}[label=\textnormal{(\Roman*)}]
\item  \emph{$\mathbb{Z}_+$-spanning} if every $x \in \mathcal{M}$ is of the form $x=\lambda_1 x_1 + \lambda_2 x_2 + \cdots + \lambda_m x_m$ for some   $x_i \in \mathcal{D}$ and $\lambda_i \in \mathbb{Z}_+$, in which case we write $x \in \mathrm{span}_{\mathbb{Z}_+}(\mathcal{D})$; 
\item  a \emph{minimum} $\mathbb{Z}_+$-spanning set  if, in addition, for every  $\mathbb{Z}_+$-spanning set $\mathcal{D}^\prime \subset \mathcal{M}$ we have $\mathcal{D} \subset \mathcal{D}^\prime$.  				
\end{enumerate}
\end{definition}
			
Note that a minimum $\mathbb{Z}_+$-spanning set is unique if it exists.

\begin{definition}
\label{def:positive-integer-cone}
A \emph{positive integer cone} $\mathcal{C}$ is a submonoid of  $\mathbb{Z}_+^k$.
\end{definition}
	
\begin{prop}
\label{prop:irreducible-elements}
The subset $\mathcal{H} \subset \mathcal{C} \subset \mathbb{Z}_+^k$ of irreducible elements of a positive integer cone $\mathcal{C} \subset \mathbb{Z}_+^k $  is the unique minimum $\mathbb{Z}_+$-spanning subset of $\mathcal{C}$.  		
\end{prop}

\begin{proof}
	If $x$ is irreducible and is in the $\mathbb{Z}_+$-span of $x^\prime_1, x^\prime_2, \dots, x^\prime_m$ for $x_i^\prime \in \mathcal{C}$, then $x=x^\prime_{i_0}$ for some $i_0$ by the irreducibility property.  Thus $\mathcal{H}$ is contained in any $\mathbb{Z}_+$-spanning set $\mathcal{D}$.  

It remains to show that every element $x$ of $\mathcal{C}$ is in the $\mathbb{Z}_+$-span of $\mathcal{H}$.      We argue by induction on the sum $\Sigma(x) \in \mathbb{Z}_+$ of the coordinates of $x \in \mathbb{Z}_+^k$; that is, on the quantity $\Sigma(x):=\sum_{i=1}^k n_i$ where $x=(n_1, n_2, \dots, n_k) \in \mathcal{C} \subset \mathbb{Z}_+^k$.  This is true if $x=0$, where $\Sigma(x)=0$.  So assume that $x$ is nonzero, and that $x^\prime$ is in the $\mathbb{Z}_+$-span of $\mathcal{H}$ whenever $\Sigma(x^\prime)<\Sigma(x)$.  If $x$ is irreducible, we are done.  Else, $x=y+z$ with both $y, z \in \mathcal{C} \subset \mathbb{Z}_+^k$ nonzero.  So $\Sigma(y) < \Sigma(x)$ and $\Sigma(z) < \Sigma(x)$.  Thus, $y$ and $z$ are in the $\mathbb{Z}_+$-span of $\mathcal{H}$ by hypothesis, so $x$ is as well.  
\end{proof}

\begin{remark}
\label{rem:spanningpropertiesofsubmonoids}
Note that the $\mathbb{Z}_+$-spanning property of $\mathcal{H}$ in Proposition \ref{prop:irreducible-elements}  is not true if we had only assumed that $\mathcal{C}=\mathcal{M}$ is a submonoid of $\mathbb{Z}^k$.  For example, the monoid $\mathcal{M}=\mathbb{Z}^k$ has no irreducible elements.  However, it is also not essential that the submonoid be contained in a single orthant; note Remark \ref{rem:moregeneralhilbertbasis} as well.
\end{remark}

\begin{definition}
\label{def:hilbert-basis}
Let $\mathcal{H} \subset \mathcal{C} \subset \mathbb{Z}_+^k$ be as in Proposition \ref{prop:irreducible-elements}.  If $\mathcal{H}$ is a finite set, then it is called the \emph{Hilbert basis} of the positive integer cone $\mathcal{C} \subset \mathbb{Z}_+^k$.  
\end{definition}

\begin{example} 
\label{ex:hilbbasis}
$ $ 
	\begin{enumerate}[label=\textnormal{(\Roman*)}]
		\item  \label{ex1:hilbbasis}
	The standard basis  of $\mathbb{Z}^k$ is the Hilbert basis $\mathcal{H}$ of the positive integer cone $\mathcal{C}=\mathbb{Z}_+^k$.  
	
		\item  \label{ex2:hilbbasis}
		By Proposition \ref{prop:irreducible-elements}, the set $\mathcal{H} = \left\{ (1,0), (1,1), (0, 2) \right\}$ is the Hilbert basis of the positive integer cone $\mathcal{C}=\mathbb{Z}_+ (1, 0)+\mathbb{Z}_+ (1,1) + \mathbb{Z}_+ (0,2) \subset \mathbb{Z}_+^2$.  
		
		\item  \label{ex3:hilbbasis}
		Similarly, again by Proposition \ref{prop:irreducible-elements}, the set $\mathcal{H} = \left\{ (0,1,0), (1,1,0), (1,0,1), (0,1,1) \right\}$ is the Hilbert basis of the positive integer cone $\mathcal{C}=\mathrm{span}_{\mathbb{Z}_+}(\mathcal{H}) \subset \mathbb{Z}_+^3$.

		\end{enumerate}
\end{example}

\begin{remark}
\label{rem:moregeneralhilbertbasis}
Hilbert bases \cite{hilbert1890ueber, schrijver1981total} appear in linear algebra and linear programming, and are defined in more generality than what we have defined here.  
\end{remark}

\subsection{Cones over the real numbers}
\label{ssec:sector-decompositions-of-cones}
	   
Recall that $\mathbb{R}_+$ (resp. $\mathbb{R}_-$) denotes the set of nonnegative (nonpositive) real numbers.

\subsubsection{Cones and sector decompositions}
\label{sssec:cones-and-sector-decompositions}
	
\begin{definition}	
\label{def:cone-defs}
$ $
\begin{enumerate}[label=\textnormal{(\Roman*)}]
\item  A \emph{(real) cone} $C \subset \mathbb{R}^k$ is a subset of $\mathbb{R}^k$ such that $	C = \{ \sum_{i=1}^m \lambda_i c_i;  \lambda_i \in \mathbb{R}_+ \} $ for some finite set $\{ c_i \}_{i=1,2,\dots,m} \subset \mathbb{R}^k$, called a \emph{generating set} of $C$.   We also write $C=\mathrm{span}_{\mathbb{R}_+}(c_1, c_2, \dots, c_m)$. 
\item  The minimum number of elements of a generating set is called the \emph{rank} of the cone $C$.  A generating set $\left\{ c_i \right\}_i$ of minimum size is called a \emph{basis} of $C$.
\item  The subspace $\widetilde{C} \subset \mathbb{R}^k$ defined by $ \widetilde{C} = \{ \sum_{i} \widetilde{\lambda}_i c_i;  \widetilde{\lambda}_i \in \mathbb{R} \} $ is independent of the choice of generating set $\{ c_i \}_i$, and its dimension is called the \emph{dimension} of the cone $C$.  Note $ \mathrm{dim}(C) \leq \mathrm{rank}(C) < \infty$.  
\item  A cone $C$ is a \emph{sector}  if $\mathrm{dim}(C)=\mathrm{rank}(C)$.
\item  A cone $C \subset \mathbb{R}^k$ is \emph{full} if $\mathrm{dim}(C)=k$.  	
\end{enumerate}
\end{definition}

	\begin{example}
	\label{ex:cone-examples}   
	$ $
	
	\begin{enumerate}[label=\textnormal{(\Roman*)}]

\item  \label{ex1:cone-examples}  $C=\mathbb{R}$ is a full cone with basis $\left\{ 1, -1 \right\}$.  Its rank is 2, and its dimension is 1, so it is not a sector.

\item  \label{ex2:cone-examples}  $C=\mathbb{R}^2$ is a full cone with basis $\left\{ e_1, e_2, -e_1-e_2 \right\}$.  Its rank is 3, and its dimension is 2, so it is not a sector.  Here, $e_i$ is the $i$-th standard basis element.

The subcone $C^\prime = \mathbb{R} e_1 \subset C$ is not full and is not a sector.

\item  \label{ex3:cone-examples}  One checks that the full cone $C=\mathbb{R}^k$ has rank $k+1$, by taking as a basis the vertices of the $k$-dimensional simplex centered at the origin.

\item  \label{ex4:cone-examples}  For $j=1,2,\dots,k$, the cone $C^{(j)} = \sum_{i=1}^j \mathbb{R}_+ e_i \subset \mathbb{R}^k$ is a sector of rank $j$.  It is full only for $j=k$, namely when $C^{(j)}=C^{(k)} = \mathbb{R}_+^k$ is the positive orthant.

	\end{enumerate}

	\end{example}

\begin{remark}
	Note that if a cone $C$ has rank $r$ and if $\left\{ c_i \right\}_i$ is a generating set for $C$, there does not necessarily exist a subset $\left\{ c_{i_j} \right\}_{j=1,2,\dots,r}$ that is a basis.  For example, take $C=\mathbb{R}^2$ with generating set $\left\{ e_1, -e_1, e_2, -e_2 \right\}$.
\end{remark}
	
\begin{definition} 
\label{def:sector-decomp}
$ $
\begin{enumerate}[label=\textnormal{(\Roman*)}]
\item  A \emph{sector decomposition} of a full cone $C \subset \mathbb{R}^k$ is a finite collection $\{ C_i \}_{i=1,2,\dots,p}$ of subcones $C_i \subset C$ satisfying:
\begin{enumerate}[label=\textnormal{(\roman*)}]
\item  each $C_i$ is a full sector;
\item  $C = C_1 \cup C_2 \cup \cdots \cup C_p$ is the union of the sectors $C_i$;
\item  for each distinct $i,j \in \{ 1,2,\dots,m \}$, the intersection $C_i \cap C_j \subset \mathbb{R}^k$ has empty interior, that is, does not contain an open subset of $\mathbb{R}^k$.
\end{enumerate}
\item  If $C$ and $C^\prime$ are two full cones in $\mathbb{R}^k$, then the intersection $C \cap C^\prime$ is a \emph{wall} if $C \cap C^\prime$ is a cone of dimension $k-1$. 
\end{enumerate}
\end{definition}

\begin{example} $ $
\label{ex:second-cone-example}

\begin{enumerate}[label=\textnormal{(\Roman*)}]
	\item 
In Example \ref{ex:cone-examples}\eqref{ex1:cone-examples}, putting $C_1 = \mathbb{R}_+$ and $C_2 = \mathbb{R}_-$ yields a sector decomposition of the cone $C=\mathbb{R}$.  The intersection $C_1 \cap C_2 =\left\{ 0 \right\}$ is a wall.

	\item    \label{ex2prime:second-cone-example}
In  Example \ref{ex:cone-examples}\eqref{ex2:cone-examples}, putting $C_1=\mathbb{R}_+ e_1 + \mathbb{R}_+ e_2$, $C_2=\mathbb{R}_+ e_2 + \mathbb{R}_+ (-e_1 - e_2)$, and $C_3=\mathbb{R}_+  (-e_1 - e_2) + \mathbb{R}_+ e_1$ yields a sector decomposition of $C=\mathbb{R}^2$.  Each sector faces two walls, and each pairwise-distinct intersection $C_i \cap C_j$ is a wall.
	
	\item  \label{ex2:second-cone-example}  Again in Example \ref{ex:cone-examples}\eqref{ex2:cone-examples}, alternatively putting $C_1=\mathrm{span}_{\mathbb{R}_+}(e_1, e_2)$, $C_2 = \mathrm{span}_{\mathbb{R}_+}(e_2, -e_1)$, $C_3=\mathrm{span}_{\mathbb{R}_+}(-e_1, -e_2)$, and $C_4=\mathrm{span}_{\mathbb{R}_+}(-e_2, e_1)$ yields a sector decomposition of $C=\mathbb{R}^2$.  Each sector faces two walls, the intersections $C_i \cap C_{i+1}$ are walls, but the intersections $C_1 \cap C_3 = C_2 \cap C_4 = \left\{ 0 \right\}$ are not walls.
	
	\item  In Example \ref{ex:cone-examples}\eqref{ex4:cone-examples}, putting $C_1=C^{(k)} = \mathbb{R}_+^k$ yields a  sector decomposition with a single sector.
\end{enumerate}
\end{example}

\begin{observation}
\label{lem:stupid-lemma-about-walls}
If $\{ C_i \}_{i=1,2,\dots,p}$ is a sector decomposition of a full cone $C \subset \mathbb{R}^k$, and if $W=C_i \cap C_\ell$ is a wall, then there is no other pair of sectors giving this wall:  $W=C_{i^\prime}\cap C_{\ell^\prime}$ if and only if $\{ i, \ell \} = \{ i^\prime, \ell^\prime \}$.
\end{observation}

This essentially follows since walls have codimension 1 in $\mathbb{R}^k$.  We give a proof here for completeness.

\begin{proof}[Proof of Observation {\upshape\ref{lem:stupid-lemma-about-walls}}]

Let the $k-1$ dimensional subcone $W \subset C$ be generated by $\left\{ w_j \right\}_{j=1,2,\dots,m}$.  Let $v_i$ in the $k$ dimensional cone $C_i$ be linearly independent from $\left\{ w_j \right\}_j$.

Let $w$ be a fixed point in the interior of $W$, namely $w=\sum_{j=1}^m \lambda_j w_j$ for some $\lambda_j > 0$.  Then there is $\epsilon_i > 0$ such that 
	\begin{equation*}
	\label{eq:stupid-wall-lemma}
	\tag{$\dagger$}
		\left\{ w+\lambda^{(i)} v_i  + \sum_{j=1}^m \widetilde{\lambda}_j w_j; \quad 0 \leq \lambda^{(i)} < \epsilon_i \quad \text{and} \quad \lambda_j-\epsilon_i < \widetilde{\lambda}_j < \lambda_j+\epsilon_i \right\} \subset C_i.
	\end{equation*}
Let $v^\perp_i \in \widetilde{W}^\perp - \left\{ 0 \right\} \subset \mathbb{R}^k$ be the nonzero component of $v_i$ perpendicular to the subspace $\widetilde{W}$ (with respect to the standard inner product, say); note $v_i^\perp$ is not necessarily in $C_i$.  By shrinking $\epsilon_i$, we can arrange that \eqref{eq:stupid-wall-lemma} holds with $v_i$ replaced by $v_i^\perp$; denote   the resulting subset \eqref{eq:stupid-wall-lemma} by $\overline{U}_i \subset C_i$.  Similarly, define a subset $\overline{U}_\ell \subset C_\ell$ depending on some $v_\ell^\perp \neq 0 \in \widetilde{W}^\perp$ and   $\epsilon_\ell > 0$.  By further shrinking, let us arrange that $\epsilon_i = \epsilon_\ell$.

So far, we have not used the assumption that $k$ is the dimension of the ambient space $\mathbb{R}^k$.  We now use this assumption, to note that $\widetilde{W}^\perp \subset \mathbb{R}^k$ is 1 dimensional, and that $\overline{U}_i$ and $\overline{U}_\ell$ have nonempty interiors in $\mathbb{R}^k$.

Without loss of generality, assume $|v_i^\perp| \leq |v_\ell^\perp|$.  If $v_\ell^\perp$ pointed in the same direction as $v_i^\perp$, then by construction $C_i \cap C_\ell \supset \overline{U}_i \cap \overline{U}_\ell = \overline{U}_i$ would have nonempty interior, violating the hypothesis that $C_i$ and $C_\ell$ are part of a sector decomposition of the full cone $C \subset \mathbb{R}^k$.  
Thus, $v_\ell^\perp$ points in the opposite direction as $v_i^\perp$.

Now, assume $C_{i^\prime} \cap C_{\ell^\prime} = W$ as well;
define $\overline{U}_{i^\prime} \subset C_{i^\prime}$ and $\overline{U}_{\ell^\prime} \subset C_{\ell^\prime}$ as above.  After possibly swapping indices, we have that $v_i^\perp$ and $v_{i^\prime}^\perp$ (resp. $v_\ell^\perp$ and $v_{\ell^\prime}^\perp$) point in the same direction.  Arguing as above, it follows that $\overline{U}_i \cap \overline{U}_{i^\prime} \subset C_i \cap C_{i^\prime}$ and $\overline{U}_\ell \cap \overline{U}_{\ell^\prime}\subset C_\ell \cap C_{\ell^\prime}$ have nonempty interiors.  Therefore, $i=i^\prime$ and $\ell=\ell^\prime$.
\end{proof}

\begin{remark}
	Observation \ref{lem:stupid-lemma-about-walls} is false for higher codimension intersections.  For instance, in Example \ref{ex:second-cone-example}\eqref{ex2:second-cone-example}, we have $C_1 \cap C_3 = C_2 \cap C_4 = \left\{ 0 \right\}$.

\end{remark}

\subsubsection{Some technical statements about cones of the form $C \subset \mathbb{R}_+^k \times \mathbb{R}^n$}
\label{ssec:technical-statements}
	
\begin{lem}
\label{lem:rank-lemma}
Let $C \subset \mathbb{R}_+^k \times \mathbb{R}^n$ be a cone satisfying the following properties:
\begin{enumerate}[label=\textnormal{(\Roman*)}]
\item  $e_i \in C$ for $i=1,2,\dots,k$, where $e_i$ is the $i$-th standard basis element of $ \mathbb{R}^k \times \mathbb{R}^n$;
\item  $\pi_n(C) = \mathbb{R}^n$, where $\pi_n : \mathbb{R}^k \times \mathbb{R}^n \to \mathbb{R}^n$ is the natural projection.  
\end{enumerate}
Then, $\mathrm{dim}(C)=k+n$.  Namely, $C$ is full.  
\end{lem}

\begin{example*}[part 1]
	The rank of such a cone $C$ as in Lemma \ref{lem:rank-lemma} can equal the dimension or be strictly greater.  For example, $C=\mathbb{R}_+ \times \mathbb{R}$ has rank 3, whereas $C^\prime=\mathrm{span}_{\mathbb{R}_+}(\left\{ (1;1), (0;-1) \right\}) \subset \mathbb{R}_+ \times \mathbb{R}$ has rank 2.  
\end{example*}

\begin{proof}[Proof of Lemma {\upshape\ref{lem:rank-lemma}}]
	Note $\mathbb{R}_+^k \times \left\{ 0 \right\} \subset C$, so $\mathbb{R}^k \times \left\{ 0 \right\} \subset \widetilde{C}$; see Definition \ref{def:cone-defs}.  The hypothesis $\pi_n(C)=\mathbb{R}^n$ thus implies $\left\{ 0 \right\} \times \mathbb{R}^n \subset \widetilde{C}$.  It follows that $\widetilde{C}=\mathbb{R}^k \times \mathbb{R}^n$.  
\end{proof}

\begin{lem}
\label{lem:second-cone-lemma}
Consider a full cone $C \subset \mathbb{R}_+^k \times \mathbb{R}^n$ as in Lemma {\upshape\ref{lem:rank-lemma}}.  Let $\{ x_j \}_{j=1,2,\dots,m}$ be a finite subset of $C$ with $m \geq n$, and let $\{ J_i \}_{i=1,2,\dots,p}$ for some $p$ be a collection of index sets $J_i=\{j^{(i)}_1, j^{(i)}_2, \dots, j^{(i)}_{n} \}  \subset \{ 1, 2, \dots, m \}$ of constant size $n$ \textup{(}$|J_i|=n$\textup{)}.  
	
Assume in addition:
\begin{enumerate}[label=\textnormal{(\Roman*)}]
\item  $C=\cup_{i=1}^{p} C_i$ is the union of the subcones
\begin{equation*}
C_i
=	\mathrm{span}_{\mathbb{R}_+}( 
\{ e_1, e_2, \dots, e_k \} 
\cup	\{ x_j;  j \in J_i \})
  \subset C \subset \mathbb{R}_+^k \times \mathbb{R}^n
\,\,  ( i=1,2,\dots,p );
\end{equation*}
\item  the subcones
\begin{equation*}
D_i = \mathrm{span}_{\mathbb{R}_+}(\{ \pi_n(x_j);  j \in J_i \})\subset \mathbb{R}^n\,\,( i=1,2,\dots,p )
\end{equation*}
are full sectors forming a sector decomposition $\{ D_i \}_{i=1,2,\dots,p}$ of $\mathbb{R}^n$ \textup{(}Definition {\upshape\ref{def:sector-decomp}}\textup{)}.  
\end{enumerate}
Then, the subcones $C_i \subset C$ are full sectors forming a sector decomposition $\{ C_i \}_{i=1,2,\dots,p}$ of $C$.  
	
Moreover, $\pi_n(C_i \cap C_\ell)=D_i \cap D_{\ell}$ for all  $i, \ell$.  In particular, the sector $C_i$ projects via $\pi_n$ to the sector $D_i$.  
\end{lem}

\begin{example*}[part 2]
For $C \subset \mathbb{R}_+ \times \mathbb{R}$ as in part 1, put $x_1=(0;1)$ and $x_2=(0;-1)$.  For $C^\prime \subset \mathbb{R}_+ \times \mathbb{R}$ as in part 1, put $x^\prime_1=(1;1)$ and $x^\prime_2=(0;-1)$.  In both cases, put $J_1=J^\prime_1=\left\{ 1 \right\}$ and $J_2=J^\prime_2=\left\{ 2 \right\}$.  

Then, in both cases, $D_1=D_1^\prime=\mathbb{R}_+$ and $D_2=D_2^\prime=\mathbb{R}_-$.  For $C$, we have $C_1=\mathbb{R}_+ \times \mathbb{R}_+$ and $C_2=\mathbb{R}_+ \times \mathbb{R}_-$.  For $C^\prime$, we have $C^\prime_1=\mathrm{span}_{\mathbb{R}_+}(\left\{(1;0), (1;1)\right\})$ and $C^\prime_2=\mathbb{R}_+ \times \mathbb{R}_-$.  

Lastly, $C_1 \cap C_2=C_1^\prime \cap C_2^\prime=\mathbb{R}_+ \times \left\{ 0 \right\}$, which projects by $\pi_1$ to $\left\{0\right\} = D_1 \cap D_2=D_1^\prime \cap D_2^\prime$.  
\end{example*}

\begin{proof}[Proof of Lemma {\upshape\ref{lem:second-cone-lemma}}]
	By construction $\pi_n(C_i)=D_i$, so $\pi_n(C_i \cap C_\ell) \subset D_i \cap D_\ell$.  To see that the reverse inclusion holds requires a little argument.  Consider a general element
	\begin{equation*}
		y=\pi_n \left( \sum_{j \in J_i} \lambda_j x_j \right) 
		= \pi_n \left( \sum_{j \in J_\ell} \lambda^\prime_j x_j \right)
		\quad  \in  D_i \cap D_\ell \subset \mathbb{R}^n
		\quad\quad  \left( \lambda_j, \lambda^\prime_j \in \mathbb{R}_+ \right).
	\end{equation*}
	For $k^\ast = 1,2,\dots, k$, the $k^\ast$-coordinate of $v_i := \sum_{j \in J_i} \lambda_j x_j \in C_i$ is either $\leq$ or $\geq$ the $k^\ast$-coordinate of $v_\ell := \sum_{j \in J_\ell} \lambda^\prime_j x_j \in C_\ell$; without loss of generality, assume it is $\leq$.    Putting $\alpha \geq 0$ to be the difference of these $k^\ast$-coordinates, we have that $v^\ast_i := v_i + \alpha e_{k^\ast} \in C_i$ has the same $k^\ast$-coordinate as $v_\ell$ and still satisfies $\pi_n(v^\ast_i)=y=\pi_n(v_\ell) \in D_i \cap D_\ell$.  We then put $v^\ast_i$ to be the new $v_i$, and reiterate this procedure for each $k^\ast$, updating $v_i$ or $v_\ell$ at each step.  The end result is that $v_i = v_\ell \in C_i \cap C_\ell$ and $\pi_n(v_i)=y=\pi_n(v_\ell)$ as desired.

	We move on to establishing that $\left\{ C_i \right\}_i$ is a sector decomposition of $C$.

	To see that $C_i \subset \mathbb{R}^k \times \mathbb{R}^n$ is a full cone, use the hypothesis that $D_i$ is open in $\mathbb{R}^n$, and proceed as in the proof of Lemma \ref{lem:rank-lemma}.  By definition, $C_i$ has a generating set consisting of $k+n$ elements.  Thus, $k+n=\mathrm{dim}(C_i)\leq\mathrm{rank}(C_i)\leq k+n$, so $C_i$ is a sector.

	That $C=\cup_i C_i$ is by hypothesis.

	It remains to show that $C_i \cap C_\ell$ has  empty interior for all pairwise-distinct $i, \ell$.  Suppose otherwise, so that there is a nonempty open subset $U \subset \mathbb{R}^k \times \mathbb{R}^n$ such that $U \subset C_i \cap C_\ell$.  Since projections are open maps, $\pi_n(U)$ is a nonempty open subset of $\mathbb{R}^n$ contained in $D_i \cap D_\ell$, contradicting that $\left\{ D_i \right\}_i$ is a sector decomposition.  
\end{proof}

\begin{lem}
\label{lem:third-cone-lemma}
Let the full cone $C \subset \mathbb{R}_+^k \times \mathbb{R}^n$, the sector decomposition $\{ D_i \}_{i=1,2,\dots,p}$ of $\mathbb{R}^n$, and the sector decomposition $\{ C_i \}_{i=1,2,\dots,p}$ of $C$ be as in Lemma {\upshape\ref{lem:second-cone-lemma}}.  
	
Assume in addition:
\begin{enumerate}[label=\textnormal{(\Roman*)}]
\item  for each $i, \ell$ such that $D_i \cap D_\ell$ is a wall in $\mathbb{R}^n$ \textup{(}Definition {\upshape\ref{def:sector-decomp}}\textup{)}, we have more specifically that the intersection $J_i \cap J_\ell \subset \{ 1, 2, \dots, m \}$ of index sets has $n-1$ elements, and  
\begin{equation*}
D_i \cap D_\ell = \mathrm{span}_{\mathbb{R}_+}
( 
\{ \pi_n(x_j);  j \in J_i \cap J_\ell \} 
)
  \subset \mathbb{R}^n. 
\end{equation*}
\end{enumerate}
Then, for any $i, \ell$, we have that $C_i \cap C_\ell \subset \mathbb{R}_+^k \times \mathbb{R}^n$ is a wall in $C$ if and only if $D_i \cap D_\ell \subset \mathbb{R}^n$ is a wall in $\mathbb{R}^n$.  In particular, if this is the case for a given $i, \ell$, then
\begin{equation*}\tag{m}
\label{eq:technical-equation-intersection}
C_i \cap C_\ell = \mathrm{span}_{\mathbb{R}_+}( \{ e_1, e_2, \dots, e_k \} \cup \{ x_j;  j \in J_i \cap J_\ell \} )  \subset \mathbb{R}_+^k \times \mathbb{R}^n.
\end{equation*}
This furnishes a $1$-to-$1$ correspondence $\{ \text{walls of } \{ D_i \}_i \text { in } \mathbb{R}^n \} \leftrightarrow \{ \text{walls of } \{ C_i \}_i \text { in } C \}$.
\end{lem}

\begin{example*}[part 3]
	For $C, C^\prime \subset \mathbb{R}_+ \times \mathbb{R}$ as in part 2, the sole wall in $\mathbb{R}$ is $\left\{ 0 \right\}$, corresponding to the sole wall $\mathbb{R}_+ \times \left\{ 0 \right\}$ in $C, C^\prime$.  
	
	As a non-example, where the hypothesis and part of the conclusion, \eqref{eq:technical-equation-intersection}, fail:  In $\mathbb{R}_+ \times \mathbb{R}^2$ let $x_{1}=(0;0,1)$, $x_{2}=(1;1,0)$, $x_{3}= (0;0,-1)$, $x_{4}=(2;1,0)$, $x_5=(0;-1,0)$.  Put $J_1=\left\{1,2\right\}$, $J_2=\left\{3,4\right\}$, $J_3=\left\{1,5\right\}$, $J_4=\left\{3,5\right\}$.  Put $C=\cup_{i=1}^4 C_i$.  This is indeed a cone, generated by $x_1, \dots, x_5$ and $e_1$.  
	Then $D_1 \cap D_2 = \mathbb{R}_+ \times \left\{ 0 \right\} \subset \mathbb{R}^2$ is a wall, but $J_1 \cap J_2 = \emptyset$ and $C_1 \cap C_2 = \mathrm{span}_{\mathbb{R}_+}(\left\{ e_1, x_4 \right\} \supsetneq \mathrm{span}_{\mathbb{R}_+}(\left\{ e_1 \right\})$.  Note that $x_2 \in \pi_2^{-1}(D_1 \cap D_2) - (C_1 \cap C_2)$.  On the other hand, $D_3 \cap D_4 = \mathbb{R}_- \times \left\{ 0 \right\}$, $J_3 \cap J_4 = \left\{ 5 \right\}$, and $C_3 \cap C_4 = \mathbb{R}_+ \times \mathbb{R}_{-} \times \left\{ 0 \right\}  = \mathrm{span}_{\mathbb{R}_+}(\left\{ e_1, x_5 \right\})$.  Moreover, $\pi_2^{-1}(D_3 \cap D_4)=C_3 \cap C_4$.  
\end{example*}

\begin{proof}[Proof of Lemma {\upshape\ref{lem:third-cone-lemma}}]
	Assume first that $C_i \cap C_\ell$ is a wall in $C$.  By Lemma \ref{lem:second-cone-lemma}, $\pi_n(C_i \cap C_\ell)=D_i \cap D_\ell$.  Thus, since $C_i \cap C_\ell$ is a cone, we have that $D_i \cap D_\ell$ is a cone.  It  follows that the subspace $\widetilde{D}_{i\ell} := \widetilde{D_i \cap D_\ell} \subset \mathbb{R}^n$ is equal to the projection under $\pi_n$ of the subspace $\widetilde{C}_{i\ell} := \widetilde{C_i \cap C_\ell} \subset \mathbb{R}^k \times \mathbb{R}^n$ (Definition \ref{def:cone-defs}), the latter which is $k+n-1$ dimensional by assumption.  By definition of $C_i$ and $C_\ell$, we have $\mathbb{R}^k \times \left\{ 0 \right\} \subset \widetilde{C}_{i\ell}$.  Thus $\widetilde{C}_{i\ell} = \mathbb{R}^k \times \pi_n(\widetilde{C}_{i\ell})=\mathbb{R}^k \times \widetilde{D}_{i\ell}$.  Hence $\widetilde{D}_{i\ell}$ is $n-1$ dimensional, as desired.

Conversely, assume $D_i \cap D_\ell$ is a wall in $\mathbb{R}^n$.  Let $j_i$ (resp. $j_\ell$) be the unique index in $J_i - (J_i \cap J_\ell)$ (resp. $J_\ell - (J_i \cap J_\ell)$).  Since $D_i$ (resp. $D_\ell$) has dimension $n$, the set of vectors $\left\{ \pi_n(x_j); j \in J_i \right\}$ (resp. $\left\{ \pi_n(x_j); j \in J_\ell \right\}$) in $\mathbb{R}^n$ is linearly independent.   Since $D_i \cap D_\ell \subset \text{(in fact, equals) } \mathrm{span}_{\mathbb{R}_+}(\left\{\pi_n(x_j); j \in J_i \cap J_\ell\right\})$ by hypothesis, it follows that any expression in $\mathbb{R}^n$   of the form 
\begin{equation*}
	\lambda_{j_i} \pi_n(x_{j_i}) + \sum_{j \in J_i \cap J_\ell} \lambda_j \pi_n(x_j)  \quad \in D_i \cap D_\ell
	\quad\quad \text{ or } \quad\quad
	\lambda^\prime_{j_\ell} \pi_n(x_{j_\ell}) + \sum_{j \in J_i \cap J_\ell} \lambda^\prime_j \pi_n(x_j) \quad \in D_i \cap D_\ell
\end{equation*}
 must imply $\lambda_{j_i}=0$ or $\lambda_{j_\ell}=0$, respectively.

Consequently, since $\pi_n(C_i \cap C_\ell) \subset D_i \cap D_\ell$ (in fact, this is an equality by Lemma \ref{lem:second-cone-lemma}), any expression in $C \subset \mathbb{R}_+^k \times \mathbb{R}^n$ of the form 
\begin{equation*}
		\sum_{r=1}^k \eta_r e_r + \lambda_{j_i} x_{j_i} + \sum_{j \in J_i \cap J_\ell} \lambda_j x_j
		=
		\sum_{r=1}^k \eta^\prime_r e_r + \lambda^\prime_{j_\ell} x_{j_\ell} + \sum_{j \in J_i \cap J_\ell} \lambda^\prime_j x_j
	 \quad \in C_i \cap C_\ell
	 \quad\quad
	 \left( 
	 \eta_r, \lambda_{j_i}, \lambda_j, \eta^\prime_r, \lambda^\prime_{j_\ell}, \eta^\prime_\ell \in \mathbb{R}_+
	 \right)
\end{equation*}
implies, by applying $\pi_n$, that $\lambda_{j_i}=\lambda^\prime_{j_\ell}=0$.  We gather that the $\subset$ inclusion of \eqref{eq:technical-equation-intersection} is true.   Since the reverse inclusion holds by the definitions of $C_i$ and $C_\ell$, we have that \eqref{eq:technical-equation-intersection} is true.  In particular, $C_i \cap C_\ell$ is a cone.

To finish, we show $C_i \cap C_\ell$ is $k+n-1$ dimensional. By the same argument as in the first paragraph of this proof, we have $\widetilde{C}_{i \ell}=\mathbb{R}^k \times \widetilde{D}_{i \ell}$.  Since $\widetilde{D}_{i \ell}$ is $n-1$ dimensional by hypothesis, the claim is true.

It remains to construct the desired bijection $f: \left\{ \text{walls of } \left\{ D_i \right\}_i \text {in } \mathbb{R}^n \right\} \to \left\{ \text{walls of } \left\{ C_i \right\}_i \text {in } C \right\}$.  Indeed, let $W=D_i \cap D_\ell$ be a wall in $\mathbb{R}^n$.  Put $f(W)$ to be the wall $C_i \cap C_\ell$ in $C$.  By Observation \ref{lem:stupid-lemma-about-walls}, the choice of $\left\{ i, \ell \right\}$ such that $W=D_i \cap D_\ell$ is unique, so $f$ is well-defined.  Conversely, if $W^\prime \subset \mathbb{R}_+^k \times \mathbb{R}^n$ is a wall of $\left\{ C_i \right\}_i$ in $C$, put $g(W^\prime)=\pi_n(W^\prime)\subset \mathbb{R}^n$.  Since $W^\prime = C_i \cap C_\ell$ for some $i, \ell$ by definition,  we have that $g(W^\prime)=\pi_n(C_i \cap C_\ell)=D_i \cap D_\ell$ is a wall in $\mathbb{R}^n$.  Thus, we have  defined a function $g : \left\{ \text{walls of } \left\{ C_i \right\}_i \text {in } C \right\} \to \left\{ \text{walls of } \left\{ D_i \right\}_i \text {in } \mathbb{R}^n \right\}$ (note we did not need to use Observation \ref{lem:stupid-lemma-about-walls} for this direction).  It is immediate from the construction that $f$ and $g$ are inverses of each other.  
\end{proof}

\subsubsection{Cone completions}
\label{sssec:cone-completions}

\begin{definition}
\label{def:completion}
Let $\mathcal{M} \subset \mathbb{R}^k$ be a submonoid (Definition \ref{def:submonoid})  having a finite $\mathbb{Z}_+$-spanning set $\{ c_i \}_{i=1,2,\dots,m}$ (we say $\mathcal{M}$ is \emph{finitely generated}).  Then, its \emph{completion} $\overline{\mathcal{M}} \subset \mathbb{R}^k$ is the corresponding real cone with the same generating set $\{ c_i \}_i$ (this is  independent of the choice of generating set).  
\end{definition}

By Proposition \ref{prop:irreducible-elements}, we immediately have:

\begin{observation}
\label{obs:rank-hilbert-basis-estimate}
Let $\mathcal{C} \subset \mathbb{Z}_+^k$ be a positive integer cone \textup{(}Definition {\upshape\ref{def:positive-integer-cone}}\textup{)} admitting a Hilbert basis $\mathcal{H} \subset \mathcal{C}$ \textup{(}Definition {\upshape\ref{def:hilbert-basis}}\textup{)}.  Then, the rank of its completion $\overline{\mathcal{C}} \subset \mathbb{R}_+^k$ is less than or equal to the number of elements of the Hilbert basis $\mathcal{H}$.  \qed
\end{observation}

\begin{example} $ $
\label{ex:hib-basis-cone-examples}
	\begin{enumerate}[label=\textnormal{(\Roman*)}]

		\item  In Example \ref{ex:hilbbasis}\eqref{ex1:hilbbasis}, we see $\overline{\mathcal{C}}=\mathbb{R}_+^k$ and $\mathrm{dim}(\overline{\mathcal{C}})=\mathrm{rank}(\overline{\mathcal{C}})=k=|\mathcal{H}|$.  
		
		\item  In Example \ref{ex:hilbbasis}\eqref{ex2:hilbbasis}, we see that $\overline{\mathcal{C}}=\mathbb{R}_+^2$ and $\mathrm{dim}(\overline{\mathcal{C}})=\mathrm{rank}(\overline{\mathcal{C}})=2<3=|\mathcal{H}|$.  
		
		\item  In Example \ref{ex:hilbbasis}\eqref{ex3:hilbbasis}, we have that $\overline{\mathcal{C}}=\mathrm{span}_{\mathbb{R}_+}(\mathcal{H}) \subset \mathbb{R}_+^3$ and one checks that $\mathrm{dim}(\overline{\mathcal{C}})=3<4=\mathrm{rank}(\overline{\mathcal{C}})=|\mathcal{H}|$, as no point of $\mathcal{H}$ lies in the $\mathbb{R}_+$-span of the other three points.  
	\end{enumerate}
\end{example}

\begin{lem}
\label{lem:actual-geometry}
Let $\mathcal{M} \subset \mathbb{R}^k$ be a monoid.  Assume there are finitely generated submonoids $\mathcal{M}_1, \mathcal{M}_2, \dots, \mathcal{M}_p \subset \mathcal{M}$  such that $\mathcal{M}=\cup_{i=1}^p \mathcal{M}_i$ \textup{(}in particular, $\mathcal{M}$ is finitely generated\textup{)}.  Then, $\overline{\mathcal{M}}=\cup_{i=1}^p \overline{\mathcal{M}}_i$.  
\end{lem}

\begin{proof}

	The inclusion $\supset$ is by definition.  Let $\left\{ c_j \right\}_{j=1,2,\dots,m}$ be a finite generating set for $\mathcal{M}$.  Put $C = \sum_{j=1}^m |c_j|$, where $|c_j|$ is the Euclidean length in $\mathbb{R}^k$.  Then for each $x \in \overline{\mathcal{M}}$, there exists some $y \in \mathcal{M}$ such that $|x-y| \leq C$.  Indeed, writing $x=\sum_{j=1}^m \lambda_j c_j$, put $y=\sum_{j=1}^m n_j c_j \in \mathcal{M}$ where $n_j \geq 0$ is the largest integer less than or equal to $\lambda_j$.  We have that
	\begin{equation*}
		|x-y|=\left| \sum_{j=1}^m (\lambda_j - n_j) c_j \right|
		\leq \sum_{j=1}^m (\lambda_j - n_j) |c_j| \leq C.
	\end{equation*}

We now argue by contradiction.  Put $A=\cup_{i=1}^p \overline{\mathcal{M}}_i$ and suppose there is $x \in \overline{\mathcal{M}} - A$.  Note $x \neq 0$.  Let $\pi : \mathbb{R}^k - \left\{ 0 \right\} \to S^{k-1}$ be the natural projection onto the unit sphere.  It suffices to show there exists an open ball $B \subset \mathbb{R}^k$ such that $\pi(x) \in B \cap S^{k-1}$ and such that $\pi(A-\left\{0\right\})$ does not intersect $B \cap S^{k-1}$.  Indeed, since $\mathcal{M} \subset A$ by hypothesis, this would contradict that there is some point of $\mathcal{M}$ in $\pi^{-1}(B \cap S^{k-1}) \subset \mathbb{R}^k$ at distance at most $C$ from a point in the ray $\mathbb{R}_{>0} x \subset \overline{\mathcal{M}} \cap \pi^{-1}(B \cap S^{k-1})$.

Suppose such a $B$ does not exist.  Then, there is a sequence $(a_n)$ in $A-\left\{0\right\}$ such that $\mathrm{lim}_{n \to \infty} \pi(a_n)=\pi(x)$ in $S^{k-1}$.  Since the subcones $\overline{\mathcal{M}}_i$ are finitely generated, they are closed subsets of $\mathbb{R}^k$; thus, $A$ is closed as well.  It follows that the intersection
\begin{equation*}
	A^c = A \cap \left\{ z \in \mathbb{R}^k; \quad \frac{1}{2} \leq |z| \leq \frac{3}{2} \right\}
\end{equation*}
is compact; thus, $\pi(A^c) \subset S^{k-1}$ is compact as well.  We gather $\pi(A^c)$ is closed in $S^{k-1}$.  

Now, since the ray through any point of $A$ intersects $A^c$, we have $\pi(a_n) \in \pi(A^c)$ for all $n$.  It follows by the closedness of $\pi(A^c)$ that $\pi(x) \in \pi(A^c)$; in particular, the ray through $x$ intersects $A=\cup_{i=1}^p \overline{\mathcal{M}}_i$.  But this ray cannot intersect any subcone $\overline{\mathcal{M}}_i$, as $\overline{\mathcal{M}}_i$ is closed under scaling by $\mathbb{R}_+$ and $x \notin A$ by assumption.  This is a contradiction.  
\end{proof}

\section{Flip examples in the square}
\label{ssec:flip-examples}

In Section \ref{section:fe}, we proved the naturality of the web tropical coordinates without having to see what the new good position of a web in the square looks like after flipping the diagonal, which is topologically subtle.  In this section, we give some examples of seeing the good position after the flip.  This gives us another way to check the formulas of Theorem \ref{thm:second-main-theorem}; see also  Remark \ref{rem:regarding-flip-proof} at the beginning of Section \ref{section:fe}.
	  	
The 9 symmetry classes of web families (see Figure \ref{figure:9cases} and Remark \ref{rem:symmetry-groupings}) fall into roughly three types.  Let us study the flip a bit more intensively for one example of each type.  
	
Let $W=W_c \in \mathcal{W}_\Box$ be a cornerless web in the square and belonging to the family $\mathcal{W}_{i_j} \subset \mathcal{W}_\Box$, where the value of $i_j$ depends on which of the 9 cases we are considering ($j=1,2,\dots,9$); see Notation \ref{not:web-families}.  
	
Recall that $\mathcal{T}$ (resp. $\mathcal{T}^\prime$) is the triangulation shown in the left hand side (resp. right hand side) of Figure \ref{figure:flip}.  Consider the web tropical coordinate maps $\Phi_\mathcal{T} : \mathcal{W}_\Box \to \mathcal{C}_\mathcal{T}$ and $\Phi_{\mathcal{T}^\prime} : \mathcal{W}_\Box \to \mathcal{C}_{\mathcal{T}^\prime}$ (see Section \ref{section:wa}).  Denote  $c = \Phi_\mathcal{T}(W) \in \mathcal{C}_\mathcal{T} \subset \mathbb{Z}_+^{12}$ by
\begin{equation*}
c=(c_j)_{j=1,2,\dots,12} = (x_1, x_2, x_3, x_4, x_5, x_6, x_7, x_8, y_1, y_2, y_3, y_4)
\end{equation*}
and $c^\prime = \Phi_{\mathcal{T}^\prime} \circ \Phi_\mathcal{T}^{-1}(c) \in \mathcal{C}_{\mathcal{T}^\prime} \subset \mathbb{Z}_+^{12}$ by 
\begin{equation*}
c^\prime=(c^\prime_j)_{j=1,2,\dots,12}=(x^\prime_1, x^\prime_2, x^\prime_3, x^\prime_4, x^\prime_5, x^\prime_6, x^\prime_7, x^\prime_8, z_1, z_2, z_3, z_4)
\end{equation*}
(compare Definition \ref{def:mutationmap} and Figure \ref{figure:flip}).  We know right away that $x_i=x_i^\prime$ for $i=1,2,\dots,8$.

Our goal is to check that  \eqref{equation:mu1}, \eqref{equation:mu2}, \eqref{equation:mu3}, \eqref{equation:mu4}  are satisfied, by presenting the explicit good position of the family $\mathcal{W}_{i_j}$ after the flip, allowing us to calculate the coordinates directly.  We do this in the three example cases $j=1, 3, 7$.  

Recall in particular $x,y,z,t \in \mathbb{Z}_+$ in Figure \ref{figure:9cases}.  

\begin{example*}[family $(1)$]     
The simplest cases are given by schematics (1) and (2) of Figure \ref{figure:9cases}.  We verify case (1) here.  The other case is similar.  We compute the $c_j$'s and $c_j^\prime$'s via Figure \ref{figure:triangle}; see also Section \ref{sssec:definitionofthetropicalwebcoordinates}.  

Notice in this case it is obvious that the web on the right hand side of Figure \ref{figure:flip1} is in good position with respect to the flipped triangulation.  

Left hand side of Figure \ref{figure:flip1},  coordinates $c_j   (\text{for }j=1,2,\dots,12)$:
\begin{enumerate}[label=(\arabic*)]
\item  $2x + y + 2z + 2t$;
\item  $x+2y+z+t$;
\item  $2x$;
\item  $x$;
\item  $2x+2y+z+2t$;
\item  $x+y+2z+t$;
\item  $2t$;
\item  $t$;
\item  $2x+y+2z+3t$;
\item  $x+2y+z+2t$;
\item  $3x+2y+z+2t$;
\item  $2x+y+2z+t$.
\end{enumerate}

Right hand side of Figure  \ref{figure:flip1},  coordinates $c^\prime_j  (\text{for }j=9,\dots,12)$:
\begin{enumerate}[label=($\arabic*^\prime$)]\setcounter{enumi}{8}
\item  $x+y+2z+2t$;
\item  $3x+2y+z+t$;
\item  $2x+2y+z+t$;
\item  $x+y+2z+3t$.
\end{enumerate}

The following computations verify  \eqref{equation:mu1}, \eqref{equation:mu2}, \eqref{equation:mu3}, \eqref{equation:mu4} in this case. 
\begin{gather*}
\text{\eqref{equation:mu1}:}
\max\{(x+2y+z+t)+(3x+2y+z+2t),(2x+y+2z+3t)+2x\}-(x+2y+z+2t)\\=\max\{4x+4y+2z+3t, 4x+y+2z+3t\}-(x+2y+z+2t)=3x+2y+z+t.
\end{gather*}
\begin{gather*}
\text{\eqref{equation:mu2}:}
\max\{(2x+y+2z+3t)+(x+y+2z+t),2t+(3x+2y+z+2t)\}-(2x+y+2z+t)
\\=  \mathrm{max}\{  
3x + 2y + 4z + 4t, 3x + 2y + z + 4t
\}-(2x+y+2z+t)
=x+y+2z+3t.
\end{gather*}
\begin{gather*}
\text{\eqref{equation:mu3}:}
\max\{(2x+y+2z+2t)+(x+y+2z+3t),t+(3x+2y+z+t)\}-(2x+y+2z+3t)
\\=\mathrm{max}\{    3x + 2y + 4z + 5t,    3x + 2y   + z + 2t
\}-(2x+y+2z+3t)
=x+y+2z+2t.
\end{gather*}
\begin{gather*}
\text{\eqref{equation:mu4}:}
\max\{(3x+2y+z+t)+(2x+2y+z+2t),(x+y+2z+3t)+x\}-(3x+2y+z+2t)
\\=\mathrm{max}\{  5x + 4y + 2z + 3t,   2x + y + 2z + 3t
\}-(3x+2y+z+2t)
=2x+2y+z+t.
\end{gather*}
\end{example*}

\begin{example*}[family $(3)$]
The next simplest cases are given by schematics (3), (4), (5), (6), and (9) of Figure \ref{figure:9cases}.  We verify case (3) here.  The other four cases are similar.  We compute the $c_j$'s and $c_j^\prime$'s via Figure \ref{figure:triangle}; see also Section \ref{sssec:definitionofthetropicalwebcoordinates}.  

Unlike in the previous example, it is less obvious that the schematic appearing on the right hand side of Figure \ref{figure:flip3} faithfully displays how the good position looks after the flip.  That this is indeed the case is a bit subtle topologically, however can be verified by an explicit procedure that draws the desired flipped bigon on top of the starting web as  represented by the left hand side of Figure \ref{figure:flip3}.  We demonstrate this bigon drawing procedure in Figure \ref{fig:case3flip}.     

The schematic diagram of the web in good position restricted to the flipped bigon in the right hand side of Figure \ref{figure:flip3} is  shown in the left hand side of Figure \ref{figure:flip3l}.  It is an enjoyable exercise to check that this bigon schematic agrees with the web example  schematically shown in Figure \ref{fig:case3flip}.  

Another guiding example showing the web in good position after the flip (without using schematics), is provided in the right hand side of Figure \ref{figure:flip3l}.

Left hand side of Figure \ref{figure:flip3},  coordinates $c_j   (\text{for }j=1,2,\dots,12)$:
\begin{enumerate}[label=(\arabic*)]
\item  $x + 2y $;
\item  $2x + y $;
\item  $x + y + z + t $;
\item  $2x+2y+2z+2t $;
\item  $x+y+z $;
\item  $2x+2y+2z $;
\item  $2y+z+2t $;
\item  $y+2z+t $;
\item  $x + 3y + 2z + t $;
\item  $2x+2y+2z+t $;
\item  $3x + 3y+  3z+2t $;
\item  $x+y+z+2t $.
\end{enumerate}

Right hand side of Figure  \ref{figure:flip3},  coordinates $c^\prime_j  (\text{for }j=9,\dots,12)$:
\begin{enumerate}[label=($\arabic*^\prime$)]\setcounter{enumi}{8}
\item  $2x + 3y +z + t $;
\item  $3x+2y+z+t $;
\item  $x+3y+2z+2t $;
\item  $2x+4y+3z+2t $.
\end{enumerate}

The following computations verify \eqref{equation:mu1}, \eqref{equation:mu2}, \eqref{equation:mu3}, \eqref{equation:mu4} in this case.
\begin{gather*}
\text{\eqref{equation:mu1}:}
\max\{(2x+y)+(3x+3y+3z+2t), (x+3y+2z+t)+(x+y+z+t)\}-(2x+2y+2z+t)
\\=\mathrm{max}\{ 5x+4y+3z+2t, 2x+4y+3z+2t  \}-(2x+2y+2z+t)
=3x+2y+z+t.
\end{gather*}
\begin{gather*}
\text{\eqref{equation:mu2}:}
\max\{    (x+3y+2z+t)+(2x+2y+2z), (2y+z+2t)+(3x+3y+3z+2t)   \}-(x+y+z+2t)
\\=\mathrm{max}\{ 3x + 5y + 4z + t, 3x+5y+4z+4t \}-(x+y+z+2t)
=2x+4y+3z+2t.
\end{gather*}
\begin{gather*}
\text{\eqref{equation:mu3}:}
\max\{(x+2y)+(2x+4y+3z+2t), (y+2z+t)+(3x+2y+z+t)\}-(x+3y+2z+t)
\\=\mathrm{max}\{  3x+6y+3z+2t,  3x+3y+3z+2t  \}-(x+3y+2z+t)
=  2x+3y+z+t.
\end{gather*}
\begin{gather*}
\text{\eqref{equation:mu4}:}
\max\{(3x+2y+z+t)+(x+y+z), (2x+4y+3z+2t)+(2x+2y+2z+2t)\}\\-(3x+3y+3z+2t)
\\=\mathrm{max}\{  4x+3y+2z+t, 4x+6y+5z+4t    \}-(3x+3y+3z+2t)
=  x+3y+2z+2t.
\end{gather*}
\end{example*}

\begin{example*}[family $(7)$]
The last group of  cases are given by schematics (7) and (8) of Figure \ref{figure:9cases}.  We verify case (7) here.  The other case is similar.  We compute the $c_j$'s and $c_j^\prime$'s via Figure \ref{figure:triangle}; see also Section \ref{sssec:definitionofthetropicalwebcoordinates}.  

Note that, unlike for the previous two examples, in this case there are two possibilities:  $z \geq t$ and $z \leq t$.  In Figure \ref{figure:flip7}, we display the case when $z \geq t$ (the $z \leq t$ case is similar).  

As for the previous example, it is not immediately obvious that the schematic appearing on the right hand side of Figure \ref{figure:flip7} displays  the correct good position.  We again verify this by explicitly drawing the flipped bigon, as shown in Figure \ref{fig:case7flip}.     

The schematic diagram of the web in good position restricted to the flipped bigon in the right hand side of Figure \ref{figure:flip7} is  shown in the left hand side of Figure \ref{figure:flip7l}.  It is an enjoyable exercise to check that this bigon schematic agrees with the web example  schematically shown in Figure \ref{fig:case7flip}.  

Another guiding example showing the web in good position after the flip (without using schematics), is provided in the right hand side of Figure \ref{figure:flip7l}.

We demonstrate the calculation when $z \geq t$ (the case $z \leq t$ is similar).  

Left hand side of Figure \ref{figure:flip7},  coordinates $c_j   (\text{for }j=1,2,\dots,12)$:
\begin{enumerate}[label=(\arabic*)]
\item  $2x + t $;
\item  $x + 2t $;
\item  $2y+z $;
\item  $y+2z $;
\item  $2x + 2y + 2t $;
\item  $x + y + t $;
\item  $2x + 2y + 2z $;
\item  $x+y+z $;
\item  $3x + y + z + t $;
\item  $2x + y + z + 2t $;
\item  $2x+3y+2z+2t $;
\item  $x + 2y + 2z + t $.
\end{enumerate}

Right hand side of Figure  \ref{figure:flip7},  coordinates $c^\prime_j  (\text{for }j=9,\dots,12)$:
\begin{enumerate}[label=($\arabic*^\prime$)]\setcounter{enumi}{8}
\item  $2x + 2y  + z + t $;
\item  $x + 2y + z + 2t $;
\item  $x + y + 2z - t $;
\item  $3x + 3y + 2z  + t $.
\end{enumerate}

The following computations verify  \eqref{equation:mu1}, \eqref{equation:mu2}, \eqref{equation:mu3}, \eqref{equation:mu4} in this case.  Note that the last equation uses the assumption $z \geq t$.  
\begin{gather*}
\text{\eqref{equation:mu1}:}
\max\{     (x+2t)+(2x+3y+2z+2t),    (3x+y+z+t)+(2y+z)
\}-(2x+y+z+2t)
\\=\mathrm{max}\{     3x+3y+2z+4t, 3x+3y+2z+t
\}-(2x+y+z+2t)
=x+2y+z+2t.
\end{gather*}
\begin{gather*}
\text{\eqref{equation:mu2}:}
\max\{(3x+y+z+t)+(x+y+t), (2x+2y+2z)+(2x+3y+2z+2t)
\}-(x+2y+2z+t)
\\=\mathrm{max}\{ 4x+2y+z+2t,   4x+5y+4z+2t
\}-(x+2y+2z+t)
=3x+3y+2z+t.
\end{gather*}
\begin{gather*}
\text{\eqref{equation:mu3}:}
\max\{(2x+t)+(3x+3y+2z+t), (x+y+z)+(x+2y+z+2t)\}-(3x+y+z+t)
\\=\mathrm{max}\{  5x+3y+2z+2t,   2x+3y+2z+2t
\}-(3x+y+z+t)
=2x+2y+z+t.
\end{gather*}
\begin{gather*}
\text{\eqref{equation:mu4}:}
\max\{    (x+2y+z+2t)+(2x+2y+2t),   (3x+3y+2z+t)+(y+2z)
\}-(2x+3y+2z+2t)
\\=\mathrm{max}\{   3x+4y+z+4t,     3x+4y+4z+t
\}-(2x+3y+2z+2t)
\\=(3x+4y+4z+t)-(2x+3y+2z+2t)
=x+y+2z-t.
\end{gather*}
\end{example*}

\begin{figure}[t]
\includegraphics[scale=0.8]{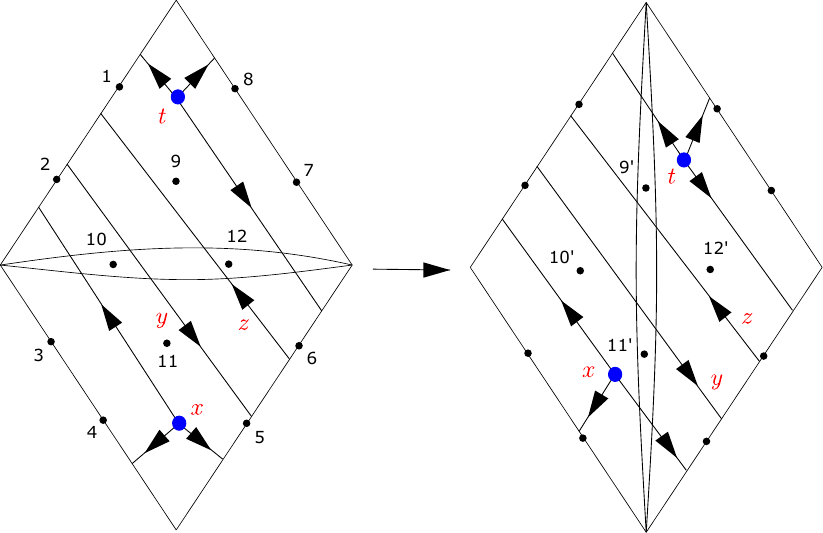}
\caption{     Family (1).}
\label{figure:flip1}
\end{figure}

\begin{figure}[t]
\includegraphics[scale=0.8]{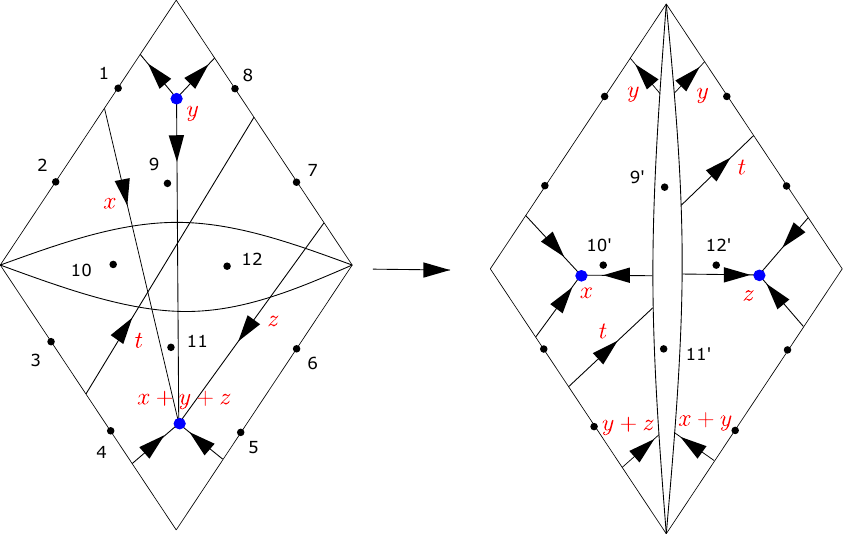}
\caption{     Family (3).}
\label{figure:flip3}
\end{figure}

\begin{figure}[t]
\includegraphics[scale=0.8]{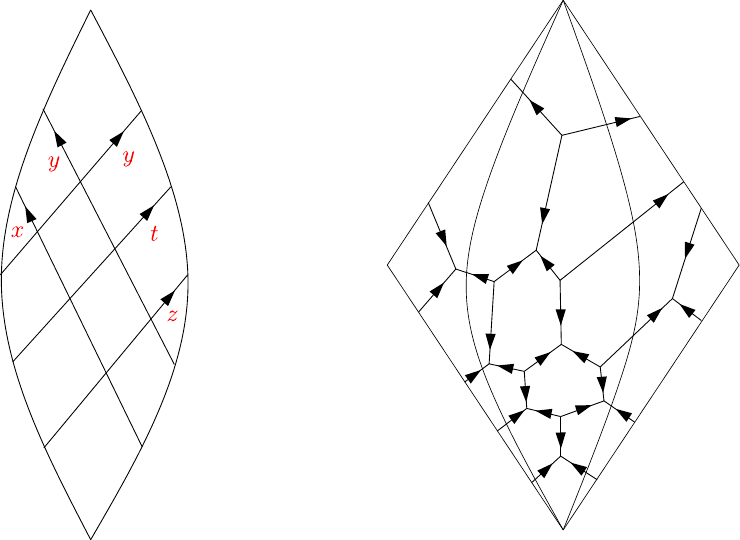}
\caption{     Family (3).  Left: Bigon schematic after the flip. Right: An example of the web in good position after the flip:  $x=y=z=t=1$.}
\label{figure:flip3l}
\end{figure}  

\begin{figure}[t]
\includegraphics[height=.8\textheight]{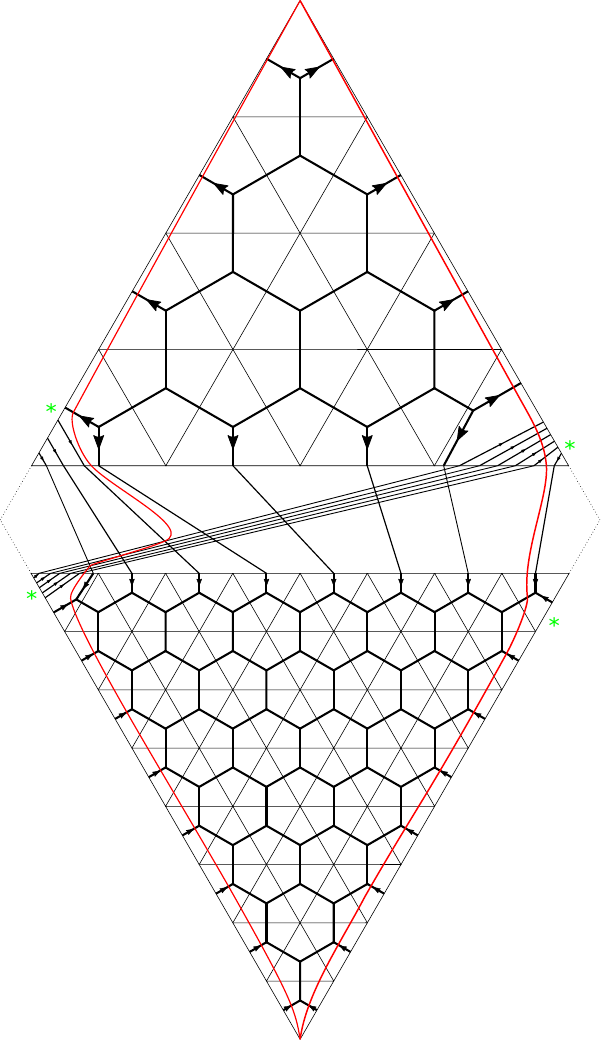}
\caption{     Family (3).  Bigon drawing procedure in the example $x=3, y=4, z=1, t=5$.  The web represented by the schematic is in good position with respect to the red bigon.  The green asterisks separate the honeycombs from the corner arcs in the flipped triangulation.  Compare Figures \ref{figure:flip3} and \ref{figure:flip3l}.  }
\label{fig:case3flip}
\end{figure}

\begin{figure}[t]
\includegraphics[scale=0.8]{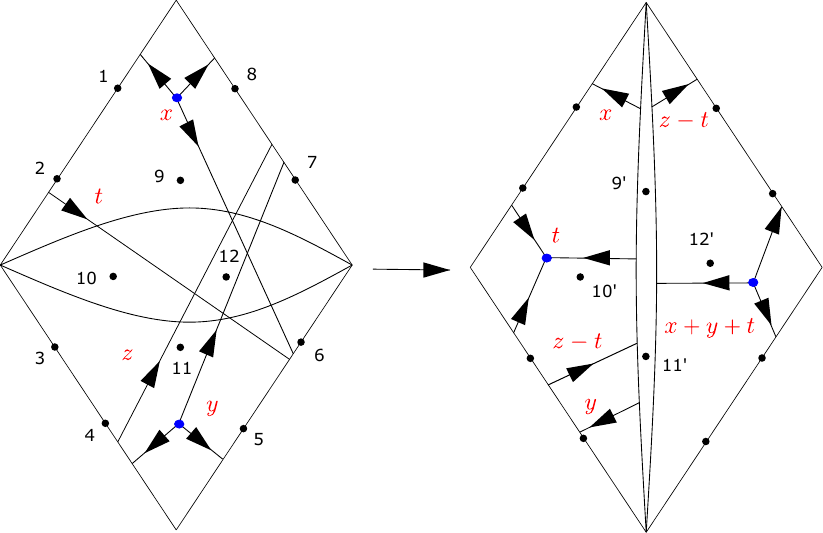}
\caption{     Family (7), shown when $z \geq t$.}
\label{figure:flip7}
\end{figure}

\begin{figure}[t]
\includegraphics[scale=0.8]{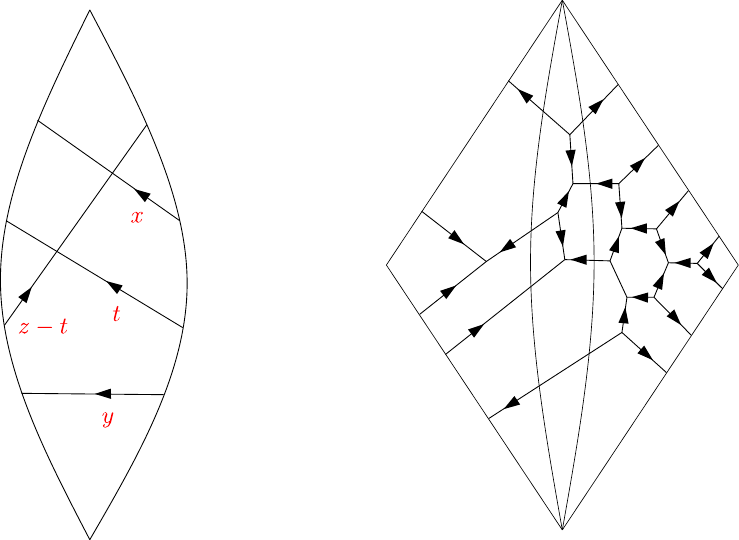}
\caption{     Family (7), $z \geq t$.  Left: Bigon schematic after the flip. Right: An example of the web in good position after the flip:  $x=y=t=1$ and $z=2$.}
\label{figure:flip7l}
\end{figure}

\begin{figure}[t]
\includegraphics[height=.8\textheight]{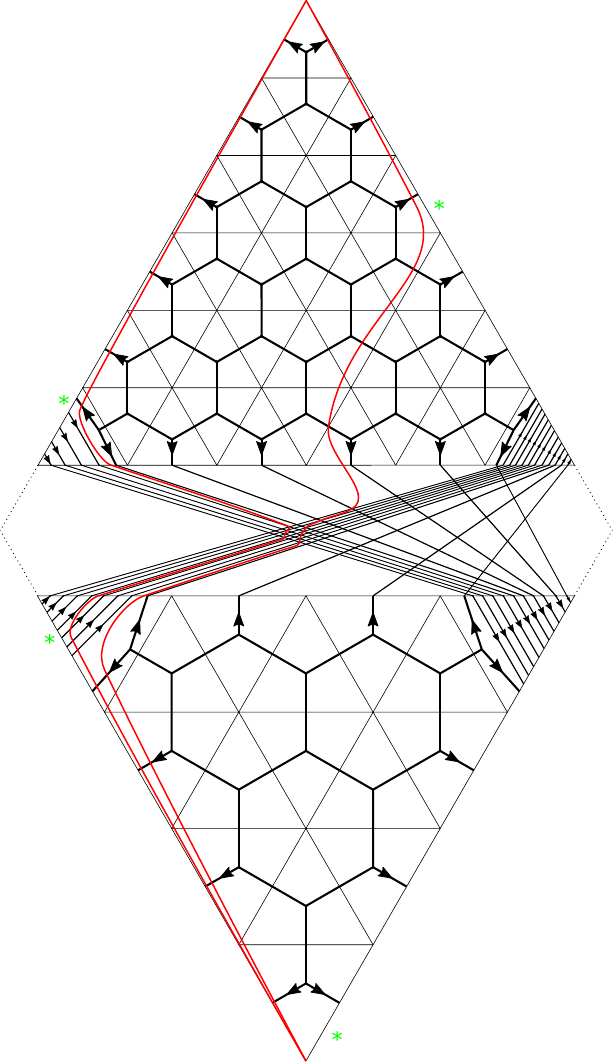}
\caption{     Family (7).  Bigon drawing procedure in the example $x=6, y=4, z=7, t=4$.  The web represented by the schematic is in good position with respect to the red bigon.  The green asterisks separate the honeycombs from the corner arcs in the flipped triangulation.  Compare Figures \ref{figure:flip7} and \ref{figure:flip7l}.  }
\label{fig:case7flip}
\end{figure}

	\section{Pentagon relation for \texorpdfstring{$\mathrm{SL}_3$}{SL3}}
	\label{app:first-appendix}
	
	The Mathematica notebook presented in Figures \ref{fig:appendixApdf1} and \ref{fig:appendixApdf2} below provides the expression for $f_5(x_1, x_2, \dots, x_{17})$ as defined in Example (part 5) from  \S \ref{def:tropical-A-coordinate-cluster-transformation}.  See Remark \ref{rem:seven-identities}.

\begin{figure}[t]
	\includegraphics[height=.95\textheight]{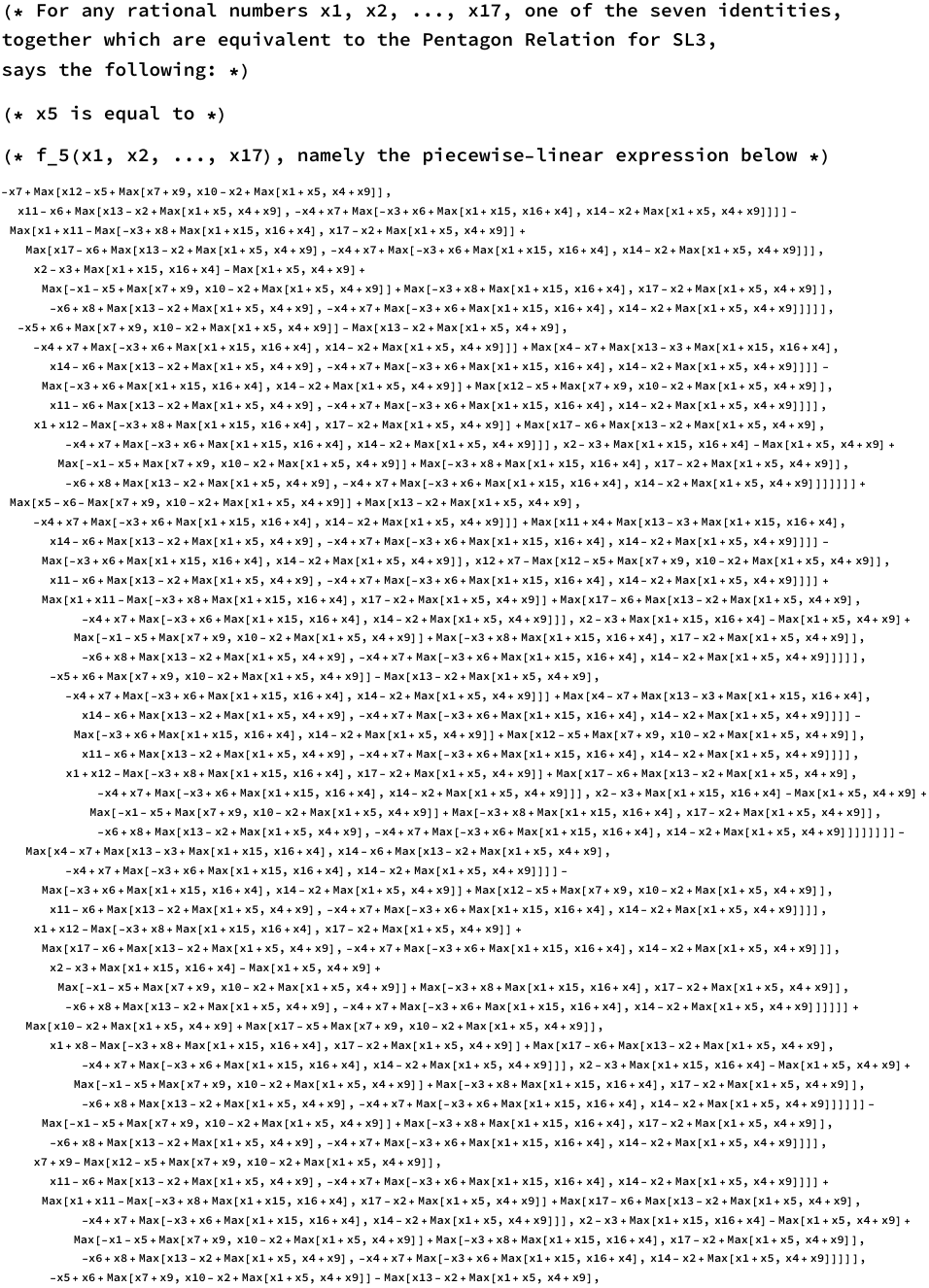}
	\caption{(Page 1 of 2; see also Figure \ref{fig:appendixApdf2}.)}
	\label{fig:appendixApdf1}
\end{figure}

\begin{figure}[t]
	\includegraphics[height=.95\textheight]{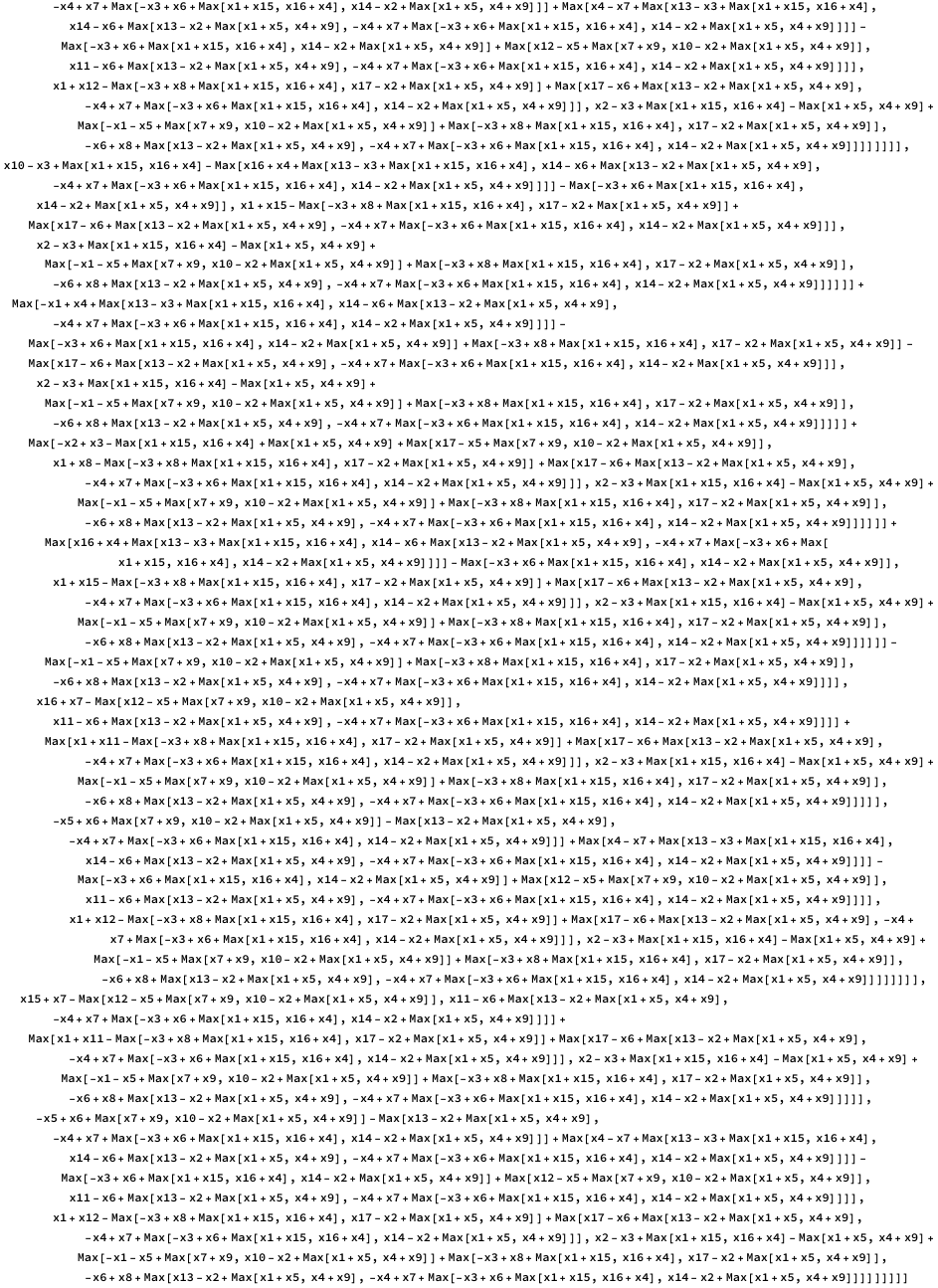}
	\caption{(Page 2 of 2; see also Figure \ref{fig:appendixApdf1}.)}
	\label{fig:appendixApdf2}
\end{figure}

\DeclareEmphSequence{\itshape}
\bibliographystyle{plain}
\bibliography{references.bib}

\begin{thebibliography}{10}

\bibitem{AbdielAlgGeomTop17}
N.~Abdiel and C.~Frohman.
\newblock The localized skein algebra is {Frobenius}.
\newblock {\em Algebr. Geom. Topol.}, 17:3341--3373, 2017.

\bibitem{akhmejanovMR4213127}
T.~Akhmejanov.
\newblock Non-elliptic webs and convex sets in the affine building.
\newblock {\em Doc. Math.}, 25:2413--2443, 2020.

\bibitem{AllegrettiAdvMath17}
D.~G.~L. Allegretti and H.~K. Kim.
\newblock A duality map for quantum cluster varieties from surfaces.
\newblock {\em Adv. Math.}, 306:1164--1208, 2017.

\bibitem{BullockCommentMathHelv97}
D.~Bullock.
\newblock Rings of ${SL}_2(\mathbb{C})$-characters and the {Kauffman} bracket
  skein module.
\newblock {\em Comment. Math. Helv.}, 72:521--542, 1997.

\bibitem{CautisMathAnn14}
S.~Cautis, J.~Kamnitzer, and S.~Morrison.
\newblock {Webs and quantum skew Howe duality}.
\newblock {\em Math. Ann.}, 360:351--390, 2014.

\bibitem{MR4790744}
D.~C. Douglas.
\newblock Points of quantum {${\rm SL}_n$} coming from quantum snakes.
\newblock {\em Algebr. Geom. Topol.}, 24:2537--2570, 2024.

\bibitem{Douglas1}
D.~C. Douglas.
\newblock Quantum traces for {${\rm SL}_n(\mathbb{C})$}: the case {$n=3$}.
\newblock {\em J. Pure Appl. Algebra}, 228, 2024.

\bibitem{DouglasArxiv20b}
D.~C. Douglas and Z.~Sun.
\newblock {Tropical {Fock-Goncharov} coordinates for $\mathrm{SL}_3$-webs on
  surfaces {II}: naturality}, to appear in \textit{{Algebr. Comb.}}
\newblock \url{https://arxiv.org/abs/2012.14202}, 2020.

\bibitem{DouglasArxiv20}
D.~C. Douglas and Z.~Sun.
\newblock Tropical {Fock-Goncharov} coordinates for $\mathrm{SL}_3$-webs on
  surfaces {I}: construction.
\newblock {\em Forum Math. Sigma}, 12, 2024.

\bibitem{FockIHES06}
V.~V. Fock and A.~B. Goncharov.
\newblock Moduli spaces of local systems and higher {Teichm\"{u}ller} theory.
\newblock {\em Publ. Math. Inst. Hautes \'Etudes Sci.}, 103:1--211, 2006.

\bibitem{Fock07}
V.~V. Fock and A.~B. Goncharov.
\newblock Dual {Teichm\"{u}ller} and lamination spaces.
\newblock {\em Handbook of {Teichm\"{u}ller} theory}, 1:647--684, 2007.

\bibitem{FockAdvMath07}
V.~V. Fock and A.~B. Goncharov.
\newblock {Moduli spaces of convex projective structures on surfaces}.
\newblock {\em Adv. Math.}, 208:249--273, 2007.

\bibitem{FockENS09}
V.~V. Fock and A.~B. Goncharov.
\newblock Cluster ensembles, quantization and the dilogarithm.
\newblock {\em Ann. Sci. \'Ec. Norm. Sup\'er.}, 42:865--930, 2009.

\bibitem{FominArxiv17}
S.~Fomin, L.~Williams, and A.~Zelevinsky.
\newblock {Introduction to cluster algebras. Chapters 4--5}.
\newblock \url{https://arxiv.org/abs/1707.07190}, 2017.

\bibitem{FominJAmerMathSoc02}
S.~Fomin and A.~Zelevinsky.
\newblock {Cluster algebras. {I.} Foundations}.
\newblock {\em J. Amer. Math. Soc.}, 15:497--529, 2002.

\bibitem{FontaineCompositio13}
B.~Fontaine, J.~Kamnitzer, and G.~Kuperberg.
\newblock Buildings, spiders, and geometric {Satake}.
\newblock {\em Compos. Math.}, 149:1871--1912, 2013.

\bibitem{FrohmanMathZ2022}
C.~Frohman and A.~S. Sikora.
\newblock {$\mathrm{SU}(3)$-skein algebras and webs on surfaces}.
\newblock {\em Math. Z.}, 300:33--56, 2022.

\bibitem{GaiottoAnnHenriPoincare13}
D.~Gaiotto, G.~W. Moore, and A.~Neitzke.
\newblock Spectral networks.
\newblock {\em Ann. Henri Poincar\'e}, 14:1643--1731, 2013.

\bibitem{GoncharovInvent15}
A.~B. Goncharov and L.~Shen.
\newblock Geometry of canonical bases and mirror symmetry.
\newblock {\em Invent. Math.}, 202:487--633, 2015.

\bibitem{GoncharovAdvMath18}
A.~B. Goncharov and L.~Shen.
\newblock Donaldson--{Thomas} transformations of moduli spaces of {G}-local
  systems.
\newblock {\em Adv. Math.}, 327:225--348, 2018.

\bibitem{GoncharovArxiv19}
A.~B. Goncharov and L.~Shen.
\newblock Quantum geometry of moduli spaces of local systems and representation
  theory.
\newblock \url{https://arxiv.org/abs/1904.10491}, 2019.

\bibitem{GrossJAmerMathSoc18}
M.~Gross, P.~Hacking, S.~Keel, and M.~Kontsevich.
\newblock Canonical bases for cluster algebras.
\newblock {\em J. Amer. Math. Soc.}, 31:497--608, 2018.

\bibitem{hilbert1890ueber}
D.~Hilbert.
\newblock Ueber die theorie der algebraischen formen.
\newblock {\em Math. Ann.}, 36:473--534, 1890.

\bibitem{HuangArxiv19}
Y.~Huang and Z.~Sun.
\newblock Mc{S}hane identities for higher {T}eichm\"uller theory and the
  {G}oncharov--{S}hen potential.
\newblock {\em Mem. Amer. Math. Soc.}, 286:v+116, 2023.

\bibitem{Ishibashi22}
T.~Ishibashi and S.~Kano.
\newblock {Unbounded $\mathfrak{sl}_3$-laminations and their shear
  coordinates}.
\newblock \url{https://arxiv.org/abs/2204.08947}, 2022.

\bibitem{jordan2021quantumdecoratedcharacterstacks}
D.~Jordan, I.~Le, G.~Schrader, and A.~Shapiro.
\newblock Quantum decorated character stacks.
\newblock \url{https://arxiv.org/abs/2102.12283}, 2021.

\bibitem{KimArxiv20}
H.~K. Kim.
\newblock {$\mathrm{SL}_3$-laminations as bases for $\mathrm{PGL}_3$ cluster
  varieties for surfaces}, to appear in \textit{{Mem. Amer. Math. Soc.}}
\newblock \url{https://arxiv.org/abs/2011.14765}, 2020.

\bibitem{KimArxiv21}
H.~K. Kim.
\newblock Naturality of $\mathrm{SL}_3$ quantum trace maps for surfaces, to
  appear in \textit{{Quantum Topol.}}
\newblock \url{https://arxiv.org/abs/2104.06286}, 2021.

\bibitem{KnutsonJAmerMathsoc99}
A.~Knutson and T.~Tao.
\newblock The honeycomb model of $\mathrm{GL}_n(\mathbb{C})$ tensor products.
  {I.} {Proof} of the saturation conjecture.
\newblock {\em J. Amer. Math. Soc.}, 12:1055--1090, 1999.

\bibitem{kontsevich2008stability}
M.~Kontsevich and Y.~Soibelman.
\newblock Stability structures, motivic {Donaldson-Thomas} invariants and
  cluster transformations.
\newblock \url{https://arxiv.org/abs/0811.2435}, 2008.

\bibitem{KuperbergCommMathPhys96}
G.~Kuperberg.
\newblock Spiders for rank 2 {Lie} algebras.
\newblock {\em Comm. Math. Phys.}, 180:109--151, 1996.

\bibitem{LeGeomTop16}
I.~Le.
\newblock Higher laminations and affine buildings.
\newblock {\em Geom. Topol.}, 20:1673--1735, 2016.

\bibitem{le2019imrn}
I.~Le.
\newblock An approach to higher {T}eichm\"uller spaces for general groups.
\newblock {\em Int. Math. Res. Not. IMRN}, pages 4899--4949, 2019.

\bibitem{le2019cluster}
I.~Le.
\newblock {Cluster structures on higher Teichm\"uller spaces for classical
  groups}.
\newblock {\em Forum Math. Sigma}, 7, 2019.

\bibitem{MR4339354}
I.~Le.
\newblock Intersection pairings for higher laminations.
\newblock {\em Algebr. Comb.}, 4:823--841, 2021.

\bibitem{maclagan2021introduction}
D.~Maclagan and B.~Sturmfels.
\newblock {\em Introduction to tropical geometry}, volume 161 of {\em Graduate
  Studies in Mathematics}.
\newblock American Mathematical Society, Providence, RI, 2015.

\bibitem{MullerQuantumTopology16}
G.~Muller.
\newblock Skein and cluster algebras of marked surfaces.
\newblock {\em Quantum Topol.}, 7:435--503, 2016.

\bibitem{Mumford94}
D.~Mumford, J.~Fogarty, and F.~Kirwan.
\newblock {\em Geometric invariant theory. {T}hird edition}.
\newblock Springer-Verlag, Berlin, 1994.

\bibitem{penner1987decorated}
R.~C. Penner.
\newblock The decorated {Teichm{\"u}ller} space of punctured surfaces.
\newblock {\em Comm. Math. Phys.}, 113:299--339, 1987.

\bibitem{PrzytyckiBullPolishAcad91}
J.~H. Przytycki.
\newblock Skein modules of 3-manifolds.
\newblock {\em Bull. Polish Acad. Sci. Math.}, 39:91--100, 1991.

\bibitem{PrzytyckiTopology00}
J.~H. Przytycki and A.~S. Sikora.
\newblock On skein algebras and ${S}l_2(\mathbb{C})$-character varieties.
\newblock {\em Topology}, 39:115--148, 2000.

\bibitem{schrijver1981total}
A.~Schrijver.
\newblock On total dual integrality.
\newblock {\em Linear Algebra Appl.}, 38:27--32, 1981.

\bibitem{Shen1}
L.~Shen, Z.~Sun, and D.~Weng.
\newblock Intersections of dual ${SL}_3$-webs.
\newblock \url{https://arxiv.org/abs/2311.15466}, 2023.

\bibitem{SikoraTrans01}
A.~S. Sikora.
\newblock $\mathrm{SL}_n$-character varieties as spaces of graphs.
\newblock {\em Trans. Amer. Math. Soc.}, 353:2773--2804, 2001.

\bibitem{SikoraAlgGeomTop05}
A.~S. Sikora.
\newblock Skein theory for ${SU}(n)$-quantum invariants.
\newblock {\em Algebr. Geom. Topol.}, 5:865--897, 2005.

\bibitem{SikoraAlgGeomTop07}
A.~S. Sikora and B.~W. Westbury.
\newblock Confluence theory for graphs.
\newblock {\em Algebr. Geom. Topol.}, 7:439--478, 2007.

\bibitem{SunGeomFunctAnal20}
Z.~Sun, A.~Wienhard, and T.~Zhang.
\newblock {Flows on the $\mathrm{PGL(V)}$-Hitchin component}.
\newblock {\em Geom. Funct. Anal.}, 30:588--692, 2020.

\bibitem{Thurston97}
W.~P. Thurston.
\newblock {\em {Three-dimensional geometry and topology. Vol. 1}}.
\newblock Princeton University Press, Princeton, NJ, 1997.

\bibitem{Turaev89}
V.~G. Turaev.
\newblock Algebras of loops on surfaces, algebras of knots, and quantization.
\newblock In {\em Braid group, knot theory and statistical mechanics}, pages
  59--95. World Sci. Publ., Teaneck, NJ, 1989.

\bibitem{wienhard2018invitation}
A.~Wienhard.
\newblock An invitation to higher {T}eichm\"{u}ller theory.
\newblock In {\em Proceedings of the {I}nternational {C}ongress of
  {M}athematicians---{R}io de {J}aneiro 2018. {V}ol. {II}. {I}nvited lectures},
  pages 1013--1039. World Sci. Publ., Hackensack, NJ, 2018.

\bibitem{WittenCommMathPhys89}
E.~Witten.
\newblock Quantum field theory and the {Jones} polynomial.
\newblock {\em Comm. Math. Phys.}, 121:351--399, 1989.

\bibitem{XieArxiv13}
D.~Xie.
\newblock {Higher laminations, webs and N=2 line operators}.
\newblock \url{https://arxiv.org/abs/1304.2390}, 2013.

\end{thebibliography}

\end{document}